\documentclass[a4paper, twoside, 11pt]{article}
\usepackage{amsmath,amssymb,amsthm,braket,nccmath, mathrsfs}
\usepackage[margin=25mm]{geometry}
\usepackage{indentfirst}
\usepackage{fancyhdr}
\usepackage{comment}
\usepackage{cite}

\numberwithin{equation}{section}

\newcommand{\C}{\mathbb{C}}
\newcommand{\N}{\mathbb{N}}
\newcommand{\R}{\mathbb{R}}
\newcommand{\T}{\mathbb{T}}
\newcommand{\Z}{\mathbb{Z}}
\newcommand{\ep}{\varepsilon}
\newcommand{\vphi}{\varphi}

\newcommand{\f}{\frac}
\newcommand{\lmd}{\lambda}

\newcommand{\cA}{\mathcal{A}}

\newcommand{\cF}{\mathcal{F}}
\newcommand{\cH}{\mathcal{H}}

\newcommand{\cK}{\mathcal{K}}

\newcommand{\cR}{\mathcal{R}}

\newcommand{\Fav}{F_{\rm{av}}}
\newcommand{\Gav}{G_{\rm{av}}}

\newcommand{\h}{\f{1}{2}}
\newcommand{\al}{\alpha}

\newcommand{\s}{\sigma}
\newcommand{\om}{\omega}
\newcommand{\e}{\eta}

\newcommand{\z}{\zeta}
\newcommand{\tB}{\tilde{B}}

\theoremstyle{theorem}

\newtheorem*{theorem*}{Theorem}
\newtheorem{thm}{Theorem}[section]

\newtheorem*{definition*}{Definition}
\newtheorem{lemma}[thm]{Lemma}
\newtheorem*{lemma*}{Lemma}
\newtheorem{prop}[thm]{Proposition}
\newtheorem{remark}[thm]{Remark}
\newtheorem{cor}[thm]{Corollary}

\title{Scattering and blow up for nonlinear Schr\"{o}dinger equation with the averaged nonlinearity}
\author{\small{JUMPEI KAWAKAMI}\thanks{Department of Mathematical Sciences, Division of Mathematics and Mathematical Sciences in the Graduate School of Science, Kyoto University, Kyoto 606-8502 JAPAN. 
		E-mail: jumpeik@kurims.kyoto-u.ac.jp 
		ORCID:0000-0002-9525-2782.}}{}
\date{}

\begin{document}
\markboth{\centerline{\footnotesize{JUMPEI KAWAKAMI}}}{\centerline{\footnotesize{Scattering and blow up for NLS with the averaged nonlinearity}}}
	
\maketitle
\renewcommand{\thefootnote}{\fnsymbol{footnote}}
\footnote[0]{
	\noindent
	2020 \textit{Mathematics Subject Classification.} 35Q55, 35P25.
	
	\textit{Key words and phrases.} Nonlinear Schr\"{o}dinger equation, strong magnetic confinement, harmonic oscillator, resonant system, high frequency averaging, super-quintic nonlinearity.}
\renewcommand{\thefootnote}{\arabic{footnote}}
\begin{abstract}
	We consider the 3-dimensional nonlinear Schr\"{o}dinger equation (NLS) with average nonlinearity. This is a limiting model of NLS with strong magnetic confinement and a generalized model of the resonant system of NLS with a partial harmonic oscillator in terms of nonlinear power.  We provide a new proof for the conservation law of kinetic energy and remove the restriction on nonlinearity. Moreover, in the case of focusing, super-quintic, and sub-nonic, we construct a new ground-state solution and classify the behavior of the solutions below the ground state. We demonstrate a sharp threshold for scattering and blow-up. 
\end{abstract}

\section{Introduction}
\subsection{The derivation of the model}
We consider the following three-dimensional nonlinear Schrödinger equation(NLS):
\begin{equation}\label{NLS}
	i\partial_t\phi(t)= -\partial_z^2\phi(t) +\lmd \Fav(\phi(t)), \qquad x=(y,z) \in \R^2\times \R.
\end{equation}
$\Fav(\phi)$ is defined as
\begin{equation}\label{Fav}
	\Fav(\phi):=\f{2}{\pi}\int_0^{\f{\pi}{2}} e^{i\theta H}(|e^{-i\theta H}\phi|^{2\sigma}e^{-i\theta H}\phi)d\theta,
\end{equation}
where $\sigma \in \R_+$ and $H$ represents the harmonic oscillator in the $y$-direction:
\begin{equation}
	H:= -\Delta_y + |y|^2.
\end{equation}
Equation \eqref{NLS} was derived in the work of Frank-Méhats-Sparber\cite{Fr}, who studied the following model:
\begin{equation}\label{NLSep}
\begin{split}
i\partial_t \psi^\ep&=\f{\cH}{\ep^2}\psi^\ep -\f{\partial_z^2}{2} \psi^\ep +V(z)\psi^\ep + \lmd|\psi^\ep|^{2\sigma}\psi^\ep, \quad x=(y,z) \in \R^2\times \R,\\
&\cH=-\f{\Delta_y}{2}+\f{|y|^2}{8}-\f{i}{2}(y_1\partial_{y_2}-y_2\partial_{y_1}),\quad 0<\ep\ll 1,\\
 &\f{d^\al}{dz^\al}V(z) \in L^\infty(\R) \qquad \text{for any}\quad  \al\ge2.
\end{split}
\end{equation}
This model was also derived in \cite{Fr} by scaling NLS equation with a strong magnetic field, which describes the behavior of rotating Bose-Einstein condensates. To analyze the strong confinement limit $\ep\to +0$, they considered the profile $\phi^\ep(t):=e^{it\cH/\ep^2}\psi^\ep(t)$. This profile solves
\begin{equation*}
	\begin{split}
		i\partial_t\phi^\ep(t)&= -\f{\partial_z^2}{2}\phi^\ep(t) + V(z)\phi^\ep(t)+\lmd e^{it\f{\cH}{\ep^2}}(|e^{-it\f{\cH}{\ep^2}}\phi^\ep(t)|^{2\sigma} e^{-it\f{\cH}{\ep^2}}\phi^\ep(t)).
	\end{split}
\end{equation*}
Regarding nonlinearity, they introduced the nonlinear function
\begin{equation*}
	F(\theta, u):= e^{i\theta\cH}(|e^{-i\theta\cH}u|^{2\sigma} e^{-i\theta\cH}u).
\end{equation*}
Because $\cH$ is essentially self-adjoint on $C_0^\infty(\R^2)$ with pure point spectrum $\{n+\h:n\in \N_0\}$, $F(\cdot, u)$ is $2\pi$-periodic. Then, to study the behavior of $F(\f{t}{\ep^2}, u)$ as $\ep\to+0$, they defined the average of $F(\cdot, u)$ as
\begin{equation*}
\begin{split}
	\Fav(u):=\lim_{T\to+\infty}\f{1}{T}\int_{0}^{T}F(\theta, u)d\theta=\f{1}{2\pi}\int_{0}^{2\pi} F(\theta, u)d\theta
\end{split}
\end{equation*}
and formally derived the limiting model of $\phi^\ep$:
\begin{equation}\label{NLSav}
	i\partial_t\phi(t)= -\f{\partial_z^2}{2}\phi(t) + V(z)\phi(t)+\lmd \Fav(\phi(t)), \qquad x=(y,z) \in \R^2\times \R.
\end{equation}
In this study, we assume that $V(z)\equiv 0$.
Moreover, for simplicity, by normalizing the coefficients and eliminating the rotation in $y$, we obtain \eqref{NLS}. This type of averaging method has often been used to derive the limiting model of NLS equations with strong confinement, cf.~\cite{N.Ben, N.Ben2, N.Ben3, M, Iwa, ADM}.

On the other hand, when $\s \in \N$, \eqref{NLS} is also derived as the resonant system of NLS with a partial harmonic oscillator:
\begin{equation}\label{NLSH0}
	i\partial_t u=-\Delta_x u +|y|^2 u+ \lmd|u|^{2\sigma}u, \qquad x=(y,z) \in \R^2\times \R.
\end{equation}
This model arises, for example, in Chen \cite{Ord}.
Let us examine the meaning of "resonance''. For simplicity, we assume that $\s=1$. Then the profile $w(t)=e^{itH}u(t)$ solves
\begin{equation}\label{prof}
	\begin{split}
		i\partial_tw(t)&= -\partial_z^2 w(t)+ \lmd e^{itH}(|e^{-itH}w(t)|^2e^{-itH}w(t)) \\
		&=  \lmd\sum_{n_1, n_2, n_3, n \in \N_0} e^{-2it(n_1+n_2-n_3-n)}\Pi_{n} (\Pi_{n_1}w(t) \Pi_{n_2}w(t) \overline{ \Pi_{n_3}w(t)}).
	\end{split}
\end{equation}
In the second line, we use the Hermite expansion, where $\Pi_n: L^2(\R^2) \to L^2(\R^2)$ ($n \in \N_0$) is the projection onto the eigenspace of $H$ corresponding to the eigenvalue $2(n+1)$. The sum of the terms satisfying $n_1+n_2-n_3-n=0$, whose oscillations disappear, is the resonant component. Therefore, by extracting the resonant part in \eqref{prof}, we obtain \eqref{NLS}. Note that it holds the following identity, cf.~Germain-Hani-Thomann\cite{CR}:
\begin{equation*}
	\sum_{n_1+n_2=n_3+n} \Pi_{n} (\Pi_{n_1}\phi \Pi_{n_2}\phi \overline{ \Pi_{n_3}\phi})	= \f{2}{\pi} \int_{-\f{\pi}{4}}^{\f{\pi}{4}} \Big[ e^{i\theta H}\big( (e^{-i\theta H}\phi)(e^{-i\theta H}\phi)\overline{(e^{-i \theta H}\phi)\big)} \Big] d\theta.
\end{equation*}
If $\sigma$ is a positive integer, this calculation can be  replicated. See Fennell \cite{RH}. However, to extend the resonant nonlinearity to the case $\s>0$, we use the representation  \eqref{Fav}.\\

We return to \eqref{NLS}. Equation \eqref{NLS} contains the following conserved quantities.

\begin{itemize}
	\item Mass:
	\begin{equation}\label{mass}
		M[\phi]:=\h\int_{\R^3} |\phi|^2 dx
	\end{equation}
	\item Hamiltonian:
	\begin{equation}\label{hamil}
		E[\phi]:=\h \int_{\R^3} |\partial_z \phi|^2 dx + \f{\lmd}{\pi(\sigma+1)}\int_{0}^{\f{\pi}{2}} \int_{\R^3} |e^{-i\theta H}\phi|^{2\sigma+2} dx d\theta
	\end{equation}
	\item Momentum:
	\begin{equation}\label{momz}
		G[\phi]:=\text{Im}\int_{\R^3}\overline{\phi} \partial_z \phi dx
	\end{equation}
	\item Kinetic energy:
	\begin{equation}\label{kin}
		K[\phi]:=\h \int_{\R^3} \Big( |\nabla_y\phi|^2 + |y\phi|^2 \Big)dx.
	\end{equation}
	
\end{itemize}
Although the conservation laws for $M$, $E$, and $G$ are classical, the conservation law of $K$ has been proven and used only when $\sigma \in \N$, cf.~\cite{RH, CR, MS, K}. However, by employing a new calculation for the representation \eqref{Fav}, we can extend this conservation law to all $\sigma >0$.

\subsection{NLS with a partial harmonic oscillator}

From a mathematical perspective, equation \eqref{NLSH0} can be generalized as 
\begin{equation}\label{NLSH}
	i\partial_t u=-\Delta_x u +|y|^2 u+ \lmd|u|^{2\sigma}u, \qquad x=(y,z) \in \R^n\times \R^{d-n}, \quad 1\le d, \quad 0\le n\le d.
\end{equation}
If $n=0$, it is well-known that the linear term induces a large time dispersive effect, which leads to asymptotically linear behavior, that is, scattering.
On the other hand, when $n=d$, the linear term corresponds to the isotropic harmonic oscillator and \eqref{NLSH} no longer has a large time dispersive effect. This case was studied in \cite{RNLSHP, ECQHO, ECNLSQP}. However, because of the absence of the large time dispersion, the global behaviors are much less known than in the case $n<d$. \\

Our study is relevant to the case $1\le n<d$. In this case, a partial harmonic oscillator with anisotropic potential appears in the linear term. Because the $d-n$-dimensional dispersive effect arises in $z$, we expect the large solution to \eqref{NLSH} to scatter as long as the nonlinearity is at least mass-critical for the spatial dimension $d-n$ and at most energy-critical for the spatial dimension $d$:
\begin{equation}\label{as00}
	\f{2}{d-n}\le \sigma \le \f{2}{d-2}.
\end{equation}

When $n=1$, the condition \eqref{as00} is not empty, even if we exclude both endpoints.
In the defocusing case, Antonelli-Carles-Silva \cite{Scpht} proved the large data scattering under the condition \eqref{as00} except for both endpoints in 
\[	\Sigma^1:=\{ \vphi \in H^1(\R^d): |x|\vphi \in L^2(\R^d)\}.\]
In the focusing case, Ardila-Carles \cite{Gdb} studied \eqref{NLSH} under the conditions $\s\ge \h $ and \eqref{as00} except for both endpoints in the energy space 
\[	B^1:=\{ \vphi \in H^1(\R^d): |y|\vphi \in L^2(\R^d)\}. \]
They classified the behavior (scattering or blow-up) of the solution below the ground state.

Cheng-Guo-Guo-Liao-Shen \cite{Sc3DNLSphp} proved the large data scattering of \eqref{NLSH} in $B^1$
under the condition $(d,n,\s)=(3,1,1)$, which is one of the lower endpoints, $\s=2/(d-n)$. They approximated \eqref{NLSH} using the corresponding resonant system, which has a structure similar to \eqref{NLS}, and proved the global well-posedness and scattering of the resonant system in the space
\[ L_z^2\Sigma_y^1:=\{ \vphi \in L^2(\R^d): \nabla_y\vphi, |y|\vphi \in L^2(\R^d)\}.\] 
The analysis of the resonant system was inspired by Dodson \cite{mc3DNLS, mc1DNLS, mc2DNLS}. In \cite{Gdb,Sc3DNLSphp}, the concentration compactness and rigidity argument developed by Kenig-Merle \cite{KM} was used.\\

When $n\ge 2$, the global dispersive effect becomes weaker than in the case $n=1$ and the left side of \eqref{as00} is equal to or greater than the right side. Therefore, to clarify the global dynamics, such as scattering, under the energy-subcritical condition $\s<2/(d-2)$, we have to consider in $\Sigma^1$ or narrower function spaces. Antonelli-Carles-Silva \cite{Scpht} proved the small data scattering and the existence of wave operators, under the condition $n\ge1$ and additional conditions for $\s$. Hani-Thomann \cite{MS} studied \eqref{NLSH} in the case $\sigma=1$ and $d-n=1$. They proved modified scattering and constructed modified wave operators for small initial and final data, respectively. In \cite{MS} they used the resonant system to analyze \eqref{NLSH}.\\

 The same type of model as \eqref{NLS} is also used as the limiting model of NLS with strong magnetic confinement, cf.~\cite{N.Ben, N.Ben2, N.Ben3, Fr, K2, K}.  The dynamics of \eqref{NLS} is closely related to many open problems concerning the dynamics of \eqref{NLSH}. \\ 

Equation \eqref{NLS} has not only been used to analyze \eqref{NLSH}, but has also been studied. The author \cite{K} showed that when $1\le\s\le 4$, \eqref{NLS} is well-posed in $\Sigma^1$ and the case $\s=4$ is energy-critical (see also Proposition \ref{wp} in this paper).
Because \eqref{NLSH0} is energy-critical when $\s=2$, we expect \eqref{NLS} will be useful in clarifying the dynamics of \eqref{NLSH0} and \eqref{NLSep} with energy-supercritical nonlinearity ($2<\sigma \le4$). \\

When $(d,n,\s)=(3,2,2)$, equation \eqref{NLSH} is both mass and energy critical (see \eqref{as00}). Moreover, the spatial dimension $3$ is of physical significance. 
Therefore, this case is difficult yet interesting. To the best of our knowledge, the global dynamics of \eqref{NLSH0} with $(n,\s)=(2,2)$ is an open problem, and our scattering result for \eqref{NLS} does not cover the case $\s=2$. This will be the subject of our future studies.\\

\subsection{NLS on the wave guide manifolds}\label{NLSonWGM}
For topics related to \eqref{NLS}, we should refer to NLS on the waveguide manifolds $\T^n\times \R^{d-n}$ for $1\le n <d$:
\begin{equation}\label{NLSM}
	i\partial_t u=-\Delta_{x} u +\lmd |u|^{2\s}u, \qquad x=(y,z)\in \T^{n}\times \R^{d-n}.
\end{equation}
The dynamics of \eqref{NLSM} have been extensively studied, cf. \cite{ScNLSR2T, wpecNLSM, ScNLSRT2, MST, StppNLS4d, GwpNLSRT3, LTDfNLSRdT, NLSRT, WpscNLSMRdT, GwpcNLSM, ScNLSR2T2, ScNLSRmT, LTDNLSM}. 
Hani-Pausader \cite{ScNLSRT2} studied the defocusing and quinic NLS on $\R\times \T^2$, that is, $(d,n,\s)=(3,2,2)$. They introduced a large-scale resonant system on $\R\times \Z^2$.
\begin{equation}\label{resM}
	\begin{split}
		&i\partial_t u_j+\partial_z^2 u_j =  \sum_{(j_1, j_2, j_3, j_4, j_5, j)\in \mathscr{R}(j)}u_{j_1}\overline{u_{j_2}} u_{j_3} \overline{u_{j_4}} u_{j_5}, \qquad \vec{u}=\{u_j\}_{j\in \Z^2},\\
		&\mathscr{R}(j)=\{(j_1, j_2, j_3, j_4, j_5)\in (\Z^2)^5: j_1-j_2+j_3-j_4+j_5=j \quad\text{and} \\
		& \qquad \qquad \qquad \qquad \qquad \qquad \qquad  |j_1|^2-|j_2|^2+|j_3|^2-|j_4|^2+|j_5|^2=|j|^2\}.
	\end{split}
\end{equation}
This equation has a structure similar to \eqref{NLS} with $\s=2$. This type of model is often referred to as a  ``vector-valued resonant Schr\"{o}dinger system." \cite{ScNLSRT2} asserted the large data scattering for \eqref{NLSM} based on a conjecture concerning the large data scattering result of \eqref{resM}. This conjecture was later confirmed by Cheng-Guo-Zhao \cite{ScNLSRT}. Similarly, Yang-Zhao \cite{Scmres} demonstrated the global well-posedness and scattering of the defocusing, cubic resonant Schr\"{o}dinger system on $\R^2\times \Z$, and using this result, Chen-Guo-Yang-Zhao \cite{ScNLSR2T} established the global well-posedness and scattering of the defocusing cubic NLS \eqref{NLSM} on  $\R^2\times \T$. In addition, the corresponding results for the cubic NLS on $\R^2\times \T^2$ were obtained by Yang-Zhao \cite{Sc2Dcres} and Zhao \cite{ScNLSR2T2}. In the case $(d,n,\s)=(3,1,2), (4,1,1)$ with $\lmd=+1$, Zhao \cite{ScNLSRmT} obtained a scattering result.

In the focusing case, Chen-Guo-Hwang-Yoon \cite{Sc2Dfres} proved the global well-posedness and scattering of a large-scale solution to \eqref{NLSM} on $\R^2\times\T$ by establishing the global well-posedness and scattering of the resonant Schr\"{o}dinger system below the ground state.\\

In \eqref{NLS}, with $\s=2$, the global existence in the defocusing case was proven in \cite{K}. In the focusing case, we prove the global existence below the ground state. However, as mentioned previously, our scattering result does not include the case $\s=2$.

\subsection{Dispersion managed NLS}\label{DMNLS}

Our study is also related to NLS with a periodically varying dispersion coefficient:
\begin{equation}\label{DM}
	i\partial_t u+d(t)\partial_x^2u+|u|^{2\s} u=0, \quad x\in \R.
\end{equation}
Where $d(t)$ is a periodic function. This model arises in the context of nonlinear fiber optics. The averaging technique has been employed in the strong dispersion management regime:
\begin{equation}
	d(t)=d_\text{av}+\f{1}{\ep}d_0\big(\f{t}{\ep}\big).
\end{equation}
$d_0$ is a periodic function of mean zero and $d_{\text{av}}\in \R$ is the average of $d(t)$ over one period. The small parameter $0<\epsilon\ll1$ affects the period and the amplitude of the dispersion $d(t)$. To analyze \eqref{DM}, Gabitov-Turitsyn  \cite{GabTur1, GabTur2} introduced its averaged model. In particular, when $d(t)$ satisfies
\begin{equation*}
	d(t)=\left\{
	\begin{array}{ll}
		1 & t\in [0,1) \\
		-1 & t\in [1,2)
	\end{array}
	\right.
\end{equation*}
over one period, the averaged model is given by
\begin{equation}\label{DMav}
	i\partial_t v+d_\text{av}\partial_x^2v+\int_{0}^{1} e^{-i\tau\partial_x^2}(|e^{i\tau\partial_x^2}v|^{2\s}e^{i\tau\partial_x^2} v)d\tau=0.
\end{equation}
Note that the case $d_{\text{av}}>0$ corresponds to focusing, while the case $d_{\text{av}}<0$ is defocusing. The case $d_{\text{av}}=0$ is a singular limit. The Cauchy problem for equation (1.5) is locally well-posed in $H^1(\R)$ when  $d_{\text{av}} \neq 0$ and $0< \s$, or $d_{\text{av}} = 0$ and $0<\s\le2$. Moreover, the $H^1$-solution exists globally when $0<\s$ for $d_{\text{av}}< 0$; $0<\s<4$ for $d_{\text{av}}>0$; $0<\s\le 2$ for $d_{\text{av}} = 0$. See\cite{WPsnlNLS, WPADM}. The case $\s=4$ is mass-critical. Following these results, Choi-Hong-Lee \cite{GbDMNLS} identified the global versus blow-up criteria for the $H^1$-solution to \eqref{DMav} in the case $d_{\text{av}}>0$ and $\s>4$. In this result, the asymptotic behavior of the global solution is unknown. We hope that our main results will provide some insight into this open problem. On the other hand, the averaging process (the strong limit $\ep\to+0$) was verified in \cite{DMNLDla, SeDM, ADM}.

\subsection{Our aim}
This study aims to clarify the dynamics of \eqref{NLS}. From  \cite[Theorem 1.2]{K}, we obtain the well-posedness in $\Sigma^1$ and can easily replace $\Sigma^1$ with $B^1$ when $V(z)\equiv 0$ in \cite{K}. Moreover, by combining the energy estimates in \cite[Section 4]{K} and  Lemmas \ref{conke} and \ref{bdnl} in this paper, we obtain the following proposition.

\begin{prop}\label{wp}
	Let $\h\le\sigma\le4$. Then, \eqref{NLS} is locally well-posed in $B^1$.
	\begin{enumerate}
		\item When $\sigma<2$, all solutions exist globally.
		\item When $2\le\sigma<4$ and $\lmd=+1$, all solutions exist globally.
	\end{enumerate}
 Moreover, when $\sigma=4$, \eqref{NLS} is energy-critical and left invariant by the scaling
\begin{equation}\label{si}
	\phi(t,y,z) \longmapsto \phi_\mu(t,y,z):=\mu^{\f{1}{4}}\phi(\mu^2t,y,\mu z), \qquad \forall \mu>0. 
\end{equation}
This scaling leaves neither $ K[\phi_\mu(t)]$ nor $E[\phi_\mu(t)]$ invariant, but conserves
\[ K[\phi_\mu(t)]^3E[\phi_\mu(t)],\]
where $E$ and $K$ are defined by \eqref{hamil} and \eqref{kin}, respectively. 
\end{prop}

When $2<\s$, equations \eqref{NLSep} and \eqref{NLSH0} become energy-supercritical and the global existence and dynamics for large data are unknown. We expect \eqref{NLS} will be useful in clarifying the dynamics of  \eqref{NLSep}  and \eqref{NLSH0} with energy-supercritical nonlinearity ($2<\sigma \le4$). However, the global behavior of \eqref{NLS} for large data is unknown in Proposition \ref{wp}. Therefore, in the case $2\le \s<4$ we consider this problem. 

\subsection{Strategy}
In the focusing case, we consider the existence of a ground-state solution and the dynamics of the solution to \eqref{NLS} below the ground state. We introduce a standing wave in \eqref{NLS}. First, we perform the following transformation
\begin{equation*}
	\psi(t):=e^{-itH}\phi(t),
\end{equation*}
where $\phi(t)$ represents the solution to \eqref{NLS}. Then, $\psi(t)$ solves the following equation (see Remark \ref{equivNLS2}):
\begin{equation}\label{NLS2}
	i\partial_t\psi(t)= H\psi(t)-\partial_z^2\psi(t) +\lmd \Fav(\psi(t)).
\end{equation}
 In the focusing case, setting $\psi(t,x)=e^{it}\vphi(x)$ (which is equivalent to $\phi(t,x)= e^{it}e^{itH}\vphi(x)$), an elliptic problem 
\begin{equation}\label{ep2}
	H\vphi-\partial_z^2\varphi + \varphi-\Fav(\varphi) =0
\end{equation}
arises. Then we introduce the action
\begin{equation*}
\begin{split}
		S[\vphi]:=& K[\vphi]+ M[\vphi] + E[\vphi]\\
	=&\h \int_{\R^3}\big (|\nabla_x\vphi|^2 + |y\vphi|^2+ |\vphi|^2\big) dx - \f{1}{\pi(\s+1)}\int_{0}^{\f{\pi}{2}} \int_{\R^3} |e^{-i\theta H}\vphi|^{2\sigma+2} dxd\theta, 
\end{split}
\end{equation*}
and refer to $Q$ as the ground-state solution if it satisfies
\begin{equation*}
	S[Q]=\inf\{S[\vphi]: \vphi\neq0 \text{ and } \vphi \text{ solves } \eqref{ep2} \}.
\end{equation*}
To construct the ground-state solution, we consider the variational problem
\begin{equation}\label{v2}
	\inf\{S[\vphi]: \vphi \in B^1 \backslash \{0\}\text{ and } I[\vphi]=0\}, 
\end{equation}
where
\begin{align*}
	I[\vphi]&:=  \int_{\R^3}\big (|\nabla_x\vphi|^2 + |y\vphi|^2+ |\vphi|^2\big) dx - \f{2}{\pi}\int_{0}^{\f{\pi}{2}} \int_{\R^3} |e^{-i\theta H}\vphi|^{2\sigma+2} dxd\theta.
\end{align*}
Then, we can show that \eqref{v2} is finite and identify a minimizer. Thus, the ground-state solution to \eqref{NLS2} exists. In fact, the minimizer $Q$ is an extremizer of the Strichartz estimate
\begin{equation}\label{SSQ}
 \Big(\int_{0}^{\f{\pi}{2}} \int_{\R^3} |e^{-i\theta H}\vphi|^{2\sigma+2} dxd\theta\Big)^{\f{1}{2\s+2}}\lesssim\Big( \int_{\R^3}\big (|\nabla_x\vphi|^2 + |y\vphi|^2+ |\vphi|^2\big) dx \Big)^{\f{1}{2}}.
\end{equation}

\begin{remark}
If we set $\phi(t)=e^{i t}\vphi(x)$, $\vphi$ solves 
\begin{equation}
	-\partial_z^2\varphi + \varphi-\Fav(\varphi) =0.
\end{equation}
To prove the existence of the ground-state solution, we need to consider the following variational problem: 
\begin{equation}\label{v}
	\inf\{E[\vphi]+ M[\vphi]: \vphi \in B^1 \backslash \{0\}\text{ and } I_z[\vphi]=0\}, 
\end{equation}
where
\begin{align*}
	I_z[\vphi]&:= \int_{\R^3} \big( |\partial_z\vphi|^2+ |\vphi|^2 \big)dx - \f{2}{\pi}\int_{0}^{\f{\pi}{2}} \int_{\R^3} |e^{-i\theta H}\vphi|^{2\sigma+2} dxd\theta.
\end{align*}
However, when $\s>1$, we can prove that the minimizer of \eqref{v} does not exist (see Remark \ref{Fvp}). Roughly speaking, this is because the following Strichartz estimate does not hold for $\s>1$:
\begin{equation*}
\Big(\int_{0}^{\f{\pi}{2}}\int_{\R^2} |e^{-i\theta H}\vphi|^{2\sigma+2} dyd\theta\Big)^{\f{1}{2\s+2}}\lesssim \Big(\int_{\R^2} |\vphi|^2dy\Big)^{\h}.
\end{equation*}
To avoid this difficulty, we set $\phi(t,x)=e^{it}e^{itH}\vphi(x)$.\\

On the other hand, when $\s\le1$, it holds from Lemma \ref{bdnl} that 
\begin{equation}\label{SS}
	\Big(\int_{0}^{\f{\pi}{2}}\int_{\R^3} |e^{-i\theta H}\vphi|^{2\sigma+2} dxd\theta\Big)^{\f{1}{2\s+2}}\lesssim \Big(\int_{\R^2} |\partial_z\vphi|^2 +|\vphi|^2dx\Big)^{\h}.
\end{equation}
Hence, we expect that there exists a minimizer of \eqref{v}, which is an extremizer for \eqref{SS}. By the lens transform (cf.~\cite{pseudo}), the Strichartz estimate for $\{e^{-i\theta H}\}_{\theta\in[0,\pi/2]}$ is associated with those for the free Schr\"{o}dinger propagator $\{e^{i\theta \Delta}\}_{\theta\in \R}$. The problems concerning the best constants and extremizers for the Strichartz estimates for $\{e^{i\theta \Delta}\}_{\theta\in \R}$ have been extensively studied, cf.~\cite{MSSt, EMSt, MSti, SSld, SiS, RSt1D, RSC}. See also \cite{CR, CRH}. Because the main subject of this study is the case $2\le\s$, we will postpone this problem to future work. 
\end{remark}

We next study the global dynamics of the solutions to \eqref{NLS} below the ground state for \eqref{v2}. 
Our strategy is based on the work of Ardila-Carles \cite{Gdb}. We introduce the functional
\begin{equation}
	P[\vphi]:= 2\int_{\R^3} |\partial_z\vphi|^2 dx - \f{2\s}{\pi(\s+1)}\int_{0}^{\f{\pi}{2}} \int_{\R^3} |e^{-i\theta H}\vphi|^{2\sigma+2} dxd\theta.
\end{equation}
We also define the following subsets of $B^1$,
\begin{equation}\label{K+0}
	\cK^+:=\{ \vphi\in B^1 : S[\vphi]<S[Q], \quad P[\vphi]\ge0\},
\end{equation}
\begin{equation}\label{K-0}
	\cK^-:=\{ \vphi\in B^1 : S[\vphi]<S[Q], \quad P[\vphi]<0\},
\end{equation}
where $Q$ is the minimizer of \eqref{v2}. Note that $S[\cdot]$ is the conserved quantity of the solutions to \eqref{NLS} and \eqref{NLS2}. We prove that all solutions whose initial data belong to $\cK^+$ exist globally and scatter, and those whose initial data belong to $\cK^-$ blow up at finite or infinite time. The scattering result is based on Kenig–Merle's concentration compactness and rigidity arguments \cite{KM}. In \cite{Gdb}, because the model contains partial harmonic oscillator, the method developed in \cite{DHS, HR} could not be applied. They employed a variational approach based on the work of Ibrahim-Masmoudi-Nakanishi \cite{IMN}. We follow this approach.

We replace the condition \eqref{as00} in our case. In equation \eqref{NLS},  the global dispersive effect arises in one dimension, and as mentioned in Proposition \ref{wp}, equation \eqref{NLS} becomes energy-critical when $\s=4$. Thus, to obtain the corresponding results, we assume that 
\[2<\s< 4.\]
However, we include the case $\s=2$ except for the scattering result.\\

In proving scattering, the condition $\s>2$ presents some technical difficulties. To obtain the scattering result, a global-in-time Srtichartz's estimate is required. Because NLS equations with a standard power-type nonlinearity, such as \eqref{NLSep} and \eqref{NLSH0}, are energy-supercritical when $\s>2$, we need to use Strichartz's estimate for $e^{-i\theta H}$ (in $(\theta, y)$) in the analysis of the averaged nonlinearity. This implies that at least one anisotropic spatial norm is necessary, and we have to consider the balance of four integrabilities in the $t$, $\theta$, $y$, and $z$ directions. 
On the other hand, for technical reasons, we would prefer to minimize the use of anisotropic norms in constructing a critical element (in Sections \ref{keyprop} and \ref{proofsc}) that is essential for the concentration of compactness and rigidity arguments.

Then, we introduce new space-time norms and derive new global-in-time Strichartz's estimates (see Section \ref{prel}). By using these tools, we can effectively employ the propagator $e^{-i\theta H}$ included in the averaged nonlinearity and overcome the difficulties.

\subsection{Main result}
First, we present the result for the variational problem \eqref{v2}.
\begin{prop}\label{main0}
	Let $\lmd=-1$ and $\sigma<4$. Then, the minimizer $Q$ of \eqref{v2} exists. Moreover, this is an extremizer of the Strichartz estimate \eqref{SSQ}.
\end{prop}
Next, we present the main result of this study. $\cK^{\pm}$ are defined by \eqref{K+0} and \eqref{K-0}, respectively.

\begin{thm}\label{main}
	Let $\lmd=-1$ and $2\le\sigma<4$. Let $\phi\in C((T_- , T_+), B^1)$ be the solution to \eqref{NLS} with $\phi(0)=\phi_0\in B^1$.
	\begin{enumerate}
		\item 
		If $\phi_0\in \cK^+$, then the corresponding solution $\phi(t)$ exists globally. Moreover, if $2<\sigma<4$, $\phi(t)$ scatters in $B^1$; there exist $\phi^\pm \in B^1$ such that
		\begin{equation}
			\lim_{t\to \pm\infty}\|\phi(t)-e^{it\partial_z^2}\phi^\pm\|_{B^1}= 0. 
		\end{equation}
		\item
		If $\phi_0\in \cK^-$, one of the following two cases occurs.
		\begin{itemize}	
			\item  The corresponding solution $\phi(t)$ blows up in positive time. That is, $T_+<\infty$ and 
			\begin{equation}
				\lim_{t\to T_+}\|\partial_z \phi(t)\|_{L^2} =\infty.
			\end{equation}
			\item 
			The corresponding solution $\phi(t)$ blows up at infinite positive time. That is, $T_+=\infty$ and there exists a sequence $\{t_k\}_{k=1}^\infty$ such that $t_k\to \infty$  and $\|\partial_z \phi(t)\|_{L^2} \to \infty$ as $k\to +\infty$.
		\end{itemize}
		The same statement holds for negative time. 
		Moreover, if $\phi_0$ satisfies $z\phi_0\in L^2(\R^3)$, the corresponding solution blows up in finite time.
	\end{enumerate}
\end{thm}

\quad \\
Our work also yields the scattering result in the defocusing case. 
\begin{cor}\label{main2}
	Let $\lmd=+1$, $2<\sigma<4$, and $\phi\in C(\R, B^1)$ be the global  solution to \eqref{NLS}, $\phi(t)$ scatters in $B^1$.
\end{cor}

\subsection{Future prospects}
Our scattering result does not include the case when $\s=2$. In this case, equation \eqref{NLS} is mass-critical with respect to $z$. Therefore, as \cite{Sc3DNLSphp}, we need to use the space $L_z^2\Sigma_y^1$ and refer to Dodson \cite{mc1DNLS}. The case $\s=4$ may be more challenging. When $\s=4$, \eqref{NLS} is energy-critical and invariant under the scaling \eqref{si}. Even if $\lmd=+1$, the global existence of the solution is unknown.
Moreover, there are still many problems concerning the Strichartz estimates \eqref{SSQ}, \eqref{SS}, and their extensions to lower and higher dimensions.

On the other hand, our new idea and tools are expected to be applied to a variety of averaged models, such as the resonant Schr\"{o}dinger systems presented in Section \ref{NLSonWGM} and the averaged  dispersion managed NLS in Section \ref{DMNLS} as well as the limiting model of NLS with strong confinement.

\quad\\
\noindent
\textbf{Organization of the paper.} In Section \ref{prel}, we introduce the notation and preliminary estimates used in this study. In Section \ref{vari}, we present the proof of Proposition \ref{main0} and variational estimates. In Section \ref{sharpt}, we prove Theorem \ref{main} except for the scattering result. Sections \ref{keyprop} and \ref{proofsc} provide the proof of the scattering result. In Section \ref{keyprop}, we prove two key propositions: a linear profile decomposition and a long-time perturbation lemma. Finally, in Section \ref{proofsc}, we complete the proof. 

\section{Preliminaries}\label{prel}

\subsection{Notation}

We write $X\lesssim Y$ to express $X\le CY$ for a positive constant $C$. We denote $L^p(\R^d)$ as the usual Lebesgue spaces. We often use the following notation: the standard Lebesgue norm
\[\|u\|_{L^p}=\|u\|_{L_x^p} =\Big( \int_{\R^3} |u(x)|^p dx \Big)^\f{1}{p}\]
and  the partial spatial norms (recall that  $x=(y,z)\in \R^2\times \R$),
\[ \|u\|_{L_y^p} =\Big( \int_{\R^2} |u(x)|^p dy \Big)^\f{1}{p} , \qquad \|u\|_{L_z^p} =\Big( \int_{\R} |u(x)|^p dz \Big)^\f{1}{p} .\]
If $I\subset \R$ is an interval, then the mixed Lebesgue norm on $I\times\R^3$ is defined as
\[ \|u\|_{L_t^qL^p_x(I\times\R^3)}= \|u\|_{L^q(I,L^p(\R^3))}=\Big(  \int_{I} (\int_{\R^3} |u(t,x)|^p dx )^{\f{q}{p}}\Big)^{\f{1}{q}}. \]
Similarly, we denote
\[ \|u\|_{L_y^{p_1}L_z^{p_2}(\R^3)} = \Big(  \int_{\R^2} (\int_{\R} |u(x)|^{p_2} dz )^{\f{p_1}{p_2}} dy \Big)^{\f{1}{p_1}} \]
\[ \|u\|_{L_z^{p_2}L_y^{p_1}(\R^3)} = \Big(  \int_{\R} (\int_{\R^2} |u(x)|^{p_1} dy )^{\f{p_2}{p_1}} dz \Big)^{\f{1}{p_2}}. \]
When $X$ and $Y$ are different norms, we denote
\[ \|u\|_{X\cap Y}:=\|u\|_X +\|u\|_Y.\]
If there is no confusion, the integral regions are often omitted. \\

\noindent
We use the bracket $\langle \cdot \rangle =(1+| \cdot |^2)^{1/2}$.\\

For $s\ge0$, we define the Hermite Sobolev space as
\[ B^s=B_{x}^s:=\{u\in H^s(\R^3):|y|^s u\in L^2(\R^3)\}\]
\begin{equation*}
	\begin{split}
		\|u\|_{B^s}&:= \Big(  \| |\nabla_x|^s u \|_{L^2(\R^3)}^2 +\| |y|^su \|_{L^2(\R^3)}^2 \Big)^\h 
	\end{split}
\end{equation*}
and
\[ \Sigma^s=\Sigma_{x}^s:=\{u\in H^s(\R^3):|x|^s u\in L^2(\R^3)\}\]
\begin{equation*}
	\begin{split}
		\|u\|_{\Sigma^s}&:= \Big( \| |\nabla_x|^s u \|_{L^2(\R^3)}^2 +\| |x|^su \|_{L^2(\R^3)}^2 \Big)^\h  
	\end{split}
\end{equation*}
where $|\nabla_x|^s= \mathcal{F}^{-1}|\xi|^s \mathcal{F}$,  $|\nabla_y|^s= \mathcal{F}^{-1}|\eta|^s \mathcal{F}$, $\xi=(\eta, \z)\in \R^2\times \R$, and $\cF$ denotes the Fourier transforms of the corresponding variables.  
Note that by the Hermite expansion in $y$, it holds that 

\begin{equation*}
	\|u\|_{B^s}^2 \simeq_s \|u\|_{B^s}^2+ \|u\|_{L^2}^2, \qquad \|u\|_{\Sigma^s}^2 \simeq_s \|u\|_{\Sigma^s}^2+\|u\|_{L^2}^2.
\end{equation*}

Recall that $H=-\Delta_y + |y|^2$. From \cite[Lemma 2.4]{Yajima}, for $1<p<\infty$ and $s \ge 0$, we obtain the norm equivalence:
\begin{equation}\label{eqH}
	\| \langle \nabla_y \rangle^s u\|_{L_y^{p} } +\| \langle y \rangle^s u \|_{L_y^{p}} \simeq \| H^{\f{s}{2}} u \|_{L_y^{p}}.
\end{equation}
Mikhlin's multiplier theorem yields
\begin{equation}\label{Mik}
	\||\nabla_x|^s f\|_{L_x^p}\simeq \||\nabla_y|^s f\|_{L_x^p} + \||\partial_z|^s f\|_{L_x^p} , \quad \text{for all}\quad  1< p <\infty, \quad 0\le s.
\end{equation}
By combining \eqref{eqH} and \eqref{Mik}, we obtain the norm equivalence
\begin{equation}\label{eqH2}
	\|\langle \nabla_x \rangle^s f\|_{L_x^p} + \| |y| f\|_{L_x^p} \simeq \| H^{\f{s}{2}} f\|_{L_x^p} + \||\partial_z|^s f\|_{L_x^p} , \quad \text{for all}\quad  1< p <\infty, \quad 0\le s.
\end{equation}
Let
\begin{equation*}
	D:=H-\partial_z^2 = -\Delta_y + |y|^2 -\partial_z^2 = -\Delta_x + |y|^2. 
\end{equation*}
From \cite[Theorem 2.1]{N.Ben}, we obtain the norm equivalence
\begin{equation}\label{eqD}
	\|D^\f{s}{2}u\|_{L^2} \simeq \|u\|_{B^s}.
\end{equation} 
Finally, the following notation is used:
\begin{equation*}
	F(f):=\lmd|f|^{2\s}f,
\end{equation*}
\begin{equation*}
	U(t):=e^{it\partial_z^2} \quad \text{and} \quad V(\theta):=e^{-i\theta H}.
\end{equation*}
On and after Section \ref{prel}, the potential $V(z)$ does not appear in this paper, so there is no confusion.

\subsection{Preliminary estimates}
\begin{lemma}[Ordinary and exotic Strichartz's estimates for $\{U(t) \}_{t\in\R}$]\label{St1}
	\quad\\
	Let $(p_0,q_0)$ and $(p_1,q_1)$ be 1-dimensional admissible pairs;  that is, for $j=0,1$,
	\[ 2\le p_j \le \infty   \qquad  \text{and} \qquad \f{2}{q_j}+\f{1}{p_j}=\f{1}{2}.\]
	Then, for any  $T \in(0, \infty]$, we obtain
	\[ \Big\| U(t)f \Big\|_{L^{q_0}([0,T], L^{p_0})} \lesssim \| f \|_{L^2}, \]
	\[ \Big\| \int_0^{t} U(t-\tilde{t}) g(\tilde{t}, \cdot) d\tilde{t} \Big\|_{L^{q_0}([0,T], L^{p_0})}   \lesssim  \| g \|_{L^{q_1'}([0,T],L^{p_1'})}.\]
	In addition, if $p\in (2, \infty]$ and $q\in (2, \infty)$ satisfy
	\[ 0<\f{1}{q}+\f{1}{p}<\h, \]
	we have for $\tilde{q}=(\h+\f{1}{q}+\f{1}{p})^{-1}$,
	\[ \Big\| \int_0^{t} U(t-\tilde{t}) g(\tilde{t}, \cdot) d\tilde{t} \Big\|_{L^{q}([0,T], L^{p})}   \lesssim  \| g \|_{L^{\tilde{q}}([0,T],L^{p'})}.\]
\end{lemma}

\begin{lemma}[\cite{StH} Ordinary and exotic Strichartz's estimates for $\{V(\theta) \}_{\theta\in\R}$]\label{St2}
	\quad\\
	Let $(p_0,q_0)$ and $(p_1,q_1)$ be 2-dimensional admissible pairs;  that is, for $j=0,1$
	\[ 2\le p_j < \infty   \qquad  \text{and} \qquad \f{2}{q_j}+\f{2}{p_j}=1.\]
	Then, for any  $T>0$, we have
	\[ \Big\| V(\theta)f \Big\|_{L^{q_0}([0,T], L^{p_0})} \lesssim (1+T)^{1/q_0} \| f \|_{L^2}, \]
	\[ \Big\| \int_0^{\theta} V(\theta-\tilde{\theta}) g(\tilde{\theta}, \cdot) d \tilde{\theta}\Big\|_{L^{q_0}([0,T], L^{p_0})}   \lesssim (1+T)^{1/q_0 +1/q_1} \| g \|_{L^{q_1'}([0,T],L^{p_1'})}.\]
	In addition, if $p,q\in (2, \infty)$ satisfy
	\[ 0<\f{1}{q}+\f{2}{p}<1,\]
	we have for $\tilde{q}=(\f{1}{q}+\f{2}{p})^{-1}$,
	\[ \Big\| \int_0^{\theta} V(\theta-\tilde{\theta}) g(\tilde{\theta}, \cdot) d \tilde{\theta}\Big\|_{L^{q}([0,T], L^{p})}   \lesssim (1+T)^{1/q +1/\tilde{q}} \| g \|_{L^{\tilde{q}}([0,T],L^{p'})}.\]
\end{lemma}

\quad\\

To obtain Proposition \ref{Ltp} in Section \ref{keyprop}, we require the following estimates.

\begin{lemma}[Combination of two Strichartz's estimates]\label{ExSt}
	Let $I\subset \R$ be an interval and $\sigma\in [2,4]$. Let $p\in [\f{2\s}{\s-1}, \f{3\s-2}{\s-1}]$, $q\in [2\s, 3\s-2]$, and $s\in[\f{3(2\s-1)(\s-2)}{2(3\s-2)(\s-1)},\f{3(\s-2)}{2(\s-1)}]$ satisfy
	\begin{equation}\label{Exi}
		\f{1}{p}=\f{1}{2\s}\Big(\f{2s(\s-1)}{3}+1\Big), \quad 	q=(\s-1)p.
	\end{equation}
	Let $G:\R\times[0,\f{\pi}{2}]\times \R^3\to \C$ be an appropriate function, and define $G_{\rm{av}}$ as
	\begin{equation}\label{Gav}
		G_{\rm{av}}(t,x):=\f{2}{\pi}\int_{0}^{\f{\pi}{2}} V(-\theta) G(t, \theta,x) d\theta.
	\end{equation}
	Then, we have 
	\begin{equation}\label{ExSt31}
		\begin{split}
			\Big\|V(\theta)\int_0^t U(t-s) G_{\rm{av}}(s)ds  &\Big\|_{L_t^{2q}L_\theta^{q}L_x^{p}(I\times[0, \f{\pi}{2}] \times \R^3)} \lesssim \| G\|_{L_t^{\tilde{q_t}}L_\theta^{\tilde{q_\theta}}L_x^{p'}(I\times[0, \f{\pi}{2}] \times \R^3)},
		\end{split}
	\end{equation}
	where $\tilde{q_t}$ and $\tilde{q_\theta}$ are given by
	\begin{equation}\label{Exitilde}
		\f{1}{\tilde{q_t}}=\f{2\s-1}{2q} + \h, \quad  	\f{1}{\tilde{q_\theta}}=\f{2\s-1}{q} .
	\end{equation}
\end{lemma}

\begin{proof}
	From Minkowski's inequality and the last estimate in Lemma \ref{St1}, we have
	\begin{equation}
		\begin{split}
			&\Big\|V(\theta)\int_0^t U(t-s) G_{\rm{av}}(s)ds  \Big\|_{L_t^{2q}L_\theta^{q}L_x^{p}(I\times[0, \f{\pi}{2}] \times \R^3)}  \\
			\lesssim & \Big\|\int_0^t U(t-s) V(\theta)G_{\rm{av}}(s)ds  \Big\|_{L_\theta^{q}L_y^{p}L_t^{2q}L_z^{p}} \\
			\lesssim & \|V(\theta)G_{\text{av}}\|_{L_\theta^{q} L_y^{p} L_t^{\tilde{q_t}}L_z^{p'}}= \Big\|\f{2}{\pi}\int_{0}^{\f{\pi}{2}} V(\theta-\tilde{\theta})G(\cdot,\tilde{\theta}, \cdot)) d\tilde{\theta}\Big\|_{L_\theta^{q} L_y^{p} L_t^{\tilde{q_t}}L_z^{p'}} \\
		\end{split}
	\end{equation}
	We use Minkowski's inequality again and the last statement in Lemma \ref{St2} to obtain
	\begin{equation}
		\begin{split}
			\Big\|\f{2}{\pi}\int_{0}^{\f{\pi}{2}} V(\theta-\tilde{\theta})G(\cdot,\tilde{\theta}, \cdot) d\tilde{\theta}\Big\|_{ L_t^{\tilde{q_t}}L_z^{p'}L_\theta^{q} L_y^{p}} 
			&\lesssim \| G\|_{L_t^{\tilde{q_t}}L_z^{p'}L_\theta^{\tilde{q_\theta}} L_y^{p'}} \lesssim \| G\|_{L_t^{\tilde{q_t}}L_\theta^{\tilde{q_\theta}} L_x^{p'}}.
		\end{split}
	\end{equation}
	As $\f{3(2\s-1)(\s-2)}{2(3\s-2)(\s-1)}\le s$ and $p\le \f{3\s-2}{\s-1}$, we have $\tilde{q_\theta}\le p'$.
\end{proof}

\begin{prop}[Strichartz's estimate for the averaged nonlinearity]\label{ExSt2}
	Let $I\subset \R$ be an interval and $\sigma\in [2,4]$. Let $p\in [\f{2\s}{\s-1}, \f{3\s-2}{\s-1}]$, $q\in [2\s, 3\s-2]$, and $s\in[\f{3(2\s-1)(\s-2)}{2(3\s-2)(\s-1)},\f{3(\s-2)}{2(\s-1)}]$ satisfy \eqref{Exi}. Let $X= I_d, H^{\f{s'}{2}}, |\partial_z|^{s'}$ for $0< s'\le1$. Then, we have
	\begin{equation*}
		\begin{split}
			\Big\|&XV(\theta)\int_0^t U(t-s) \Fav(\phi(s))ds  \Big\|_{L_t^{2q}L_\theta^{q}L_x^{p}(I\times[0, \f{\pi}{2}] \times \R^3)} \\
			&\lesssim \|XV(\theta)\phi\|_{L_t^{2q} L_\theta^{q}L_x^{p}(I\times[0, \f{\pi}{2}] \times \R^3)} \|H^\h \phi \|_{L_t^4 L_z^\infty L_y^2(I\times[0, \f{\pi}{2}] \times \R^3)}^2 \|V(\theta)\phi\|_{L_t^{2q} L_\theta^{q} L_x^{p_0}(I\times[0, \f{\pi}{2}] \times \R^3)}^{2\s-2},
		\end{split}
	\end{equation*}
	where, $p_0=(1/p-s/3)^{-1}$.
\end{prop}

\begin{proof}
	From Lemma \ref{ExSt}, we obtain
	\begin{equation*}
		\begin{split}
			\Big\|&X V(\theta)\int_0^t U(t-s) \Fav(\phi(s))ds  \Big\|_{L_t^{2q}L_\theta^{q}L_x^{p}(I\times[0, \f{\pi}{2}] \times \R^3)} \lesssim  \| X(|V(\theta)\phi|^{2\s}V(\theta)\phi)\|_{L_t^{\tilde{q_t}}L_\theta^{\tilde{q_\theta}} L_x^{p'}},
		\end{split}
	\end{equation*}
	where $\tilde{q_t}$ and $\tilde{q_\theta}$ are the same as \eqref{Exitilde}. 
	From the definition of $p$, we see that it holds
	\begin{equation}\label{ptilde}
		\f{1}{p'}=\f{1}{p}+2(\f{1}{2}-\f{1}{2}) + (2\s-2)(\f{1}{p}-\f{s}{3})=  \f{1}{p}+\f{2\s-2}{p_0}.
	\end{equation}
	Note that \eqref{Exitilde} and \eqref{ptilde}, when $X=I_d$, we use  H\"{o}lder's inequality and Gagliardo-Nirenberg's inequality to obtain
	\begin{equation*}
		\begin{split}
			\| X(|  V(\theta)\phi|^{2\s}V(\theta)\phi)\|_{L_t^{\tilde{q_t}}L_\theta^{\tilde{q_\theta}} L_x^{p'}}
			&\lesssim \|X V(\theta)\phi\|_{L_t^{2q} L_\theta^{q}L_x^{p}} \||  V(\theta)\phi|^{2\s}\|_{L_t^{\f{1}{\f{\s-1}{q}+\h}} L_\theta^{\f{q}{2\s-2}}L_x^{\f{p_0}{2\s-2}}} \\
			&\lesssim \|X V(\theta)\phi\|_{L_t^{2q} L_\theta^{q}L_x^{p}} \|H^\h V(\theta) \phi \|_{L_t^4 L_\theta^\infty L_z^\infty L_y^2}^2 \|V(\theta)\phi\|_{L_t^{2q} L_\theta^{q} L_x^{p_0}}^{2\s-2}\\
			&\lesssim \|X V(\theta)\phi\|_{L_t^{2q} L_\theta^{q}L_x^{p}} \|H^\h  \phi \|_{L_t^4 L_z^\infty L_y^2}^2 \|V(\theta)\phi\|_{L_t^{2q} L_\theta^{q} L_x^{p_0}}^{2\s-2}.
		\end{split}
	\end{equation*}	
	When $X=H^{\f{s'}{2}}$, we use \eqref{eqH} and the fractional chain rule in $y$ to obtain the first inequality. Note that $[H^{s'/2}, V(\theta)]=0$ holds for all $s'>0$. In the case $X=|\partial_z|^{s'}$, we use the fractional chain rule in $z$.
\end{proof}

\begin{remark}\label{ind}
	To determine the aforementioned Lebesgue indices, we consider the following simultaneous equations:

	\begin{equation*}
		1-\f{1}{p_y}= \f{1}{p_y}+2(\f{1}{r_y}-\f{1}{2}) + (2\s-2)(\f{1}{p_y}-\f{s}{3})
	\end{equation*}
	\begin{equation*}
		1-\f{1}{p_z}= \f{1}{p_z}+ \f{2}{r_z} + (2\s-2)(\f{1}{p_z}-\f{s}{3})
	\end{equation*}
	\begin{equation*}
		\f{1}{q_t}+1=\h(1-\f{2}{p_z})+\f{1}{q_t}+\f{2}{r_t}+ \f{2\s-2}{q_t}
	\end{equation*}
	\begin{equation*}
		\f{1}{q_\theta}+1=1-\f{2}{p_y}+\f{1}{q_\theta} + \f{2}{r_\theta} +\f{2\s-2}{q_\theta}.
	\end{equation*}
	These equations are inspired by the scale invariance when $\sigma=4$ (and $s=1$), cf.~{\rm\cite{K}}. 
	To simplify the spatial indices, we restrict $p_y=p_z(=:p)$. Moreover, in order to make $r_y$ an element of a 2-dimensional admissible pair, we impose $r_y\ge2$. Subsequently, we set $r_y=2$ and $r_z=\infty$. Then, by the first and second equations, we obtain the first identity in \eqref{Exi}. To make the pair $(r_z, r_t)$ 1-dimensional admissible, we set $r_t=4$. Similarly, we set $r_\theta=\infty$. Then, from the third and fourth equations, we obtain $q_\theta=(\s-1)p$ and $q_t=2(\s-1)p$. These equations yield the second identity in \eqref{Exi}. Moreover, if we consider $\f{1}{q_\theta}+\f{1}{p}\le\h$, we can determine the upper endpoints of $p$, $q$, and $s$. Regarding the lower endpoints, refer to the last line of the proof of Lemma \ref{ExSt}.

\end{remark}

\begin{remark}\label{indad}
	In particular, if we set $s=\f{3(\s-2)}{2(\s-1)}$ in \eqref{Exi}, we obtain 
	\begin{equation}
		p=\f{2\s}{\s-1}, \quad q=2\s, \quad s=\f{3(\s-2)}{2(\s-1)}.
	\end{equation} 
	Then, $(p,2q)$ is a 1-dimensional admissible pair and $(p,q)$ is a 2-dimensional admissible pair.
\end{remark}
\quad

When $s=\f{3(\s-2)}{2(\s-1)}$, we can easily extend Lemma \ref{ExSt} and Proposition \ref{ExSt2} to the following lemmas, respectively.

\begin{lemma}\label{St3}
	Let $I\subset \R$ be an interval, $(p_z, q_t)$ and $(\tilde{p}_z, \tilde{q}_t)$ be 1-dimensional admissible pairs, and $(\tilde{p}_y, \tilde{q}_\theta)$ be a 2-dimensional admissible pair.
	Let $G:\C\to \C$ be an appropriate function, and $G_{\rm{av}}$ be \eqref{Gav}.
	If $\tilde{p}_z \le \tilde{q}_\theta$, then we have 
	\begin{equation*}
		\begin{split}
			\Big\|\int_0^t U(t-s) G_{\rm{av}}(s)ds  &\Big\|_{L_t^{q_t}L_z^{p_z}L_y^2(I\times[0, \f{\pi}{2}] \times \R^3)} \lesssim \| G\|_{L_t^{\tilde{q_t}'}L_\theta^{\tilde{q_\theta}'}L_z^{\tilde{p_z}'}L_y^{\tilde{p_y}'}(I\times[0, \f{\pi}{2}] \times \R^3)}.
		\end{split}
	\end{equation*}
\end{lemma}

\begin{proof} The proof is similar to Lemma \ref{ExSt}.
\end{proof}

\begin{lemma}\label{St4}
	Let $I\subset \R$ be an interval, $(p_z, q_t)$ be a 1-dimensional admissible pair.
	Let $X= I_d, H^{\f{s'}{2}}, |\partial_z|^{s'}$ for $0<s'\le 1$. Then, we have 
	\begin{equation}\label{St}
		\begin{split}
			&\Big\|X\int_0^t U(t-s) F_{\rm{av}} (\phi(s))ds  \Big\|_{L_t^{q_t}L_z^{p_z}L_y^{2}(I\times[0, \f{\pi}{2}] \times \R^3)} \\
			\lesssim& \|X V(\theta) \phi\|_{L_t^{2q}L_\theta^{q}L_x^p(I\times[0, \f{\pi}{2}] \times \R^3)} \|H^\f{1}{2}\phi\|_{L_t^{4}L_z^{\infty}L_y^2 (I\times \R^3)}^2  \| V(\theta) \phi\|_{L_t^{2q}L_\theta^{q}L_x^{p_0}(I\times[0, \f{\pi}{2}] \times \R^3)}^{2\s-2}, 
		\end{split}
	\end{equation}
	where indices $p$, $q$, $r$, $s$ are given in Remark \ref{indad}.
\end{lemma}
\begin{proof}
	 Combine the proof of Lemma \ref{St3} with Proposition \ref{ExSt2}.
\end{proof}

Next, we provide the following energy estimate.

\begin{lemma}\label{bdnl}
	Let $1\le \s\le 4$. For any $\psi\in B^1$, we have
	\begin{equation}\label{ieqnl}
		\|V(\theta)\psi\|_{L_{\theta,x}^{2\s+2}([0,\f{\pi}{2}]\times\R^3)} \lesssim \|\psi\|_{L^2}^{\f{4-\s}{2(\s+1)}}\|\psi\|_{B^1}^{\f{3\s-2}{2(\s+1)}}.
	\end{equation}	
\end{lemma}

\begin{proof}
	First, we prove \eqref{ieqnl} for $\sigma=1$. From Lemma \ref{St2}, Minkowski's inequality, and Gagliardo-Nirenberg's inequality in $z$, we obtain
	\begin{equation}\label{ieqnl-0}
		\begin{split}
			\|V(\theta)\psi\|_{L_{\theta,x}^{4}([0,\f{\pi}{2}]\times\R^3)} \lesssim  \|\psi\|_{L_z^4L_y^2} \lesssim  \|\psi\|_{L_y^2L_z^4}\lesssim \|\psi\|_{L^2}^\f{3}{4} \|\partial_z\psi\|_{L^2}^\f{1}{4}.
		\end{split}
	\end{equation}
	Next, we prove \eqref{ieqnl} holds for $\sigma=4$. From Gagliardo-Nirenberg's and H\"{o}lder's inequalities, we have
	\begin{equation}\label{ieqnl-1}
		\begin{split}
			\|V(\theta)\psi\|_{L_{\theta,x}^{10}([0,\f{\pi}{2}]\times\R^3)}&\lesssim \Big\| \|\nabla_yV(\theta)\psi\|_{L_y^3}^{\f{1}{5}} \|V(\theta)\psi\|_{L_y^6}^\f{4}{5} \Big\|_{L_{\theta,z}^{10}}\\
			&\lesssim  \|\nabla_yV(\theta)\psi\|_{L_\theta^6L_z^2L_y^3}^{\f{1}{5}} \|V(\theta)\psi\|_{L_\theta^{12}L_z^{\infty}L_y^6}^\f{4}{5}. \\
		\end{split}
	\end{equation}
	From \eqref{eqH}, Minkowski's inequality, and Lemma \ref{St2}, we have
	\begin{equation}\label{ieqnl-2}
		\|\nabla_yV(\theta)\psi\|_{L_\theta^6L_z^2L_y^3} \lesssim  \|H^\h V(\theta)\psi\|_{L_z^2L_\theta^6L_y^3} \lesssim \|H^\h \psi\|_{L_x^2}.
	\end{equation}
	From Minkowski's and Gagliardo-Nirenberg's inequality, \eqref{eqH}, and Lemma \ref{St2}, we have
	\begin{equation}\label{ieqnl-3}
		\begin{split}
			\|V(\theta)\psi\|_{L_\theta^{12}L_z^{\infty}L_y^6} &\lesssim 	\|V(\theta)\psi\|_{L_\theta^{12}L_y^6L_z^{\infty}}\\
			&\lesssim\|\partial_z V(\theta)\psi\|_{L_\theta^{3}L_y^6L_z^2}^\f{1}{4} \|V(\theta)\psi\|_{L_\theta^{\infty}L_{y,z}^6}^\f{3}{4} \\
			&\lesssim\|\partial_z V(\theta)\psi\|_{L_z^2L_\theta^{3}L_y^6}^\f{1}{4} \|\nabla_y V(\theta)\psi\|_{L_\theta^{\infty}L_x^2}^\f{1}{2}   \|\partial_z V(\theta)\psi\|_{L_\theta^{\infty}L_x^2}^\f{1}{4}\\	
			&\lesssim \|H^\h \psi\|_{L^2}^{\h} \|\partial_z\psi\|_{L^2}^\h.
		\end{split}
	\end{equation}
	By combining \eqref{ieqnl-1} with \eqref{ieqnl-3}, we have 
	\begin{equation*}
		\|V(\theta)\psi\|_{L_{\theta,x}^{10}} \lesssim \| H^{\h} \psi\|_{L^2}^{\f{3}{5}} \|\partial_z \psi\|_{L^2}^{\f{2}{5}}.
	\end{equation*}
	When $1<\sigma<4$, it is sufficient to interpolate the cases $\sigma=1$ and $\sigma=4$:
	\begin{equation}\label{ieqnl-4}
		\|V(\theta)\psi\|_{L_{\theta,x}^{2\s+2}([0,\f{\pi}{2}]\times\R^3)} \lesssim	\|V(\theta)\psi\|_{L_{\theta,x}^{4}}^{\f{2(4-\s)}{3(\s+1)}}
		\|V(\theta)\psi\|_{L_{\theta,x}^{10}}^{\f{5(\s-1)}{3(\s+1)}}\lesssim  \|\psi\|_{L^2}^{\f{4-\s}{2(\s+1)}}\|\psi\|_{B^1}^{\f{3\s-2}{2(\s+1)}}.
	\end{equation}
\end{proof}

\begin{remark}\label{rembdnl}
	When $0<\s\le1$, from \eqref{ieqnl-0} and H\"{o}lder's inequality in $\theta$, we obtain 
	\[\|V(\theta)\psi\|_{L_{\theta,x}^{2\s+2}([0,\f{\pi}{2}]\times\R^3)} \lesssim \|\psi\|_{L^2}^{\f{3}{4}}\|\partial_z\psi\|_{L^2}^\f{1}{4}.
\]
\end{remark}

\subsection{Conservation law and the equivalent model}

\begin{lemma}\label{conke}
	Let $\h\le\sigma\le4$. Then, any solution to (\ref{NLS}) in $B^1$ conserves 
	\begin{equation}\label{ke}
		K[\phi]:=\h \langle H\phi, \phi\rangle =  \h\| \nabla_y \phi \|_{L^2}^2+\f{1}{2}\| y \phi\|_{L^2}^2 .
	\end{equation}
\end{lemma}
 When $\sigma\in \N$, this lemma is proved in \cite[Lemma 2.8]{K}. We extend the conservation law of $K$ to all positive $\sigma$.

\begin{proof}
	Let $\phi(t)$ be a solution to \eqref{NLS} in $B^2$, which is a Banach algebra. Then, a simple calculation yields
	\begin{equation*}
		\begin{split}
			\f{d}{dt}K[\phi(t)]&=\text{Re} \langle H\phi(t), \partial_t\phi(t)\rangle = -\text{Im} \langle H\phi(t), \lmd\Fav(\phi(t))\rangle\\
			&=-\lmd\f{2}{\pi}\text{Im} \int_{0}^{\f{\pi}{2}} \int_{\R^3} |V(\theta)\phi(t)|^{2\s} \big(HV(\theta)\phi(t)\big) \overline{V(\theta)\phi(t) }dx d\theta \\
			&=-\lmd\f{2}{\pi}\text{Im} \int_{0}^{\f{\pi}{2}} \int_{\R^3} |V(\theta)\phi(t)|^{2\s} \big(i\partial_\theta V(\theta)\phi(t)\big) \overline{V(\theta)\phi(t) } dx d\theta \\
			&= -\lmd\f{2}{\pi}\int_{0}^{\f{\pi}{2}} \int_{\R^3} |V(\theta)\phi(t)|^{2\s} \text{Re}\Big[\big(\partial_\theta V(\theta)\phi(t)\big) \overline{V(\theta)\phi(t) }\Big]  dxd\theta\\
			&=-\f{\lmd}{\pi} \int_{\R^3}  \int_{0}^{\f{\pi}{2}}|V(\theta)\phi(t)|^{2\s} \partial_\theta |V(\theta)\phi(t)|^2  d\theta dx\\
			&=-\f{\lmd}{\pi(\s+1)}  \int_{\R^3}\int_{0}^{\f{\pi}{2}} 
			\partial_\theta|V(\theta)\phi(t)|^{2\s+2}
			d\theta dx.
		\end{split}
	\end{equation*}
	Using the Hermite expansion, we obtain 	
	\begin{equation}\label{sym}
		\begin{split}
			V\Big(\f{\pi}{2}\Big)\phi(y,z)&=\sum_{n=0}^\infty e^{-2i(n+1)\times \f{\pi}{2}}\Pi_n\phi(y,z) = \sum_{n=0}^\infty (-1)^{(n+1)} \Pi_n\phi(y,z) = -\phi(-y,z),\\
		\end{split}
	\end{equation}
	where we use $\Pi_n f(y)=(-1)^n\Pi_nf(-y)$ (see cf.~\cite{CR}). Therefore, we obtain
	\begin{equation*}
		\int_{\R^3}\int_{0}^{\f{\pi}{2}} 
		\partial_\theta|V(\theta)\phi(t)|^{2\s+2}d\theta dx=\int_{\R^3} \Big[ |V(\theta)\phi(t)|^{2\s+2} \Big]_{\theta=0}^{\theta=\f{\pi}{2}} dx=0,
	\end{equation*}
	and
	\begin{equation*}
		\begin{split}
			\f{d}{dt}K[\phi(t)]&= 0.
		\end{split}
	\end{equation*}
	When $\h\le\s\le 4$, from Proposition \ref{wp}, \eqref{NLS} is locally well-posed in $B^1$. Therefore, we can extend the conservation laws of $K$ to solutions in $B^1$.
\end{proof}

\begin{remark}\label{equivNLS2}
	Suppose that $\phi(t)$ is a solution to \eqref{NLS} in $B^1$. Then, $\psi(t):=V(t)\phi(t)$ solves \eqref{NLS2}. A simple calculation shows
	\begin{equation*}
		i\partial_t\psi(t)=H\psi(t)-\partial_z^2 \psi(t) +\lmd V(t)\Fav(\phi(t)).
	\end{equation*}
	and
	\begin{equation*}
		\begin{split}
			V(t)\Fav(\phi(t))=\f{2}{\pi} \int_{-t}^{\f{\pi}{2}-t} V(-\theta)\Big(|V(\theta)\psi(t)|^{2\s}V(\theta)\psi(t)\Big) d\theta.
		\end{split}
	\end{equation*}
	By using \eqref{sym}, we have
	\begin{equation*}
		\begin{split}
			&\int_{-t}^0 V(-\theta)\Big(|V(\theta)\psi(t)|^{2\s}V(\theta)\psi(t)\Big) d\theta \\
			&= \int_{\f{\pi}{2}-t}^{\f{\pi}{2}} V(-\theta+\f{\pi}{2})\Big(|V(\theta-\f{\pi}{2})\psi(t)|^{2\s}V(\theta-\f{\pi}{2})\psi(t)\Big) d\theta \\
			&= -\int_{\f{\pi}{2}-t}^{\f{\pi}{2}} V(-\theta+\f{\pi}{2})\Big(|V(\theta)\psi(t, -y,z)|^{2\s}V(\theta)\psi(t, -y,z)\Big) d\theta \\
			&= \int_{\f{\pi}{2}-t}^{\f{\pi}{2}} V(-\theta)\Big(|V(\theta)\psi(t, y,z)|^{2\s}V(\theta)\psi(t, y,z)\Big) d\theta.
		\end{split}
	\end{equation*}
	Hence, it holds $V(t)\Fav(\phi(t))=\Fav(\psi(t))$.
	
\end{remark}

\section{Variational problem}\label{vari}

In this section, we assume $\lmd=-1$. We define the action function 
\begin{equation}
	S[\psi]= \h\|\nabla_x \psi\|_{L^2}^2+\h\|y\psi\|_{L^2}^2 + \h \|\vphi\|_{L^2}^2-\f{1}{\pi(\sigma+1)}\|V(\theta)\psi\|_{L_{\theta,x}^{2\sigma+2}([0,\f{\pi}{2}]\times \R^3)}^{2\sigma+2}
\end{equation}
and the Nehari functional
\begin{equation}
	I[\psi]= \|\nabla_x \psi\|_{L^2}^2+\|y\psi\|_{L^2}^2+  \|\vphi\|_{L^2}^2-\f{2}{\pi}\|V(\theta)\psi\|_{L_{\theta,x}^{2\sigma+2}([0, \f{\pi}{2}]\times \R^3)}^{2\sigma+2}
\end{equation}
on $B^1$. We prove the existence of a ground-state solution $Q$ to \eqref{ep2} and some lemmas that are required for classifying the solutions to \eqref{NLS} below the ground state. 
We define the scaling function $\psi_{\mu}^{a,b}$ of $\psi$ as
\begin{equation}\label{scaleab}
	\psi_{\mu}^{a,b}(x):=e^{a\mu} \psi(y, e^{b\mu}z),  \quad \mu\in \R.
\end{equation}
 The pair $(a,b)$ satisfies the following conditions:
\begin{equation}\label{asab}
	a>0, \quad  b\ge0, \quad 2a-b\ge0, \quad \s a-b>0, \quad (a,b) \neq (0,0).
\end{equation}
From the first and third conditions in \eqref{asab}, we have $(2\s +2)a-b>0$. Additionally, we often assume that 
\begin{equation}\label{asab2}
	2a-b>0.
\end{equation}
By \eqref{scaleab}, it holds that 
\begin{equation*}
	\begin{split}
		\| \psi_\mu^{a,b}\|_{L^2}^2&= e^{(2a-b)\mu}\|\psi\|_{L^2}^2,\\
		\|\nabla_y\psi_\mu^{a,b}\|_{L^2}^2&= e^{(2a-b)\mu}\|\nabla_y\psi\|_{L^2}^2, \quad 	\| y\psi_\mu^{a,b}\|_{L^2}^2= e^{(2a-b)\mu}\|y\psi\|_{L^2}^2,
	\end{split}
\end{equation*}
\begin{equation*}
	\|\partial_z \psi_\mu^{a,b}\|_{L^2}^2=e^{(2a+b)\mu}\|\partial_z\psi\|_{L^2}^2,
\end{equation*}
\begin{equation*}
	\|V(\theta)\psi_\mu^{a,b}\|_{L_{\theta,x}^{2\s+2}([0,\f{\pi}{2}]\times\R^3)}^{2\s+2} = e^{((2\s+2)a-b)\mu}\|V(\theta)\psi\|_{L_{\theta,x}^{2\s+2}([0,\f{\pi}{2}]\times\R^3)}^{2\s+2}.
\end{equation*}
We define the functional $J^{a,b}$ as
\begin{equation*}
	\begin{split}
		&J^{a,b}[\psi]:= \partial_\mu S[\psi_\mu^{a,b}] |_{\mu=0}\\
		&=(2a-b)\Big(K[\psi]+ M[\psi]\Big)+\f{2a+b}{2}\|\partial_z \psi\|_{L^2}^2-\f{(2\s+2)a-b}{\pi(\s+1)}\|V(\theta)\psi\|_{L_{\theta,x}^{2\s+2}([0,\f{\pi}{2}]\times\R^3)}^{2\s+2}.
	\end{split}
\end{equation*}
In particular, $J^{1.0}=I$ and $J^{1,2}=P$.
Hereafter, we will omit the integral region of the potential energy,
\begin{equation*}
	\|V(\theta)\psi\|_{L_{\theta,x}^{2\s+2}}:=\|V(\theta)\psi\|_{L_{\theta,x}^{2\s+2}([0,\f{\pi}{2}]\times\R^3)}.
\end{equation*}

\begin{lemma}
	Let $0<\s<4$, and let $(a,b)$ satisfy \eqref{asab} and \eqref{asab2}. Let $\{v_k\}_{k=1}^\infty \subset B^1/\{0\}$ be a  sequence such that $\lim_{k\to+\infty}\|v_k\|_{B^1}=0$. Then, for a sufficiently large $k$, $J^{a,b}(v_k)>0$. 
\end{lemma}

\begin{proof}
	By using Lemma \ref{bdnl}, we obtain the lower bound for $J^{a,b}$:
	\begin{equation*}
		\begin{split}
			J^{a,b}[v_k]&\ge \f{2a-b}{2}\|v_k\|_{B^1}^2-\f{(2\s+2)a-b}{\pi(\s+1)}\|V(\theta) v_k\|_{L_{\theta,x}^{2\s+2}}^{2\s+2} \\
			&\ge \f{2a-b}{2}\|v_k\|_{B^1}^2-C\f{(2\s+2)a-b}{\pi(\s+1)} \|v_k\|_{B^1}^{2\s+2},
		\end{split}
	\end{equation*}
	where $C$ denotes a positive constant. Since $2a-b>0$, we obtain $J^{a,b}[v_k]>0$ for sufficiently large $k$.
\end{proof}

Next, we consider the following minimization problem:
\begin{equation}\label{minp}
	d^{a,b}:=\inf_{\psi \in \cA^{a,b}}S[\psi], \quad \cA^{a,b}:=\{\psi \in B^1 \backslash \{0\}, J^{a,b}[\psi]=0\}.
\end{equation}
Let
\begin{equation}
	U^{a,b}:=\{ \psi\in \cA^{a,b}:S[\psi]=d^{a,b} \}.
\end{equation}

\begin{lemma}\label{Uab}
	Let $0<\s<4$, and let $(a,b)$ satisfy \eqref{asab} and \eqref{asab2}. Then, set $U^{a,b}$ is not empty.
\end{lemma}

\begin{proof}
First, we introduce the functional:
	\begin{equation*}
		\begin{split}
			\tB^{a,b}[\psi]&:=S[\psi]-\f{1}{(2\s+2)a-b}J^{a,b}[\psi]\\
			&=\al(\|\psi\|_{L^2}^2 + \|\nabla_y\psi\|_{L^2}^2 + \|y\psi\|_{L^2}^2) +\beta \|\partial_z \psi\|_{L^2}^2,
		\end{split}
	\end{equation*}
	where, 
	\begin{equation*}
		\al=\f{1}{2}\Big(1-\f{2a-b}{(2\s+2)a-b} \Big)>0, \qquad 	\beta=\f{1}{2}\Big(1-\f{2a+b}{(2\s+2)a-b} \Big)>0.
	\end{equation*}
	Thus, $\tB^{a,b}[\psi] \simeq_{a,b} \|\psi\|_{B^1}^2$ for all $\psi \in B^1$. Since $J^{a,b}[\psi]=0$ for any $\psi\in \cA^{a,b}$, $\tB^{a,b}[\psi]=S[\psi]$ holds. Thus, 
	\begin{equation}\label{eqinf}
		d^{a,b}=\inf_{\psi\in\cA^{a,b}}\tB^{a,b}[\psi].
	\end{equation}
	\underline{\textbf{Step 1.} $d^{a,b}>0$.}
	We assume that $\psi\in \cA^{a,b}$. Because $J^{a,b}[\psi]=0$ and Lemma \ref{bdnl}, we obtain
	\begin{equation*}
		0< \|\psi\|_{B^1}^2 \simeq_{a,b} \|V(\theta)\psi\|_{L_{\theta,x}^{2\s+2}}^{2\s+2} \lesssim \|\psi\|_{B^1}^{2\s+2}.
	\end{equation*} 
	Therefore, $1\lesssim_{a,b} \|\psi\|_{B^1}^{2\s} \simeq_{a,b} (\tB^{a,b}[\psi])^{\s}$, implying $d^{a,b}>0$.\\
	
	\noindent
	\underline{\textbf{Step 2.} If $\psi\in B^1$ satisfies  $J^{a,b}[\psi]<0$, then $d^{a,b}<\tB^{a,b}[\psi]$.}
	Indeed, there exists $\nu\in (0,1)$ such that $J^{a,b}[\nu\psi]=0$. Thus, from the definition of $d^{a,b}$, we have
	\begin{equation*}
		d^{a,b}\le \tB^{a,b}[\nu\psi] =\nu^2 \tB^{a,b}[\psi] < \tB^{a,b}[\psi]. 
	\end{equation*}
	
	\noindent
	\underline{\textbf{Step 3.} Construction of a candidate for the minimizer. } Let $\{\psi_n\}_{n=1}^{\infty} \subset \cA^{a,b} $ be a minimizing sequence of $d^{a,b}$. That is, it holds that $\tB^{a,b}[\psi_n]\to d^{a,b}$ as $n\to \infty$.  Then, the sequence $\{\psi_n\}_{n=1}^{\infty}$ is bounded in $B^1$. Furthermore, in Step 1, we obtain:
	\begin{equation*}
		\|V(\theta)\psi_n\|_{L_{\theta,x}^{2\s+2}}^{2\s+2} \gtrsim_{a,b} \|\psi_n\|_{B^1}^2 \gtrsim_{a,b} \tB^{a,b}[\psi_n] \to d^{a,b}>0.
	\end{equation*}
	Hence, we have
	\[\limsup_{n\to \infty} \|V(\theta)\psi_n\|_{L_{\theta,x}^{2\s+2}}\gtrsim_{a,b}1.\] 
	In addition, by Lemma \ref{bdnl}, we obtain
	\begin{align*}
		\|V(\theta)\psi_n\|_{L_{\theta,x}^{2\s+2}} \lesssim  \|\psi_n\|_{L^2}^{\f{4-\s}{2(\s+1)}}\|\psi_n\|_{B^1}^{\f{3\s-2}{2(\s+1)}} \qquad \text{for} \quad 1\le\s<4,\\
		\|V(\theta)\psi_n\|_{L_{\theta,x}^{2\s+2}} \lesssim  \|\psi_n\|_{L^2}^{\f{3}{4}}\|\psi_n\|_{B^1}^{\f{1}{4}} \qquad \text{for} \quad 0<\s<1.
	\end{align*}
	By combining these estimates with the upper bound for $\|\psi_n\|_{B^1}$, we have
	\begin{equation*}
		\limsup_{n\to \infty}\|\psi_n\|_{L^2}\ge \exists C_{a,b} >0.
	\end{equation*}
	By \cite[Lemma 3.4]{Bell} (see also \cite{Ohta}), up to a subsequence, there exist $\psi^*\in B^1 \backslash \{0\}$ and $\{z_n\}_{n=1}^\infty \subset \R$ such that
	\begin{equation}\label{wc}
		\tilde{\psi}_n:=\psi_n(y,z-z_n) \rightharpoonup \psi^* \quad \text{weakly in} \quad B^1.
	\end{equation}
	
	\noindent
	\underline{\textbf{Step 4.} $J^{a,b}[\psi^*]=0$.} If $J^{a,b}[\psi^*]<0$, Step 2 results in $d^{a,b}<\tB^{a,b}[\psi^*]$. However, it holds that
	\begin{equation*}
		\tB^{a,b}[\psi^*]\le \liminf_{n\to\infty} \tB^{a,b}[\tilde{\psi}_n]=\liminf_{n\to\infty} \tB^{a,b}[\psi_n] =d^{a,b}.
	\end{equation*}
	This is a contradiction. 
	We assume $J^{a,b}[\psi^*]>0$. If we check 
	\begin{equation}\label{Bre1}
		\sup_{n\in \N}\|V(\theta)\tilde{\psi}_n\|_{L_{\theta,x}^{2\s+2}}<\infty
	\end{equation}
	and
	\begin{equation}\label{Bre2}
		V(\theta)\tilde{\psi}_n(x) \to V(\theta)\psi^*(x) \quad \text{a.e. in } (\theta,x)\in [0,\f{\pi}{2}]\times \R^3,
	\end{equation}
	Brezis-Lieb's lemma provides
	\begin{equation}\label{BL}
		\lim_{n\to\infty} (\|V(\theta)\tilde{\psi}_n\|_{L^{2\s+2}}^{2\s+2}-\|V(\theta)(\tilde{\psi}_n-\psi^*)\|_{L^{2\s+2}}^{2\s+2})= \|V(\theta)\psi^*\|_{L^{2\s+2}}^{2\s+2}.
	\end{equation}
	From Lemma \ref{bdnl} and the boundedness of $\{\tilde{\psi}_n\}$ in $B^1$, \eqref{Bre1} is obvious. We now prove \eqref{Bre2}. Let $\vphi\in C_0^\infty(\R)$ be a bump function, such that $\text{supp}\vphi\subset [-\pi, \pi]$ and $\vphi\equiv 1$ on $[0,\f{\pi}{2}]$.
	Then, we calculate
	\begin{equation*}
		\begin{split}
			&\||\partial_\theta|^\h \vphi (\theta) V(\theta)\tilde{\psi}_n\|_{L_{\theta,x}^2(\R^{1+3})}= \int_{\R^3} \int_\R |\om| \big|\cF_{\theta\to \om} \big[\vphi (\theta) V(\theta)\tilde{\psi}_n \big] \big|^2 d\om dx \\
			&\le \Big|\int_{\R^3} \int_{-\infty}^0 \om |\cF_{\theta\to \om} \big[\vphi (\theta) V(\theta)\tilde{\psi}_n \big]|^2 d\om dx \Big| +  \Big|\int_{\R^3} \int_0^{\infty} \om |\cF_{\theta\to \om} \big[\vphi (\theta) V(\theta)\tilde{\psi}_n \big]|^2 d\om dx \Big|.
		\end{split}
	\end{equation*}
	For the first term on the second line, we have
	\begin{equation*}
		\begin{split}
			&\quad \Big|\int_{\R^3} \int_{-\infty}^0 \om \big|\cF_{\theta\to \om} \big[\vphi (\theta) V(\theta)\tilde{\psi}_n \big] \big|^2 d\om dx \Big|\\
			&=\Big| \int_{\R^3} \int_{-\infty}^0 \cF_{\theta\to \om} \big[i\partial_\theta \big(\vphi (\theta) V(\theta)\tilde{\psi}_n \big) \big] \overline{ \cF_{\theta\to \om} \big[ \vphi (\theta) V(\theta)\tilde{\psi}_n \big]} d\om dx \Big|\\
			&=\Big| \int_{\R^3} \int_{-\infty}^0 \cF_{\theta\to \om} \big[ i\vphi' (\theta) V(\theta)\tilde{\psi}_n  +  \vphi (\theta) V(\theta)H\tilde{\psi}_n  \big] \overline{ \cF_{\theta\to \om} \big[ \vphi (\theta) V(\theta)\tilde{\psi}_n \big]} d\om dx \Big|\\
			&\lesssim  \|\vphi' V(\theta)\tilde{\psi}_n\|_{L_{\theta,x}^2(\R^{1+3})} \|\vphi V(\theta)\tilde{\psi}_n\|_{L_{\theta,x}^2(\R^{1+3})} + \|\vphi V(\theta)H^{\h} \tilde{\psi}_n\|_{L_{\theta,x}^2(\R^{1+3})}^2\\
			&\lesssim \|H^\h \tilde{\psi}_n\|_{L^2}^2.
		\end{split}
	\end{equation*}
	Similarly, we also have
	\begin{equation*}
		\Big|\int_{\R^3} \int_{0}^\infty \om \big|\cF_{\theta\to \om} \big[\vphi (\theta) V(\theta)\tilde{\psi}_n \big] \big|^2 d\om dx \Big| 
		\lesssim \|H^\h \tilde{\psi}_n\|_{L^2}^2.
	\end{equation*}
	Thus, we obtain
	\begin{equation*}
		\||\partial_\theta|^{\h}\vphi V(\theta)\tilde{\psi}_n\|_{L_{\theta,x}^{2}(\R^{1+3})} \lesssim \|\psi_n\|_{B^1}.
	\end{equation*}
	Because $\{\psi_n\}_{n\in \N}$ is bounded in $B^1$, according to Rellich–Kondrachov's theorem and the diagonal argument, up to a subsequence,  $\{\vphi V(\theta)\tilde{\psi}_n\}_{n\in \N}$ converges to some $\Psi$ in the $L_{\text{loc}}^{2}([-\pi,\pi]\times \R^3)$ topology, and almost everywhere. By combining this with \eqref{wc}, it holds
	\begin{equation}
		V(\theta)\tilde{\psi}_n \to \Psi(\theta, x)=V(\theta)\psi^* \quad \text{a.e. \quad in\quad} [0,\f{\pi}{2}]\times \R^3 \text{\quad as\quad } n\to \infty.  
	\end{equation}
	Thus, we obtain \eqref{BL}. Moreover, from \eqref{wc}, we obtain
	\begin{align*}
		&\lim_{n\to \infty} K[\tilde{\psi}_n]-K[\tilde{\psi}_n-\psi^*]=K[\psi^*] \\
		&\lim_{n\to \infty} M[\tilde{\psi}_n]-M[\tilde{\psi}_n-\psi^*]=M[\psi^*] \\	
		&\lim_{n\to \infty} \|\partial_z\tilde{\psi}_n\|_{L^2}^2-\|\partial_z(\tilde{\psi}_n-\psi^*)\|_{L^2}^2=\|\partial_z\psi^*\|_{L^2}^2 .
	\end{align*}
	Therefore, 
	\begin{equation*}
		\lim_{n\to\infty} J^{a,b}[\tilde{\psi}_n-\psi^*]= \lim_{n\to\infty} J^{a,b}[\tilde{\psi}_n]-J^{a,b}[\psi^*]= -J^{a,b}[\psi^*]<0.
	\end{equation*}
	This implies $ J^{a,b}[\tilde{\psi}_n-\psi^*]<0$ for all sufficiently large $n$.
	Combining Step 2 with the preceding argument, we see that
	\begin{equation*}
		d^{a,b}<\lim_{n\to\infty} \tB^{a,b}[\tilde{\psi}_n-\psi^*] = \lim_{n\to\infty} \tB^{a,b}[\tilde{\psi}_n] - \tB^{a,b}[\psi^*] = d^{a,b} - \tB^{a,b}[\psi^*] < d^{a,b}.
	\end{equation*} 
	However, this is contradictory. Thus, $J^{a,b}[\psi^*]=0$.\\
	
	\noindent
	\underline{\textbf{Step 5.} $U^{a,b}$ are not empty.} From Step 4, $\psi^* \in \cA^{a,b}$ holds and this implies $d^{a,b}\le \tB^{a,b}[\psi^*]$. Conversely, from \eqref{wc}, we have
	\begin{equation*}
		\tB^{a,b}[\psi^*] \le \liminf_{n\to\infty} \tB^{a,b}[\tilde{\psi}_n] = d^{a,b}.
	\end{equation*}
	Therefore, $\tB^{a,b}[\psi^*]=d^{a,b}$ and $\psi^*\in U^{a,b}$.
\end{proof}
\quad

\begin{remark}\label{Extr}
	From Lemma \ref{Uab} with $(a,b)=(1,0)$, we obtain the ground state solution $Q\in B^1$. This is an extremizer of the Strichartz estimate
\begin{equation*}
	\|V(\theta) \psi\|_{L_{\theta,x}^{2\s+2}([0,\f{\pi}{2}]\times\R^3)} \lesssim (\|\psi\|_{L^2}^2+\|\nabla_x\psi\|_{L^2}^2 + \||y|\psi\|_{L^2}^2)^\h.
\end{equation*}
Indeed, for any $\psi\in B^1\setminus\{0\}$, it hold that $I[\al_{\psi} \psi]=0$ for 
\[\al_{\psi}:=\Big(\f{\|\psi\|_{L^2}^2+\|\nabla_x \psi\|_{L^2}^2 + \||y|\psi\|_{L^2}^2}{\|V(\theta) \psi\|_{L_{\theta,x}^{2\s+2}([0,\f{\pi}{2}]\times\R^3)}^{2\s+2}}\Big)^{\f{1}{2\s}}\in \R_+,\]
and 
\begin{equation*}
	\begin{split}
		S[\al_{\psi} \psi]&=\f{\s\al_{\psi}^2}{2\s+2}(\|\psi\|_{L^2}^2+\|\nabla_x\psi\|_{L^2}^2 + \||y|\psi\|_{L^2}^2)\\
		&=\f{\s}{2\s+2}\Big(\f{(\|\psi\|_{L^2}^2+\|\nabla_x\psi\|_{L^2}^2 + \||y|\psi\|_{L^2}^2)^\h}{\|V(\theta) \psi\|_{L_{\theta,x}^{2\s+2}([0,\f{\pi}{2}]\times\R^3)}}\Big)^{2+\f{2}{\s}}.
	\end{split}
\end{equation*}
Since $Q$ is the minimizer of \eqref{minp} with $(a,b)=(1,0)$, $\al_Q=1$ and $S[Q]\le S[\al_{\psi}\psi]$ hold. Hence we have 
\begin{equation*}
	 \f{\|V(\theta) Q\|_{L_{\theta,x}^{2\s+2}([0,\f{\pi}{2}]\times\R^3)}}{(\|Q\|_{L^2}^2+\|\nabla_xQ\|_{L^2}^2 + \||y|Q\|_{L^2}^2)^\h}=\sup_{\psi\in B^1\setminus\{0\}}\f{\|V(\theta) \psi\|_{L_{\theta,x}^{2\s+2}([0,\f{\pi}{2}]\times\R^3)}}{(\|\psi\|_{L^2}^2+\|\nabla_x\psi\|_{L^2}^2 + \||y|\psi\|_{L^2}^2)^\h}.
\end{equation*}
\end{remark}
\quad
\begin{remark}\label{Fvp}
When $\s>1$, it holds true that
\begin{equation}\label{infME}
	\inf\{M[\psi]+E[\psi]: \psi \in B^1\setminus\{0\} \text{ and } I_z[\psi]=0\}=0,
\end{equation}
where 
\[I_z[\psi]=\|\partial_z\psi\|_{L^2}^2+ \|\psi\|_{L^2}^2-\f{2}{\pi}\|V(\theta)\psi\|_{L_{\theta,z}^{2\s+2}}^{2\s+2}.\]
Therefore, the infimum is not attained.
We will now prove this. For any $\psi\in B^1$, a calculation using the lens transform
\[V(\theta)\psi(y)=\f{1}{\cos{2\theta}}e^{-\f{i}{2}y^2\tan{2\theta}}(e^{i(\f{\tan{2\theta}}{2})\Delta_y}\psi)\big(\f{y}{\cos{2\theta}}\big)\] and the change of variables $\tilde{y}=\f{y}{\cos{2\theta}}$ and $\tilde{\theta}=\f{\tan{2\theta}}{2}$ gives
\begin{equation}\label{lens}
\begin{split}
	&\int_0^{\f{\pi}{2}}\int_{\R^3}|V(\theta)\psi|^{2\s+2} dx d\theta = \int_{-\f{\pi}{4}}^{\f{\pi}{4}}\int_{\R^3}|V(\theta)\psi|^{2\s+2} dx d\theta \\
	&=\int_{\R}(1+4\theta^2)^{\s-1}\int_{\R^3}|e^{i\theta\Delta_y}\psi|^{2\s+2}dxd\theta \ge \|e^{i\theta\Delta_y}\psi\|_{L_{\theta,x}^{2\s+2}(\R\times\R^3)}^{2\s+2}.
\end{split}
\end{equation}
If we consider the scaling
\begin{equation*}
	\psi^r(y,z):=r^{\f{1}{\s}}\psi(ry,z), \quad r>0
\end{equation*}
for any $\psi\in B^1\setminus \{0\}$, it holds that
\begin{equation*}
  \f{	\|V(\theta)\psi^r\|_{L_{\theta,x}^{2\s+2}([0,\f{\pi}{2}]\times\R^3)}}{(\|\partial_z\psi^r\|_{L^2}^2+ \|\psi^r\|_{L^2}^2)^\h}\ge r^{\f{\s-1}{\s+1}}\f{	\|e^{i\theta\Delta_y}\psi\|_{L_{\theta,x}^{2\s+2}(\R\times\R^3)}}{(\|\partial_z\psi\|_{L^2}^2+ \|\psi\|_{L^2}^2)^\h}  \to +\infty\qquad \text{as}\quad r\to +\infty.
\end{equation*}
Let $\psi\in B^1\setminus\{0\}$ satisfy $\|\partial_z\psi\|_{L^2}^2+ \|\psi\|_{L^2}^2 =\f{2}{\pi}\|e^{it\Delta_y}\psi\|_{L_{\theta,x}^{2\s+2}(\R\times\R^3)}^{2\s+2}$ and let $\{r_n\}_{n\in\N}\subset(1, \infty)$ be a sequence such that $r_n\to +\infty$ as $n\to +\infty$. From the second equality in \eqref{lens} we have
\begin{equation*}
	\begin{split}
\f{2}{\pi}\|V(\theta)\psi^{r_n}\|_{L_{\theta,x}^{2\s+2}([0,\f{\pi}{2}]\times\R^3)}^{2\s+2}&=\f{2}{\pi}r_n^{\f{2}{\s}-2} \|(1+\f{4\theta^2}{r_n^4})^{\s-1}e^{i\theta\Delta_y}\psi\|_{L_{\theta,x}^{2\s+2}(\R\times\R^3)}^{2\s+2}\\
	&> \f{2}{\pi}r_n^{\f{2}{\s}-2} \|e^{i\theta\Delta_y}\psi\|_{L_{\theta,x}^{2\s+2}(\R\times\R^3)}^{2\s+2}=\|\partial_z\psi^{r_n}\|_{L^2}^2+ \|\psi^{r_n}\|_{L^2}^2.
\end{split}
\end{equation*}
Therefore, there exists a sequence $\{\al_n\}_{n\in\N}\subset (0,1)$ such that $I[\al_n\psi^{r_n}]=0$ and  we obtain
\begin{equation*}
	S[\al_n\psi^{r_n}]=\f{\al_n^2r_n^{\f{2}{\s}-2} \s}{2\s+2}(\|\partial_z\psi\|_{L^2}^2+\|\psi\|_{L^2}^2)\to 0 \quad \text{as}\quad n\to+\infty.
\end{equation*}
This implies \eqref{infME}.
\end{remark}

\begin{lemma}[\cite{Gdb} Lemma 3.3]\label{3.3}
	Let $0<\s<4$, and let $(a,b)$ satisfy \eqref{asab} and \eqref{asab2}. Then, the constant $d^{a,b}$ is independent of $(a,b)$.
\end{lemma}
In the following, we represent $d$ as $d^{a,b}$. Now, we introduce 
\begin{align*}
	&\cK^{a,b,+}:=\{\psi\in B^1: S[\psi]<d, \quad J^{a,b}[\psi]\ge0\},\\
	&\cK^{a,b,-}:=\{\psi\in B^1: S[\psi]<d, \quad J^{a,b}[\psi]<0\}.
\end{align*}
Recall that $P=J^{1,2}$ and 
\begin{align*}\label{Kpm1}
	\cK^+=\{ \vphi\in B^1 : S[\vphi]<d, \quad P[\vphi]\ge0\},\\
	\cK^-=\{ \vphi\in B^1 : S[\vphi]<d, \quad P[\vphi]<0\}.
\end{align*}
\begin{lemma}[\cite{Gdb} Lemma 3.4]\label{Kpm}
	Let $2\le\s<4$. Then, the sets $\cK^{a,b,\pm}$ are independent of $(a,b)$ satisfying \eqref{asab}. 
\end{lemma}
Hence, when $2<\sigma<4$ and $(a,b)$ satisfies \eqref{asab}, $\cK^{\pm}=\cK^{a,b,\pm}$ holds. To obtain Lemma \ref{Kpm},  we need the following Lemma \ref{3.2}.
In particular, by Lemma \ref{3.2}, we can confirm that the set $\cK^+$ is open. 

\begin{lemma}[\cite{Gdb} Remark 3.2]\label{3.2}
	Let $2<\s<4$, and let $(a,b)$ satisfy \eqref{asab} and $2a-b=0$. If $\psi\in \cA^{a,b}$, then $S[\psi]\ge d$. 
\end{lemma}

\begin{proof}
	We consider the following scaling: 
	\begin{equation*}
		\psi^r(y,z):=r^{a}\psi(y,r^{b}z)=r^{a}\psi(y,r^{2a}z),\qquad r>0.
	\end{equation*}
	Then, 
	\begin{equation}
		I[\psi^r]=r^{4a}\|\partial_z\psi\|_{L^2}^2-r^{2\s a}\f{2}{\pi}\|V(\theta)\psi\|_{L_{\theta,x}^{2\s+2}}^{2\s+2} + 2K[\psi] + 2M[\psi].
	\end{equation}
	Since $J^{a,b}[\psi]=0$ and $2a=b$, we have
	\begin{equation}
		\begin{split}
			\|\partial_z\psi\|_{L^2}^2=\f{\s}{\pi(\s+1)}\|V(\theta)\psi\|_{L_{\theta,x}^{2\s+2}}^{2\s+2}.
		\end{split}
	\end{equation}
	Thus, 
	\begin{equation}
		I[\psi^r]=\f{2}{\pi}\Big(\f{\s}{2(\s+1)}r^{4a}-r^{2\s a}\Big)\|V(\theta)\psi\|_{L_{\theta,x}^{2\s+2}}^{2\s+2} + 2K[\psi] +2 M[\psi].
	\end{equation}
	As $\s>2$, there exists $r_0\in (0,\infty)$ such that $I[\psi^{r_0}]=0$. This implies $S[\psi^{r_0}]\ge d$. Moreover, a simple calculation yields
	\begin{equation*}
		\partial_r S[\psi^r]>0 \text{ for } r<1, \qquad  	\partial_r S[\psi^r]|_{r=1}= J^{a,2a}[\psi]=0, \qquad  \partial_r S[\psi^r]<0 \text{ for } r>1.
	\end{equation*}
	Hence, the function $(0,\infty)\ni r \mapsto S[\psi^r]$ reaches its maximum at $r=1$. Therefore, we obtain $S[\psi]\ge S[\psi^{r_0}]\ge d$.
\end{proof}

In fact, we can extend Lemma \ref{Kpm} to the case $\s=2$ and $2a-b=0$. This is the key proposition for proving Theorem \ref{main} when $\s=2$. 

\begin{prop}\label{Kpm2}
	Let $\s=2$, and let $(a,b)$ satisfy \eqref{asab} except $\s a-b >0$. Then, $\cK^{\pm}=\cK^{a,b,\pm}$ holds true. 
\end{prop}

\begin{proof}
	From Lemmas \ref{3.3} and \ref{Kpm}, for any $(a,b)$ that satisfies
	\begin{equation}\label{asab3}
		a>0, \quad b\ge 0, \quad 2a-b>0,
	\end{equation}
	$d^{a, b}$ and $\cK^{a,b,-}$ are independent of $(a,b)$. As the mapping $(a,b)\mapsto J^{a,b}[\psi]$ is continuous for any fixed $\psi\in B^1$, we have
	
	\begin{equation*}
		\{\psi\in B^1 : S[\psi]<d, \quad P[\psi]>0\} \subset \cK^{a,b,+} ,\quad 	\cK^{-} \subset \cK^{a,b,-}
	\end{equation*}
	for $(a,b)$ satisfying \eqref{asab3}. Since $\cK^+\cup \cK^-=\cK^{a,b,+}\cup \cK^{a,b,-}$, it is sufficient to prove that 
	\begin{equation}\label{bdK}
		\{\psi\in B^1 : S[\psi]<d, \quad P[\psi]=0\} \subset \cK^{a,b,+}
	\end{equation}
	for $(a,b)$ satisfying \eqref{asab3}. As 
	\[ \{\psi\in B^1 : S[\psi]<d, \quad P[\psi]=0\}\subset \cK^{a,b,+}\cup \cK^{a,b,-}   \]
	holds, \eqref{bdK} follows from
	\begin{equation}\label{empK}
		\{\psi\in B^1 : S[\psi]<d, \quad P[\psi]=0\} \cap \cK^{a,b,-}=\emptyset .
	\end{equation}
	We now prove this. We assume that thete exists $\psi \in \cK^{a,b,-}=\cK^{1,0,-}$ that satisfies $P[\psi]=0$. We note that $\psi \neq 0$. Let
	\begin{equation*}
		\psi^r:= r^\h \psi(y,rz), \qquad r>0.
	\end{equation*}
	Then, for any $r>0$,
	\begin{equation}\label{Sr}
		S[\psi^r]= K[\psi] +  M[\psi]+ \f{r^2}{4}P[\psi]= K[\psi] +  M[\psi]= S[\psi]<d.
	\end{equation}
	Since $P[\psi]=0$, it holds that
	\begin{equation*}
		\|\partial_z \psi\|_{L^2}^2= \f{2}{3\pi}\|V(\theta)\psi\|_{L_{\theta,x}^{6}}^6
	\end{equation*}
	and 
	\begin{equation*}
		\begin{split}
			I[\psi^r]= 2K[\psi]+ 2M[\psi] -r^4\f{4}{3\pi}\|V(\theta)\psi\|_{L_{\theta,x}^{6}}^6.
		\end{split}
	\end{equation*}
	Thus, there exists $r_0\in (0,1)$ such that $I[\psi^{r_0}]=0$ and, by the definition of $d$, $S[\psi^{r_0}]\ge d$. This contradicts \eqref{Sr}. Therefore, we obtain \eqref{empK}.	
\end{proof}

\begin{lemma}[\cite{Gdb} Lemma 3.5]\label{3.5}
	Let $\psi\in \cK^+$. Then,
	\begin{equation*}
		\f{\s}{2(\s+1)}\|\psi\|_{B^1}\le S[\psi] \le \h \|\psi\|_{B^1}.
	\end{equation*} 
\end{lemma}

\begin{lemma}[\cite{Gdb} Lemma 3.6]\label{3.6}
	Let $2<\s<4$ and $\psi\in \cK^-$. Then,
	\begin{equation*}
		P[\psi]\le -\f{1}{4}(d-S[\psi])<0.
	\end{equation*} 
\end{lemma}

\section{Sharp threshold for blow up}\label{sharpt}

In this section, we prove Theorem \ref{main} except for the scattering result.

\begin{proof}[Proof of Theorem \ref{main}]
	Let $\phi_0\in \cK^+$. Because $S$ is the conserved quantity of \eqref{NLS} and $d$ is the minimum of \eqref{minp}, the corresponding solution $\phi(t)$ belongs to $\cK^+$ as long as it exists.
	From Lemma \ref{3.5} and the blow-up alternative, $\phi(t)$  exists globally.\\ 
	
	Similarly, when $\phi_0\in \cK^-$, the corresponding solution $\phi(t)$ belongs to $\cK^-$ for every $t\in [0,T^+)$, where $T^+$ represents the maximum positive existence time. \\
	
	We assume that $T^+<\infty$. Then, from Proposition \ref{wp} and the conservation law of $K$, we have $\lim_{t\to  T^+}\|\partial_z \phi(t)\|_{L^2}=\infty$. We prove that if $T^+=\infty$, then there exists a sequence $\{t_k\}_{k=1}^\infty$ such that $\lim_{k\to\infty}t_k=\infty$ and $\lim_{k\to\infty}\|\partial_z \phi(t_k)\|_{L^2}=\infty$. Suppose that 
	\begin{equation}\label{k0}
		k_0:= \sup_{t\in[0,\infty)}\|\partial_z \phi(t)\|_{L^2}<\infty.
	\end{equation}
	We define
	\begin{equation}\label{vir}
		W(t):= \int_{\R^3} \chi(z)|\phi(t)|^2 dx,
	\end{equation}
	where $\chi$ is $C^4(\R)$ and is radial. By direct calculations, we obtain
	\begin{equation}\label{vir2}
		W'(t)= 2\text{Im}\int_{\R^3}\partial_z\phi(t) \overline{\phi(t)} \chi'(z) dx, 
	\end{equation}
	
	\begin{equation}\label{vir3}
		\begin{split}
			W''(t)=& 4 \int_{\R^3}|\partial_z\phi(t)|^2 \chi''(z) dx -\f{4\s}{\pi(\s+1)} \int_0^{\f{\pi}{2}}\int_{\R^3}|V(\theta)\phi(t)|^{2\s+2} \chi''(z) d\theta dx	\\
			&- \int_{\R^3}|\phi(t)|^2 \chi^{(4)}(z) dx.
		\end{split}
	\end{equation}
	
	\begin{lemma}[\cite{Gdb} Lemma 4.1]\label{4.1}
		Let $\eta, R>0$. Then, for all $t\le \eta R/(8k_0\|\phi_0\|_{L^2})$, we have
		\begin{equation}
			\int_{\R^2\times\{|z|\ge R\}}|\phi(t,x)|^2dx \le \eta + o_R(1), \qquad R\to \infty.
		\end{equation}
	\end{lemma}
	
	\quad \\
	
	Next, we choose $\chi$ such that
	\begin{equation}\label{chi1}
		\chi(z)=
		\left\{
		\begin{array}{lll}
			z^2 &&	 0<|z|\le R \\
			0  &&  2R<|z|
		\end{array}
		\right. 
	\end{equation}
	with 
	\begin{equation}\label{chi2}
		0\le \chi \le z^2, \quad \chi''\le 2, \quad \chi^{(4)}\le \f{4}{R}.
	\end{equation}
	Then, \eqref{vir3} can be rewritten as
	\begin{equation}
		\begin{split}
			W''(t)&= 4P[\phi(t)] + 4 \int_{\R^3}|\partial_z\phi(t)|^2 (\chi''(z)-2) dx \\ 
			&\quad -\f{4\s}{\pi(\s+1)}\int_0^{\f{\pi}{2}}\int_{\R^3}|V(\theta)\phi(t)|^{2\s+2}(\chi''(z)-2)d\theta dx	 - \int_{\R^3}|\phi(t)|^2 \chi^{(4)}(z) dx. \\
		\end{split}
	\end{equation}
	If $2<\s<4$, from Lemma \ref{3.6} and the conservation law of $S$, we have for all $t\in [0,T^+)$ that
	\begin{equation*}
		P[\phi(t)]\le- \f{1}{4}(d-S[\phi(t)])= -\f{1}{4}(d-S[\phi_0])=:-C_1<0.
	\end{equation*}
	If $\s=2$, then $P[\phi(t)]=4E[\phi(t)]$. That is, $P$ is the conserved quantity of \eqref{NLS}. Since $\phi(t)\in \cK^-$, we have
	\begin{equation*}
		P[\phi(t)]= P[\phi_0]=:-C_1<0.
	\end{equation*}
	From \eqref{chi2}, for all $t\in [0,T^+)$ we have
	\begin{equation*}
		\begin{split}
			\int_{\R^3}|\partial_z\phi(t)|^2 (\chi''(z)-2) dx \le 0.
		\end{split}
	\end{equation*}
	In addition, from \eqref{chi2}, Lemma \ref{bdnl}, \eqref{k0}, and Lemma \ref{4.1}, we obtain for all $t\in [0,\f{\e R}{8k_0\|\phi_0\|_{L^2}}]$ that
	\begin{equation}
		\begin{split}
			&\int_0^{\f{\pi}{2}}\int_{\R^3}|V(\theta)\phi(t)|^{2\s+2}(\chi''(z)-2)d\theta dx \\
			&\lesssim 	\|V(\theta)\phi(t)\|_{L_{\theta,x}^{4}(\{|z|> R\})}^{\f{4(4-\s)}{3}}
			\|V(\theta)\phi(t)\|_{L_{\theta,x}^{10}(\{|z| > R\})}^{\f{10(\s-1)}{3}}\\
			&\lesssim  \|\phi(t)\|_{L^2(\{|z|> R\})}^{4-\s}\|\partial_z \phi(t)\|_{L^2}^{\s}\|H^\h \phi(t)\|_{L^2}^{2\s-2}\\
			&\lesssim  \|\phi(t)\|_{L^2(\{|z|> R\})}^{4-\s} k_0^\s K[\phi_0]^{2\s-2}\\
			&\le C (\eta^{4-\s}+ o_R(1)), \qquad R\to \infty.
		\end{split}
	\end{equation}
	Where $C$ is a positive constant that depends on $\s$ and $\phi_0$. Finally, from Lemma \ref{4.1}, we obtain 
	\begin{equation*}
		\int_{\R^3}|\phi(t)|^2 \chi^{(4)}(z) dx \lesssim \eta +o_R(1).
	\end{equation*}
	Therefore, for all $t\le \eta R/(8k_0\|\phi_0\|_{L^2})$, we have
	\begin{equation*}
		W''(t)\le -4C_1 +C_2(\eta^{4-\s} +\eta +o_R(1)), \qquad R\to \infty.
	\end{equation*}
	where $C_1$ and $C_2$ are positive constants. 
	Choosing $\eta$ sufficiently small and taking $R$ large enough,  it follows that for all $t\le \eta R/(8 k_0 \|\phi_0\|_{L^2})$ that
	
	\begin{equation*}
		W''(t)\le -2C_1.
	\end{equation*}
	By integrating this inequality from $0$ to $T=\eta R/(8 k_0 \|\phi_0\|_{L^2})$, we infer
	\begin{equation*}
		W(T)\le W(0)+W'(0)T -C_1 T^2.
	\end{equation*}
	Based on the argument in \cite[Section 4]{Gdb}, we obtain $W(0)\le o_R(1)R^2$ and $W'(0)\le o_R(1)R$. Hence, we obtain
	\begin{equation*}
		W(T)\le(o_R(1)-C_3 \eta^2)R^2,	
	\end{equation*}
	where $C_3$ is a positive constant that depends on $\s$ and $\phi_0$. Taking $R$ large so that $o_R(1)-C_3 \eta^2<0$, we have $W(T)<0$, which contradicts the fact that $W(T)$ is positive. Therefore, we have the conclusion. \\
	
	Finally, we consider the case $\phi_0\in \cK^-\cap \Sigma^1$. Then, according to the local theory and the property of $\cK^-$, the corresponding solution $\phi(t)$ belongs to $\cK^-\cap \Sigma^1$ for any $t\in [0,T^+)$. Suppose that $T^+=\infty$. We set $\chi(z)=z^2$:
	\begin{equation}\label{W2}
		W(t)=\|z\phi(t)\|_{L_x^2}^2>0.
	\end{equation}
	Therefore, $W(t)$ exists globally. From \eqref{vir3}, we have that 
	\begin{equation*}
		W''(t)=4P[\phi(t)]=-4C_1<0,
	\end{equation*}
	and
	\begin{equation*}
		W(t)\le W(0)+ W'(0)t-2C_1t^2.
	\end{equation*}
	Then, there exists $T>0$ such that $W(T)=0$. This contradicts \eqref{W2} and we have $T^+<\infty$.
\end{proof}

\section{Key propositions for the scattering result}\label{keyprop}

In this section, we fix the indices $p$, $q$, $s$, $p_0$ as
\begin{equation}\label{pqr}
	p=\f{2\s}{\s-1}, \quad q=2\s, \quad s=\f{3(\s-2)}{2(\s-1)}, \quad p_0=(\f{1}{p}-\f{s}{3})^{-1}=2\s(\s-1).
\end{equation} 
As mentioned in Remark \ref{indad}, $(p,2q)$ and $(p,q)$ satisfy 1-dimensional and 2-dimensional admissible conditions, respectively. Using Sobolev's inequality, \eqref{eqH2}, Lemmas \ref{St1} and \ref{St2}, and Minkowski's inequality, we have
\begin{equation}\label{UVB}
	\begin{split}
		&\|U(t)V(\theta)\phi_0\|_{L_t^{2q} L_\theta^q L_x^{p_0}(\R\times [0,\f{\pi}{2}]\times\R^3)}\\
		&\lesssim  \|H^{\f{s}{2}}U(t)V(\theta)\phi_0\|_{L_t^{2q} L_\theta^q L_x^{p}(\R\times [0,\f{\pi}{2}]\times\R^3)} +   \||\partial_z|^{s}U(t)V(\theta)\phi_0\|_{L_t^{2q} L_\theta^q L_x^{p}(\R\times [0,\f{\pi}{2}]\times\R^3)}\\
		&\lesssim \|\phi_0\|_{B^s}.
	\end{split}
\end{equation}
In the following we frequently use this type of inequality.\\

Before proving the scattering part of Theorem \ref{main}, we present some important propositions.

\subsection{Scattering condition and wave operator} 

\begin{lemma}\label{Ssc}
	Let $2\le \s <4$, $\lmd=\pm 1$, and $\phi$ be the global solution to \eqref{NLS} with $\phi(0)=\phi_0\in B^1$. If $\phi$ satisfies $\|\phi\|_{L_t^\infty B^1(\R\times \R^3)}<\infty$ and
	\begin{equation}\label{ScC}
		\|V(\theta)\phi\|_{L_t^{2q}L_\theta^qL_x^{p_0}(\R\times [0,\f{\pi}{2}]\times \R^3)}<\infty,
	\end{equation}
	then $\phi$ scatters in $B^1$ both forward and backward. Moreover, there exists $\delta>0$ such that if $\|\phi_0\|_{B^1} \le \delta$, the corresponding solution $\phi$ scatters.
\end{lemma}

\begin{proof}
	By the same argument used to justify \eqref{uS1} with reference to  \eqref{Ltp2} in Proposition \ref{Ltp}, \eqref{ScC} yields
	\[ \|\phi\|_{ L_t^\infty B^1\cap L_t^{4} \Sigma_0^{\infty, 2}(\R\times \R^3)}<\infty, \]
	where,  
	\[ \| u \|_{\Sigma_0^{p, 2}} : =  \| \nabla_x u\|_{L_z^{p}L_y^{2}(\R^3)} + \| \langle y \rangle u \|_{L_z^{p} L_y^{2}(\R^3)}, \quad 2< p\le \infty. \]
	Then, by Lemma \ref{St4}, it holds that
	\begin{equation*}
		\begin{split}
			\Big\| \int_t^\infty U(t-\tau)\Fav(\phi(\tau))d\tau \Big\|_{B^1} \lesssim&  \|\phi\|_{ L_t^\infty B^1\cap L_t^{4} \Sigma_0^{\infty, 2}}^{3}  \|V(\theta) \phi\|_{L_t^{2q} L_\theta^q L_x^{p_0}([t, \infty]\times [0,\f{\pi}{2}]\times \R^3)}^{2\s-2} \to 0 		
		\end{split}
	\end{equation*}
	as $t\to\infty$. Therefore, 
	\begin{equation*}
		\phi_+=\phi_0 -i\lmd\int_0^\infty U(-\tau)\Fav(\phi(\tau))d\tau
	\end{equation*}
	exhibits the desired properties. Similarly, we obtain $\phi_- \in B^1$.\\

	Let $\phi_0\in B^1$. We define the function space as
	\begin{equation*}
		\begin{split}
			X(\phi_0):=\{ \phi\in L_t^\infty B^1\cap L_t^{4} \Sigma_0^{\infty, 2}(\R \times \R^3) :\|\phi&\|_{ L_t^\infty B^1\cap L_t^{4} \Sigma_0^{\infty, 2}(\R\times \R^3)} \le 2 C_1\|\phi_0\|_{B^1},\\
			\|U(t)\phi\|_{ L_t^{2q} \Sigma_0^{p,2} \cap L_t^{4} \Sigma_0^{\infty, 2}(\R\times \R^3)} &\le 2 \|U(t)\phi_0\|_{ L_t^{2q} \Sigma_0^{p,2} \cap L_t^{4} \Sigma_0^{\infty, 2}(\R\times \R^3)} \},
		\end{split}
	\end{equation*}
	where $C_1$ is a constant such that it holds for any $\psi\in B^1$ that
	\begin{equation*}
		\|U(t)\psi\|_{ L_t^\infty B^1\cap L_t^{4} \Sigma_0^{\infty, 2}(\R\times \R^3)} \le C_1\|\psi\|_{B^1}.
	\end{equation*}
	We define the mapping
	\begin{equation}\label{Phi}
		\Phi[\phi](t):=U(t)\phi_0 -i\lmd\int_0^t U(t-\tau)\Fav(\phi(\tau))d\tau.
	\end{equation}
	By Lemma \ref{St4}, Minkowski's inequality, and Lemma \ref{St2}, we have
	\begin{equation}\label{Phi1}
		\begin{split}
			\|\Phi[\phi]\|_{ L_t^\infty B^1\cap L_t^{4} \Sigma_0^{\infty, 2}} &\le C_1 \|\phi_0\|_{B^1} + C \|\phi\|_{ L_t^{4} \Sigma_0^{\infty, 2}}^{2}  \| \phi\|_{L_t^{2q} \Sigma_0^{p,2}}^{2\s-1}\\
			&\le  C_1 \|\phi_0\|_{B^1} + C_2 \|U(t)\phi_0\|_{L_t^{2q} \Sigma_0^{p, 2}  \cap L_t^{4} \Sigma_0^{\infty, 2}}^{2\s+1}.
		\end{split}
	\end{equation}
	Particularly, in the second line, we use Sobolev's embedding, \eqref{eqH2}, Minkowski's inequality, and Lemma \ref{St2};
	\begin{equation}\label{ScC2}
		\begin{split}
			\|V(\theta) \phi\|_{L_t^{2q} L_\theta^q L_x^{p_0}} &\lesssim  \|H^{\f{1}{2}}V(\theta) \phi\|_{L_t^{2q} L_\theta^q L_x^{p}} +\|\partial_z V(\theta) \phi\|_{L_t^{2q} L_\theta^q L_x^{p}} \\
			&\lesssim \|H^{\f{1}{2}} \phi\|_{L_t^{2q} L_z^{p} L_y^2} +\|\partial_z \phi\|_{L_t^{2q} L_z^{p} L_y^2}\simeq\|\phi\|_{ L_t^{2q} \Sigma_0^{p, 2}}.
		\end{split}
	\end{equation}
	Similarly, we also have 
	\begin{equation}\label{Phi2}
		\begin{split}
			\|\Phi[\phi]\|_{L_t^{2q} \Sigma_0^{p, 2}  \cap L_t^{4} \Sigma_0^{\infty, 2}} &\le  \|U(t)\phi_0\|_{L_t^{2q} \Sigma_0^{p, 2}  \cap L_t^{4} \Sigma_0^{\infty, 2}}  + C_2 \|U(t)\phi_0\|_{L_t^{2q} \Sigma_0^{p, 2}  \cap L_t^{4} \Sigma_0^{\infty, 2}}^{2\s+1}.
		\end{split}
	\end{equation}
	Because we have
	\[ \|U(t)\phi_0\|_{L_t^{2q} \Sigma_0^{p, 2}  \cap L_t^{4} \Sigma_0^{\infty, 2}} \le C_1 \|\phi_0\|_{B^1} ,\]
	if we set $\delta>0$ to be sufficiently small, and $\|\phi_0\|_{B^1}\le \delta$ holds, 
	$\Phi$ is the mapping on $X(\phi_0)$. The contraction property, uniqueness, and continuity statements are easy consequences. From \eqref{ScC2} and the definition of $X(\phi_0)$, we obtain the
	space-time bound
	\begin{equation}
		\|V(\theta) \phi\|_{L_t^{2q} L_\theta^q L_x^{p_0}}  \lesssim \|\phi\|_{ L_t^\infty B^1\cap L_t^{4} \Sigma_0^{\infty, 2}} \lesssim  \|\phi_0\|_{B^1}<\infty.
	\end{equation}
	Therefore, $\phi$ scatters.
\end{proof}

\begin{lemma}\label{waveo}
	Let $2\le \s <4$, $\lmd=-1$, and $\psi\in B^1$ satisfy
	\begin{equation}
		\h\|\psi\|_{B^1}^2 +\h\|\psi\|_{L^2}^2 <d,
	\end{equation}
	where $d$ is given by \eqref{minp} and Lemma \ref{3.3}. Then, there exists a global solution $\phi(t)\in \cK^+ $to \eqref{NLS}, such that
	\begin{equation}\label{wave}
		\lim_{t\to\infty}\|U(-t)\phi(t)-\psi\|_{B^1}=0.
	\end{equation}
\end{lemma}

\begin{proof}
	Let $T>0$. We define the function space
	\begin{equation*}
		\begin{split}
			X(\psi,T):=\{ \phi\in& L_t^\infty B^1\cap L_t^{4} \Sigma_0^{\infty, 2}([T, \infty)\times \R^3) : \|\phi\|_{ L_t^\infty B^1\cap L_t^{4} \Sigma_0^{\infty, 2}([T, \infty)\times \R^3)} \le 2 C_1\|\psi\|_{B^1},\\
			&\|V(\theta)\phi\|_{L_t^{2q}L_{\theta}^{q}L_x^{p_0}([T,\infty)\times [0,\f{\pi}{2}]\times \R^3)}\le \|U(t)V(\theta)\psi\|_{L_t^{2q}L_{\theta}^{q}L_x^{p_0}([T,\infty)\times [0,\f{\pi}{2}]\times \R^3)}\},
		\end{split}
	\end{equation*}
	and the mapping \eqref{Phi}. Using the same estimates as \eqref{Phi1} and \eqref{Phi2}, and choosing $T$ sufficiently large depending on $\psi$, we can construct the solution $\phi$ to \eqref{NLS}, satisfying \eqref{wave} on $X(\psi, T)$.\\
	
	Next, we prove
	\begin{equation}\label{pelimit}
		\lim_{t\to \infty}\|V(\theta)U(t)\psi\|_{L_{\theta,x}^{2\s+2}}=0.
	\end{equation}
	For any $\ep>0$, there exists $f_\ep\in C_0^\infty(\R^3)$ such that $\|f_\ep-\psi\|_{B^1}\le \ep$. Then, by applying the triangle inequality and Lemma \ref{bdnl}, 
	\begin{equation*}
		\begin{split}
			\|V(\theta)U(t)\psi\|_{L_{\theta,x}^{2\s+2}} &\lesssim \|f_\ep-\psi\|_{B^1} + \|V(\theta)U(t)f_\ep \|_{L_{\theta,x}^{2\s+2}}\\
			&\le \ep  + \|V(\theta)U(t)f_\ep \|_{L_{\theta,x}^{2\s+2}}.
		\end{split}
	\end{equation*}
	We use interpolation and equations \eqref{ieqnl-0}–\eqref{ieqnl-3} to obtain
	\begin{equation*}
		\begin{split}
			\|V(\theta)U(t)f_\ep \|_{L_{\theta,x}^{2\s+2}}&\lesssim	\|V(\theta)U(t)f_\ep \|_{L_{\theta,x}^{4}}^{\f{2(4-\s)}{3(\s+1)}}	\|V(\theta)U(t)f_\ep \|_{L_{\theta,x}^{10}}^{\f{5(\s-1)}{3(\s+1)}}\\
			&\lesssim \| f_\ep\|_{B^1}^{\f{2}{\s+1}} \|V(\theta)U(t)f_\ep\|_{L_\theta^\infty L_x^6}^{\f{\s-1}{\s+1}}.
		\end{split}
	\end{equation*}
	From the dispersive estimate in $z$, Minkowski, and Sobolev's embedding in $y$, we have
	\begin{equation*}
		\begin{split}
			\|U(t)V(\theta)f_\ep\|_{L_\theta^\infty L_x^6} 
			&\lesssim |t|^{-\h} \|V(\theta)f_\ep\|_{L_\theta^\infty L_y^6 L_z^{\f{6}{5}}}\\
			&\lesssim     |t|^{-\h} \|H^\h V(\theta)f_\ep\|_{L_\theta^\infty  L_z^{\f{6}{5}}L_y^2}\\ 
			&\lesssim     |t|^{-\h} \|H^\h f_\ep\|_{  L_z^{\f{6}{5}}L_y^2} \to 0 \quad \text{as}\quad t\to \infty.
		\end{split}
	\end{equation*}
	Therefore, $\limsup_{t\to\infty}	\|V(\theta)U(t)\psi\|_{L_{\theta,x}^{2\s+2}}\le \ep$ holds true. By letting $\ep\to+0$, we obtain \eqref{pelimit}.\\
	From \eqref{pelimit}, it holds
	\begin{equation*}
		\lim_{t\to\infty}S[\phi(t)]= 	\lim_{t\to\infty}S[U(t)\psi]= \h\|\psi\|_{B^1}^2 +\h\|\psi\|_{L^2}^2 <d,
	\end{equation*}
	and $\phi(t)\in \cK^+$ for a large $t$. Hence, $\phi$ always belongs to $\cK^+$  and exists globally.
\end{proof}

\subsection{Linear profile decomposition}

In this subsection, we prove the linear profile decomposition. The reader can refer to \cite{Gdb, ECQHO}.

\begin{lemma}[Inverse Strichartz estimate]\label{ISt}
	Let $2<\s<4$, and $\{\phi_n\}_{n=1}^\infty \subset B^1$ satisfy
	\[ 0< \delta\le \|U(t)V(\theta)\phi_n\|_{L_t^{2q}L_\theta^q L_x^{p_0}(\R\times [0,\f{\pi}{2}]\times \R^3)} \lesssim \|\phi_n\|_{B^1}\le A<\infty \] 
	for some $\delta, A>0$. Then, up to a subsequence, there exists $\psi^*\in B^1$ and $\{(t_n, z_n)\}_{n=1}^{\infty}\in \R\times \R$ such that the following statements hold:
	
	\begin{equation}\label{IStwc}
		U(-t_n)\phi_n(\cdot, \cdot+z_n) \rightharpoonup \psi^* \quad \text{weakly in} \quad B^1.
	\end{equation}
	The sequence $\{ \psi_n^*:=U(t_n)\psi^*(\cdot, \cdot-z_n)\}_{n=1}^\infty$ satisfies
	\begin{equation}\label{ISt1}
		\|\psi_n^*\|_{B^1} =  \|\psi^*\|_{B^1} \gtrsim \delta^{\f{1}{1-\al}}A^{-\f{\al}{1-\al}}, \qquad \al=\max\{\f{1}{\s-1}, \f{3}{2\s} \}.
	\end{equation}
	\begin{equation}\label{ISt2}
		\lim_{n\to\infty} \big[ \|\phi_n\|_{B^1}^2-\|\phi_n-\psi^*_n\|_{B^1}^2-\|\psi^*_n\|_{B^1}^2\big]=0.
	\end{equation}
	\begin{equation}\label{ISt3}
		\lim_{n\to \infty} \big[\|V(\theta)\phi_n\|_{L_{\theta,x}^{2\s+2}}^{2\s+2}-\|V(\theta)(\phi_n-\psi^*_n)\|_{L_{\theta,x}^{2\s+2}}^{2\s+2}-\|V(\theta)\psi^*_n\|_{L_{\theta,x}^{2\s+2}}^{2\s+2}\big]=0.
	\end{equation}
	Moreover, $\{t_n\}_{n=1}^\infty$ satisfies $t_n \equiv 0$ or $t_n\to \pm \infty$ as $n\to \infty$.
\end{lemma}

\begin{proof}
	Let $N\in 2^{\N}$ and $\vphi\in C_0^\infty(\R^d)$ be radially symmetric functions supported in $\{\xi\in \R^3; |\xi|\le2\}$ and $\vphi(\xi)=1$ for $|\xi|\le 1$. We define the operator
	\[ P_{\le N} = \vphi\Big(\f{H}{N^2}\Big)\vphi\Big(\f{-\partial_z^2}{N^2}\Big). \]
	$\vphi\Big(\f{H}{N^2}\Big)$ is the spectral cut-off with respect to the harmonic oscillator $H$, and $\vphi\Big(\f{-\partial_z^2}{N^2}\Big)$ is the Fourier cut-off in $z$. Note that $P_{\le N}$ commutes with both $e^{-itH}$ and $e^{it\partial_z^2}$. \\
	
	Generally, for any $\psi\in B^1$ and $s'\in [0,1)$, we have from \eqref{eqD},
	\begin{equation*}
		\|\psi-P_{\le N}\psi\|_{B^{s'}}\lesssim N^{-(1-s')}\|\psi\|_{B^1}.
	\end{equation*} 
	From \eqref{UVB} and the above estimate, we obtain
	\begin{equation*}
		\|U(t)V(\theta)(1-P_{\le N})\phi_n )\|_{L_t^{2q}L_\theta^qL_x^{p_0}} \lesssim N^{-(1-s)}\|\phi_n\|_{B^1} \le  N^{-(1-s)}A.
	\end{equation*}
	Thus, if $N\gtrsim (A/\delta)^{\f{1}{1-s}}$, then we have
	\begin{equation*}
		\begin{split}
			\|U(t)V(\theta)P_{\le N}\phi_n \|_{L_t^{2q}L_\theta^qL_x^{p_0}}&\ge 	  \|U(t)V(\theta)\phi_n \|_{L_t^{2q}L_\theta^qL_x^{p_0}} -\|U(t)V(\theta)(1-P_{\le N})\phi_n )\|_{L_t^{2q}L_\theta^qL_x^{p_0}}\\
			&\gtrsim \delta. 
		\end{split}
	\end{equation*}
	By interpolation, for $\al=\max\{\f{1}{\s-1}, \f{3}{2\s} \}$, we obtain
	\begin{equation}
		\begin{split}
			\delta&\lesssim \|U(t)V(\theta)P_{\le N}\phi_n \|_{L_t^{2q}L_\theta^qL_x^{p_0}} \\
			&\lesssim  \|U(t)V(\theta)P_{\le N}\phi_n \|_{L_t^{\infty}L_\theta^\infty L_x^{\infty}}^{1-\al}   \|U(t)V(\theta)P_{\le N}\phi_n \|_{L_t^{2\al q}L_\theta^{\al q}L_x^{\al p_0}}^{\al} \\
			&\lesssim \|U(t)V(\theta)P_{\le N}\phi_n \|_{L_t^{\infty}L_\theta^\infty L_x^{\infty}}^{1-\al} A^{\al}.
		\end{split}
	\end{equation}
	In the last line, we divide into the cases $2<\s\le 3$ and $3<\s<4$. When $2<\s\le 3$, $p_0 \le 2q$ holds. Thus, by Minkowski's inequality and Lemmas \ref{St1} and \ref{St2}, we have
	\[  \|U(t)V(\theta)P_{\le N}\phi_n \|_{L_t^{2\al q}L_\theta^{\al q}L_x^{\al p_0}} \lesssim \|U(t)V(\theta)P_{\le N}\phi_n \|_{L_\theta^{\al q}L_y^{\al p_0} L_z^2} \lesssim \|\phi_n\|_{L^2}.   \]
	When $3<\s<4$, after applying Sobolev's embedding and \eqref{eqH2}, we perform the same estimate as above.
	\begin{equation*}
		\begin{split}
			&\|U(t)V(\theta)P_{\le N}\phi_n \|_{L_t^{2\al q}L_\theta^{\al q}L_x^{\al p_0}}\\ &\lesssim   \|H^{\f{\s-3}{4(\s-1)}}U(t)V(\theta)P_{\le N}\phi_n \|_{L_t^{2\al q}L_\theta^{\al q}L_x^{2\al q}}+ \||\partial_z|^{\f{\s-3}{2(\s-1)}}U(t)V(\theta)P_{\le N}\phi_n \|_{L_t^{2\al q}L_\theta^{\al q}L_x^{2\al q}} \\
			&\lesssim \|\phi_n\|_{B^{\f{\s-3}{2(\s-1)}}}.
		\end{split}
	\end{equation*} 
	Therefore, if $N\gtrsim (A/\delta)^{\f{1}{1-s}}$, then we obtain
	\begin{equation*}
		\|U(t)V(\theta)P_{\le N}\phi_n \|_{L_t^{\infty}L_\theta^\infty L_x^{\infty}} \gtrsim \delta^{\f{1}{1-\al}}A^{-\f{\al}{1-\al}}.
	\end{equation*}
	Applying \cite[Lemmas 3.1 and 3.2]{Random} to the $y$-direction, there exists $c>0$ independent of $\phi_n$ and $t$ such that for all $x\in \R^3$, it holds that
	\begin{equation}
		|P_{\le N}U(t)V(\theta)\phi_n(x)|\lesssim N e^{-c\f{|y|^2}{N^2}}\|P_{\le N}U(t)\phi_n(\cdot,z)\|_{L_y^2}.
	\end{equation}
	Using Bernstein's estimate in $z$, we obtain
	\begin{equation}
		\|P_{\le N}U(t)\phi_n\|_{L_y^2L_z^\infty} \lesssim N^\h	\|P_{\le N}U(t)\phi_n\|_{L_x^2}\le N^\h \|\phi_n\|_{L^2},
	\end{equation}
	and so, 
	\begin{equation}\label{smooth}
		|P_{\le N}U(t)V(\theta)\phi_n(x)|\lesssim N^{\f{3}{2}} e^{-c\f{|y|^2}{N^2}}\|\phi_n\|_{L^2}\le  N^{\f{3}{2}} e^{-c\f{|y|^2}{N^2}}A.
	\end{equation}
	If $R$ is sufficiently large such that 
	\begin{equation}\label{R}
		N^{\f{3}{2}}e^{-c\f{R^2}{N^2}}\le \min\Big\{\f{1}{C}\delta^{\f{1}{1-\al}}A^{-\f{\al}{1-\al}-1},1\Big\},
	\end{equation}	
	($R$ depends only on $\sigma$, $\delta$, and $A$) we obtain
	\begin{equation*}
		\begin{split}
			&\|P_{\le N}U(t)V(\theta)\phi_n \|_{L_{t, \theta, x}^{\infty}(\{|y|\le R\})}\\
			\ge& 	  \|P_{\le N}U(t)V(\theta)\phi_n \|_{L_{t, \theta, x}^{\infty}} -\|P_{\le N}U(t)V(\theta)\phi_n )\|_{L_{t, \theta, x}^{\infty}(\{|y|\ge R\})}\\
			\gtrsim &\delta^{\f{1}{1-\al}}A^{-\f{\al}{1-\al}}- N^{\f{3}{2}} e^{-c\f{R^2}{N^2}}A\\
			\gtrsim &\delta^{\f{1}{1-\al}}A^{-\f{\al}{1-\al}}. 
		\end{split}
	\end{equation*}
	Then, there exist $t_n\in \R$, $\theta_n\in [0,\f{\pi}{2}]$, $y_n\in\{|y|\le R\}$, and $z_n\in \R$, such that
	\begin{equation}\label{nonzero}
		|P_{\le N}U(-t_n)V(\theta_n)\phi_n(y_n, z_n)|\gtrsim \delta^{\f{1}{1-\al}}A^{-\f{\al}{1-\al}}.
	\end{equation}
	Since $|y_n|\le R$ and $\theta_n\in [0,\f{\pi}{2}]$, after passing to a subsequence, $y_n \to \exists y_\infty$  and $\theta_n \to \exists \theta_\infty$ as $n\to \infty$. In addition, for $\{t_n\}_{n=1}^\infty$, we reduce this to the cases $\sup_{n\in \N}| t_n|<\infty$ or $\lim_{n\to \infty} |t_n|=\infty$. \\
	
	First, we consider the case  $\lim_{n\to \infty} |t_n|=\infty$. Let 
	\[ w_n(x):=U(-t_n)V(\theta_n)\phi_n(y, z+z_n). \]
	Then, $\{w_n\}_{n=1}^\infty$ is bounded in $B^1$, and after passing to a subsequence, there exists $\psi \in B^1$ such that $w_n \rightharpoonup \psi$ in $B^1$ as $n\to \infty$ and $\|\psi\|_{B^1}\le A$. Because 
	\[ P_{\le N}w_n\in B^2 \hookrightarrow C\cap L^\infty (\R^3), \]
	Rellich–Kondrachov's theorem and \eqref{nonzero} show that:
	\begin{equation*}
		|P_{\le N}\psi(y_\infty, 0)|\gtrsim \delta^{\f{1}{1-\al}}A^{-\f{\al}{1-\al}}.
	\end{equation*}
	Using the same estimate as in \eqref{smooth} and \eqref{R}, we obtain
	\[ \|\psi\|_{L^2}\ge \|P_{\le N}\psi\|_{L^2} \gtrsim |P_{\le N}\psi(y_\infty, 0)|\gtrsim \delta^{\f{1}{1-\al}}A^{-\f{\al}{1-\al}}>0 .\]
	Let $\psi*:=V(-\theta_\infty)\psi$, $\psi_n^*:=U(t_n)\psi^*(\cdot, \cdot-z_n)$, and  $r_n(x):=\phi_n(x)- \psi_n^*$, it holds
	\begin{equation*}
		U(-t_n)r_n(\cdot, \cdot+z_n) \rightharpoonup 0 \quad \text{ weakly in } \quad B^1,  
	\end{equation*}
	\begin{equation*}
		\begin{split}
			\|\phi_n\|_{B^1}^2-\|\psi_n^*\|_{B^1}^2-\|r_n\|_{B^1}^2 \to 0, \quad \text{as}\quad  n\to\infty.
		\end{split}	
	\end{equation*}
	Moreover, we apply the same argument as \eqref{pelimit} to  $\{\psi_n^*\}_{n=1}^\infty$ and obtain
	\[\lim_{n\to\infty}	\|V(\theta)\psi^*_n\|_{L_{\theta,x}^{2\s+2}}=0. \]
	By combining this with the boundedness of $\phi_n$ and $\psi^*_n$ in $B^1$ with respect to $n$, we obtain \eqref{ISt3}.\\
	
	Next, we consider the case $\{t_n\}_{n=1}^\infty$ is bounded. Then, by passing to a subsequence, $t_n \to \exists t_\infty$ as $n\to \infty$. By iterating the argument in the case $\lim_n |t_n|=\infty$, we obtain $\psi^*\in B^1$ such that
	\[ w_n= U(-t_n)V(\theta_n)\phi_n(\cdot, \cdot+z_n) \rightharpoonup \psi^* \quad \text{ weakly in} \quad B^1.\]
	Then it holds 
	\[  \tilde{w}_n:=\phi_n(\cdot, \cdot+z_n) \rightharpoonup \tilde{\psi^*}:= U(t_\infty)V(-\theta_\infty)\psi^* \quad \text{ weakly in} \quad B^1, \]
	and we may set $t_n\equiv 0$. Finally, we prove \eqref{ISt3} holds. It is sufficient to show
	\begin{equation*}
		\lim_{n\to\infty}\|V(\theta)\tilde{w}_n\|_{L_{\theta,x}^{2\s+2}}^{2\s+2}-\|V(\theta)(\tilde{w}_n-\tilde{\psi^*})\|_{L_{\theta,x}^{2\s+2}}^{2\s+2}-\|V(\theta)\tilde{\psi^*}\|_{L_{\theta,x}^{2\s+2}}^{2\s+2}=0.
	\end{equation*}
	This follows from Brezis-Lieb's lemma. Refer to the proof of \eqref{BL}.
\end{proof}
\quad

From Lemma \ref{ISt}, we obtain the following profile decomposition:

\begin{prop}[Linear profile decomposition]\label{Lpd}
	Let $2<\s<4$ and $\{\phi_n\}_{n=1}^\infty$ be a uniformly bounded sequence in $B^1$. Then, up to a subsequence, there exist $J^*\in \{0,1,2,\cdots\}\cup \{\infty \}$, $\{\psi^j\}_{j=1}^{J^*} \subset B^1$, and  $\{\{(t_n^j, z_n^j)\}_{n=1}^\infty\}_{j=1}^{J^*}\subset \R^2$ such that the following decomposition holds for any finite $J\le J^*$;
	\begin{equation}
		\phi_n(x)=\sum_{j=1}^J U(t_n^j)\psi^j(y,z-z_n^j) +r_n^J(x).
	\end{equation}
	Moreover, we obtain the following properties as $n\to\infty$ for each $J$:
	\begin{equation}\label{LPDwc}
		U(-t_n^J)r_n^J(z+z_n^J) \rightharpoonup 0 \quad \text{ weakly in} \quad B^1,
	\end{equation}
	\begin{equation}\label{LPDo1}
		\|\phi_n\|_{L^2}^2 = \sum_{j=1}^J\|U(t_n^j)\psi^j\|_{L^2}^2 +\|r_n^J\|_{L^2}^2 +o_n(1),
	\end{equation}
	
	\begin{equation}\label{LPDo2}
		\|\phi_n\|_{B^1}^2 = \sum_{j=1}^J\|U(t_n^j)\psi^j\|_{B^1}^2 +\|r_n^J\|_{B^1}^2 +o_n(1),
	\end{equation}
	
	\begin{equation}\label{LPDo3}
		\|V(\theta)\phi_n\|_{L_{\theta,x}^{2\s+2}}^{2\s+2} = \sum_{j=1}^J\|V(\theta)U(t_n^j)\psi^j\|_{L_{\theta,x}^{2\s+2}}^{2\s+2} +\|V(\theta)r_n^J\|_{L_{\theta,x}^{2\s+2}}^{2\s+2} +o_n(1).
	\end{equation}
	In particular, for any finite $J\le J^*$, it hold that
	\begin{equation}
		S[\phi_n] = \sum_{j=1}^JS[U(t_n^j)\psi^j] +S[r_n^J] +o_n(1),
	\end{equation}
	\begin{equation}
		I[\phi_n] = \sum_{j=1}^JI[U(t_n^j)\psi^j] +I[r_n^J] +o_n(1).
	\end{equation}
	Finally, we have
	\begin{equation}\label{LPDo4}
		\lim_{n\to\infty}|t_n^j-t_n^k|+|z_n^j-z_n^k|=\infty, \quad \text{for any} \quad j\neq k,
	\end{equation}
	\begin{equation}\label{LPDrem}
		\lim_{J\to J^*} \limsup_{n\to \infty}\|U(t)V(\theta)r_n^J\|_{L_t^{2q}L_\theta^{q}L_x^{p_0}(\R\times [0,\f{\pi}{2}]\times \R^3)} =0.
	\end{equation}

\end{prop}

\begin{proof}
	We will proceed with the proof inductively by using Lemma \ref{ISt}. Let $r_n^0=\phi_n$. We assume that we have a decomposition up to level $J \ge 1$ satisfying the properties \eqref{LPDo1} through \eqref{LPDo3}. After passing to a subsequence, we define
	\begin{equation*}
		\delta_J=\lim_{n\to \infty} \|U(t)V(\theta)r_n^J\|_{L_t^{2q}L_\theta^{q}L_x^{p_0}}, \qquad A_J= \lim_{n\to \infty} \|r_n^J\|_{B^1}.
	\end{equation*} 
	If $\delta_J=0$, we stop and set $J^*=J$. Otherwise, by applying Lemma \ref{ISt} to $\{r_n^J\}_{n=1}^\infty$, we obtain $\psi^{J+1}\in B^1$ and $\{(t_n^{J+1}, z_n^{J+1})\}_{n=1}^\infty \subset \R^2$ satisfying the conclusions of Lemma \ref{ISt}. Let $r_n^{J+1}:=r_n^J-\psi_n^{J+1}$, where $\psi_n^{J+1}:=U(t_n^{J+1})\psi^{J+1}(\cdot, \cdot-z_n^{J+1})$. Then, from \eqref{IStwc}, we obtain \eqref{LPDwc} for $J+1$. In addition, based on the inductive assumptions, \eqref{LPDo1}–\eqref{LPDo3} hold for $J+1$. 
	After passing to a subsequence, we define
	\begin{equation*}
		\delta_{J+1}=\lim_{n\to \infty} \|U(t)V(\theta)r_n^{J+1}\|_{L_t^{2q}L_\theta^{q}L_x^{p_0}}, \qquad A_{J+1}= \lim_{n\to \infty} \|r_n^{J+1}\|_{B^1}.
	\end{equation*} 
	If $\delta_{J+1}=0$, we stop and set $J^*=J+1$. Otherwise, the induction will continue. If the algorithm never terminates, we set $J^*=\infty$. From \eqref{ISt1} and \eqref{ISt2}, it holds that
	\begin{equation}
		A_J^2-A_{J+1}^2=\lim_{n\to\infty}(\|r_n^J\|_{B^1}^2-\|r_n^{J+1}\|_{B^1}^2)=\|\psi^{J+1}\|_{B^1}^2\gtrsim \delta_{J}^{\f{2}{1-\al}}A_{J}^{-\f{2\al}{1-\al}}=\Big(\f{\delta_J}{A_J}\Big)^\f{2}{1-\al} A_J^2,
	\end{equation}
	and we have
	\begin{equation*}
		A_J^2\Big(1-c \Big(\f{\delta_J}{A_J}\Big)^\f{2}{1-\al}\Big) \ge A_{J+1}^2.
	\end{equation*}
	Suppose that $\limsup_{J\to J^*}\delta_J=\delta_\infty>0$. Because $\{A_J\}_J$ is a decreasing sequence, for a sufficiently large $J$, we have
	\begin{equation*}
		A_J^2\Big(1-c \Big(\f{\delta_\infty}{2A_0}\Big)^\f{2}{1-\al}\Big) \ge A_{J+1}^2.
	\end{equation*}
	This implies that $\lim_{J\to J^*}A_J=0$. On the other hand, from \eqref{UVB}, we obtain $\delta_J\lesssim A_J$ uniformly, and  we have $\lim_{J\to J^*}A_J\gtrsim \delta_\infty$. However, this is contradictory. Therefore, we obtain \eqref{LPDrem}.\\

	We prove \eqref{LPDo4} holds. If not, there exist $0 \le j_1<j_2$ such that  
	\begin{equation}\label{j12}
		t_n^{j_1}-t_n^{j_2}\to \exists t^* \quad \text{and}\quad 	z_n^{j_1}-z_n^{j_2}\to \exists z^*, \quad \text{as} \quad n\to\infty.
	\end{equation}
	We may assume that for any $j_1<k<j_2$, $\{(t_n^{j_1}, z_n^{j_1})\}_n$ and $\{(t_n^{k}, z_n^{k})\}_n$ are orthogonal. Based on the construction of the profile, we have
	\begin{equation*}
		r_n^{j_1-1}= U(t_n^{j_1})\psi^{j_1}(\cdot-z_n^{j_1})  +\sum_{j_1<k<j_2} U(t_n^{k})\psi^{k}(\cdot-z_n^{k}) + U(t_n^{j_2})\psi^{j_2}(\cdot-z_n^{j_2}) + r_n^{j_2}.
	\end{equation*} 
	Therefore, we also have
	\begin{equation}\label{rj12}
		\begin{split}
			U(-t_n^{j_1})r_n^{j_1-1}(\cdot+z_n^{j_1})&- \psi^{j_1} =\sum_{j_1<k<j_2} U(t_n^{k}-t_n^{j_1})\psi^{k}(\cdot-z_n^{k}+z_n^{j_1})\\
			&\quad + U(t_n^{j_2}-t_n^{j_1})\psi^{j_2}(\cdot-z_n^{j_2}+z_n^{j_1}) + U(-t_n^{j_1})r_n^{j_2}(\cdot+z_n^{j_1}).	
		\end{split}
	\end{equation} 
	As $n\to \infty$, the left side converges to $0$ weakly in $B^1$. On the right side, since  $\{(t_n^{j_1}, z_n^{j_1})\}_{n=1}^\infty$ and $\{(t_n^{k}, z_n^{k})\}_{n=1}^\infty$ are orthogonal, the terms in the summation converge to $0$ weakly in $B^1$. Additionally, according to \eqref{j12}, we have
	\begin{equation*}
		U(t_n^{j_2}-t_n^{j_1})\psi^{j_2}(\cdot-z_n^{j_2}+z_n^{j_1}) \to    U(-t^*)\psi^{j_2}(\cdot+z^*)\quad \text{strongly in}\quad B^1.
	\end{equation*}
	Moreover, from \eqref{j12} and \eqref{LPDwc} for $j_2$, the last term on the right side of \eqref{rj12} converges to $0$ weakly in $B^1$. Therefore, by letting $n\to\infty$ in \eqref{rj12}, we obtain $0=U(-t^*)\psi^{j_2}(\cdot+z^*)$, which contradicts 
	\[ \|\psi^{j_2}\|_{B^1}\gtrsim(\delta_{j_2-1})^{\f{1}{1-\al}}A_{j_2-1}^{-\f{\al}{1-\al}}>0.\]
	Hence, \eqref{LPDo4} holds.
\end{proof}

\subsection{Perturbation lemma}

In this subsection, we prove the long-time perturbation lemma. We refer to \cite[Section 3]{NLScri}, but due to the anisotropy of \eqref{NLS}, the averaged nonlinearity, and the condition $\s\ge2$, we need some modifications. \\

we introduce the following notations :

\begin{align*}
	\|\psi\|_{S(I)}&:= 	\|\psi\|_{L_t^\infty L_x^2\cap L_t^4 L_z^\infty L_y^2 (I)}=\|\psi\|_{L_t^\infty L_x^2(I)} + 	\|\psi\|_{L_t^4 L_z^\infty L_y^2(I)}\\
	\|\psi\|_{S^1(I)}&:= \|\psi\|_{S(I)}+ \|\nabla_x \psi\|_{S(I)} +  \|y \psi\|_{S(I)}\\
	\|\psi\|_{N(I)}&:=\inf\{\|f_1\|_{L_t^1 L_x^2(I)}+ \|f_2\|_{L_t^\f{4}{3}L_z^1L_y^2(I)} : \psi=f_1+f_2\}\\
	\|\psi\|_{N^1(I)}&:=  \|\psi\|_{N(I)} +\|\nabla_x \psi\|_{N(I)} +  \|y \psi\|_{N(I)}.	
\end{align*}

We also define 

\begin{equation}\label{X0}
	\begin{split}
		\|\psi\|_{X^0(I)}&:=\|  \psi\|_{L_t^{2q^*} L_\theta^{q^*} L_x^{p_0^*}(I)} ,	
	\end{split}
\end{equation}

\begin{equation}\label{Xstar}
	\begin{split}
		\|\psi\|_{X^*(I)}&:=\| H^\f{s}{2} \psi\|_{L_t^{2q^*}L_\theta^{q^*}L_x^{p^*}(I)} + \| |\partial_z|^s \psi\|_{L_t^{2q^*}L_\theta^{q^*}L_x^{p^*}(I)} \\
		&\simeq\| \langle \nabla_x \rangle^s \psi\|_{L_t^{2q^*}L_\theta^{q^*}L_x^{p^*}(I)} + \| \langle y \rangle^s \psi\|_{L_t^{2q^*}L_\theta^{q^*}L_x^{p^*}(I)},	
	\end{split}
\end{equation}

\begin{equation}\label{Ystar}
	\begin{split}
		\|\psi\|_{Y^*(I)}&:=\| H^\f{s}{2}\psi\|_{L_t^{q_t^*}L_\theta^{q_\theta^*}L_x^{(p^*)'}(I)} +\||\partial_z|^s \psi\|_{L_t^{q_t^*}L_\theta^{q_\theta^*}L_x^{(p^*)'}(I)}\\
		&\simeq \| \langle \nabla\rangle^s\psi\|_{L_t^{q_t^*}L_\theta^{q_\theta^*}L_x^{(p^*)'}(I)} +\|\langle y\rangle^s \psi\|_{L_t^{q_t^*}L_\theta^{q_\theta^*}L_x^{(p^*)'}(I)},
	\end{split}
\end{equation}
where 
\begin{equation}
	p^*=\f{3\s-2}{\s-1}, \quad q^*=3\s-2, \quad s^*=\f{3(2\s-1)(\s-2)}{2(3\s-2)(\s-1)},
\end{equation}
\begin{equation}
	p_0^*=\Big(\f{1}{p^*}-\f{s^*}{3}\Big)^{-1}= \f{2(\s-1)(3\s-2)}{\s},
\end{equation}
and 
\begin{equation}
	q_t^*=\Big(\f{2\s-1}{2q}+\h\Big)^{-1}, \quad q_\theta^*=\Big(\f{2\s-1}{q}\Big)^{-1}
\end{equation}
The norm equivalences in \eqref{Xstar} and \eqref{Ystar} follow from \eqref{eqH2}.

\begin{lemma}\label{Xint} We have
	\begin{equation}\label{Xint0}
		\|f\|_{X^0(I)} \lesssim \|f\|_{X^*(I)} 
	\end{equation}
	\begin{equation}\label{Xint1}
		\begin{split}
			\|V(\theta)f\|_{X^*(I)} \lesssim \|f\|_{S^1(I)}^{\f{2\s-1}{3\s-2}}\|V(\theta)f\|_{L_t^{2q}L_\theta^{q}L_x^{p_0}(I)}^{\f{\s-1}{3\s-2}},
		\end{split}
	\end{equation}
	
	\begin{equation}\label{Xint2}
		\|V(\theta)f\|_{L_t^{2q}L_\theta^{q}L_x^{p_0}(I)}\lesssim  \|V(\theta)f\|_{X^0(I)}^\h \|f\|_{S^1(I)}^\h.
	\end{equation}
\end{lemma}

\begin{proof}
	By Sobolev's inequality and Gagliardo-Nirenberg's inequality, we obtain
	\begin{equation}\label{GNX}
		\begin{split}
			\|f\|_{L_x^{p_0^*}} \lesssim \||\nabla_x|^{s^*} f\|_{L_x^{p^*}}&\lesssim \||\nabla_x|^s f\|_{L_x^{\f{2\s(2\s-1)}{2\s(\s-1)-1}}}^{\f{2\s-1}{3\s-2}} \|f\|_{L_x^{p_0}}^{\f{\s-1}{3\s-2}}.
		\end{split}
	\end{equation}
	\eqref{Xint0} follows from the first inequality in \eqref{GNX} and \eqref{eqH2}. Furthermore, from the second inequality in \eqref{GNX}, the interpolation
	\[  \||y|^{s^*} f\|_{L_x^{p^*}}\le \||y|^s f\|_{L_x^{\f{2\s(2\s-1)}{2\s(\s-1)-1}}}^{\f{2\s-1}{3\s-2}} \|f\|_{L_x^{p_0}}^{\f{\s-1}{3\s-2}}, 
	\]
	and H\"{o}lder's inequality in $\theta$ and $t$, we obtain
	\begin{equation*}
		\begin{split}
			\|V(\theta) f\|_{X^*}
			&\lesssim (\||\nabla_x|^s V(\theta) f\|_{L_t^{\f{4\s(2\s-1)}{\s+1}}L_{\theta}^{\f{2\s(2\s-1)}{\s+1}}L_x^{\f{2\s(2\s-1)}{2\s(\s-1)-1}}}\\
			&\qquad\qquad +\||y|^sV(\theta) f\|_{L_t^{\f{4\s(2\s-1)}{\s+1}}L_{\theta}^{\f{2\s(2\s-1)}{\s+1}}L_x^{\f{2\s(2\s-1)}{2\s(\s-1)-1}}})^{\f{2\s-1}{3\s-2}} \|f\|_{L_t^{2q} L_{\theta}^{q} L_x^{p_0}}^{\f{\s-1}{3\s-2}}.
		\end{split}
	\end{equation*}
	From \eqref{eqH2}, Minkowski's inequality, and Lemma \ref{St2}, we obtain
	\eqref{Xint1}. To obtain \eqref{Xint2}, we use
	\begin{equation*}
		\|f\|_{L_x^{p_0}}\lesssim \|f\|_{L_x^{p_0^*}}^\h \||\nabla_x|^sf\|_{L_x^{\f{2\s(\s-1)(3\s-2)}{3\s^3-9\s^2+10\s-4}}}^\h.
	\end{equation*}
	This follows from Gagliardo-Nirenberg's inequality.
\end{proof}

\begin{lemma}\label{XYstar}
	Let $F(z):=\lmd|z|^{2\s}z$ ($z\in \C$). Then, we have
	\begin{equation}\label{XYstar1}
		\Big\|V(\theta)\int_{0}^t U(t-\tilde{t}) \Fav(f(\tilde{t})) d \tilde{t}\Big\|_{X^*(I)}\lesssim \big\|F( V(\theta)f)\big\|_{Y^*(I)},
	\end{equation}
	\begin{equation}\label{XYstar2}
		\|F (V(\theta)f)\|_{Y^*(I)}\lesssim \|V(\theta)f\|_{X^*(I)} \|H^\h f\|_{L_t^4L_z^\infty L_x^2(I)}^2 \|V(\theta)f\|_{X^0(I)}^{2\s-2},
	\end{equation}
	\begin{equation}\label{XYstar3}
		\begin{split}
			\Big\|F(V(\theta)f)- F(V(\theta)g)\Big\|_{Y^*(I)}
			\lesssim & \|V(\theta)\big(f-g \big)\|_{X^*(I)}\\
			&\times(\|H^\h f\|_{L_t^4L_z^\infty L_x^2(I)}^2+  \|H^\h g\|_{L_t^4L_z^\infty L_x^2(I)}^2)\\
			&\times( \|V(\theta)f\|_{X^*(I)}^{2\s-2}+  \|V(\theta)g\|_{X^*(I)}^{2\s-2}).
		\end{split}	
	\end{equation}
\end{lemma}

\begin{proof}
	We see that $p^*$, $q^*$, $s^*$ satisfy \eqref{Exi}. Thus, \eqref{XYstar1} and \eqref{XYstar2} follow Lemma \ref{ExSt} and Proposition \ref{ExSt2} respectively.
	We prove \eqref{XYstar3} holds.  For $z,w \in C$, it holds that
	\begin{equation*}
		\begin{split}
			F(z)-F(w)=(z-w)\int_{0}^1 F_z(w+\al (z-w))d\al+ \overline{(z-w)}\int_{0}^1 F_{\bar{z}}(w+\al (z-w))d\al.
		\end{split}
	\end{equation*}
	Then, using similar calculations to the proof of  Proposition \ref{ExSt2}, the fractional chain rule (cf.~\cite[Lemma A.10 and A.11]{NLScri}), and \eqref{Xint0}, we obtain
	\begin{equation*}
		\begin{split}
			\Big\|&F(V(\theta)f)- F(V(\theta)g)\Big\|_{Y^*(I)}\\
			\lesssim & \|V(\theta)\big(f-g \big)\|_{X^*(I)}(\|H^\h f\|_{L_t^4L_z^\infty L_x^2(I)}^2 \|V(\theta)f\|_{X^0(I)}^{2\s-2}  +  \|H^\h g\|_{L_t^4L_z^\infty L_x^2(I)}^2\|V(\theta)g\|_{X^0(I)}^{2\s-2})\\
			& + \|V(\theta)\big(f-g \big)\|_{X^0(I)}\times\big( \|V(\theta)f\|_{X^*(I)}\|H^\h f\|_{L_t^4L_z^\infty L_x^2(I)}^2\|V(\theta)f\|_{X^0(I)}^{2\s-3} \\
			&\qquad \qquad \qquad \qquad \qquad \qquad +\|V(\theta)g\|_{X^*(I)}\|H^\h g\|_{L_t^4L_z^\infty L_x^2(I)}^2 \|V(\theta)g\|_{X^0(I)}^{2\s-3}\big)\\
			\lesssim & \|V(\theta)\big(f-g \big)\|_{X^*}\big(\|H^\h f\|_{L_t^4L_z^\infty L_x^2}^2+  \|H^\h g\|_{L_t^4L_z^\infty L_x^2}^2\big) \big( \|V(\theta)f\|_{X^*}^{2\s-2}+  \|V(\theta)g\|_{X^*}^{2\s-2}\big).
		\end{split}
	\end{equation*}
	
\end{proof}
\begin{lemma}[Short-time perturbation]\label{Stp}
	Let $I\subset \R$, $t_0\in I$, and $u\in C(I,B^1)$ be a solution to 
	\begin{equation}\label{NLSap}
		i\partial_t u=-\partial_z^2u+\lmd \Fav(u)+e,
	\end{equation}
	for some function $e$. Let $\phi\in C(I, B^1)$ be a unique solution to \eqref{NLS} with $\phi|_{t=t_0}=\phi(t_0) \in B^1$. Assume that $u$ and $\phi$ satisfy
	\begin{equation}\label{Stp1}
		\|u\|_{L_t^\infty B^1(I)}\le M, \quad \|\phi\|_{L_t^\infty B^1(I)}\le M'
	\end{equation}
	for some positive constants $M$ and $M'$. We also assume that there exist $\delta=\delta(M, M')>0$ and $\ep_0=\ep_0(M,M')$, and $u$ and $\phi$ satisfy the smallness conditions
	\begin{equation}\label{Stp2}
		\|V(\theta)u\|_{X^*(I)} \le \delta
	\end{equation}
	\begin{equation}\label{Stp3}
		\|U(t-t_0)V(\theta)(\phi(t_0)-u(t_0))\|_{X^*(I)}\le \ep
	\end{equation}
	\begin{equation}\label{Stp4}
		\|e\|_{N^1(I)}\le \ep.
	\end{equation}
	hold for some $\ep\in (0,\ep_0]$. Then, the following estimates hold,
	\begin{equation}\label{Stp5}
		\|V(\theta)(\phi-u)\|_{X^*(I)} \lesssim \ep
	\end{equation}
	\begin{equation}\label{Stp6}
		\|\phi-u\|_{S^1(I)} \lesssim M+M'
	\end{equation}
	\begin{equation}\label{Stp7}
		\|\phi\|_{S^1(I)}\lesssim M'
	\end{equation}
	\begin{equation}\label{Stp8}
		\|F(V(\theta)\phi)-F(V(\theta)u)\|_{Y^*(I)} \lesssim \ep
	\end{equation}
	\begin{equation}\label{Stp9}
		\|\Fav(\phi)-\Fav(u)\|_{N^1(I)}\lesssim M+M'.
	\end{equation}
\end{lemma}

\begin{proof}
	We begin by deriving the bounds on $u$ and $\phi$: From Lemmas \ref{St3}, \ref{St4}, and \ref{Xint}, we have
	\begin{equation}\label{BootSu1}
		\begin{split}
			\|u\|_{S^1(I)} &\lesssim \|u\|_{L_t^\infty B^1(I)} + \|\Fav(u)\|_{N^1(I)} +\|e\|_{N^1(I)} \\
			&\lesssim M + \|u\|_{S^1(I)}\|H^\h u\|_{L_t^4 L_z^\infty L_y^2(I)}^2\|V(\theta)u\|_{L_t^{2q}L_\theta^{q}L_x^{p_0}(I)}^{2\s-2} +\ep\\
			&\lesssim M + \delta^{\s-1}\|u\|_{S^1(I)}^{\s+2}+\ep.
		\end{split}
	\end{equation}
	By choosing $\delta$ and $\ep_0$ to be small, depending on $M$, the bootstrap argument yields
	\begin{equation}
		\|u\|_{S^1(I)}\lesssim M.
	\end{equation}
	Moreover, choosing $\delta$ is small, depending on $M$, and using Lemmas \ref{XYstar} and \ref{St2}, we obtain
	\begin{equation}\label{BootSu2}
		\begin{split}
			\|U(t-t_0)V(\theta)u(t_0)\|_{X^*(I)} &\lesssim \|V(\theta)u\|_{X^*(I)} + \|H^\h u\|_{L_t^4 L_z^\infty L_y^2(I)}^2 \|V(\theta)u\|_{X^*(I)}^{2\s-1} + \|e\|_{N^1(I)} \\
			&\lesssim \delta  + M^{2} \delta^{2\s-1}+\ep \lesssim \delta.
		\end{split}
	\end{equation}
	From this estimate, \eqref{Stp3}, and the triangle inequality, we have
	\begin{equation}
		\|U(t-t_0)V(\theta)\phi(t_0)\|_{X^*(I)}\le C_0 \delta	
	\end{equation}
	for some $C_0>0$. Let $C_1>0$ be a positive constant that satisfies
	\begin{equation}\label{S1B1}
		\|U(t)V(\theta)\psi\|_{S^1(\R)}\le C_1 \|\psi\|_{B^1}
	\end{equation}
	for any $\psi\in B^1$. Then, there exists a small finite interval $J\subset I$ such that $t_0\in J$ and $\phi$ satisfies
	\begin{equation*}
		\|\phi\|_{S^1(J)} \le 2C_1 M',\quad \|V(\theta)\phi\|_{X^*(J)}\le 2 C_0\delta.
	\end{equation*}
	Similar to \eqref{BootSu1} and \eqref{BootSu2}, we obtain
	\begin{equation}\label{BootS1}
		\begin{split}
			\|\phi\|_{S^1(J)} &\le  C_1 \|\phi(t_0)\|_{B^1(J)} + C\|\Fav(\phi)\|_{N^1(J)} \\
			&\le C_1 M' + C\|\phi\|_{S^1(J)}\|H^\h \phi\|_{L_t^4 L_z^\infty L_y^2(J)}^2\|V(\theta)\phi\|_{L_t^{2q}L_\theta^{q}L_x^{p_0}(J)}^{2\s-2}\\
			&\le  C_1 M' + C\|\phi\|_{S^1(J)}^{\s+2}\|V(\theta)\phi\|_{X^*(J)}^{\s-1}\\
			&\le  C_1 M' + C_2( 2C_1 M')^{\s+2}(2C_0\delta)^{\s-1}.
		\end{split}
	\end{equation}
	and 
	\begin{equation}\label{BootS2}
		\begin{split}
			\|V(\theta)\phi\|_{X^*(J)} &\le C_0\delta + C\|F(\phi)\|_{Y^*(J)} \\
			&\le C_0\delta + C\|H^\h \phi\|_{L_t^4 L_z^\infty L_y^2(J)}^2\|V(\theta)\phi\|_{X^*(J)}^{2\s-1}\\
			&\le C_0\delta + C\|\phi\|_{S^1(J)}^{2}\|V(\theta)\phi\|_{X^*(J)}^{2\s-1}\\
			&\le C_0\delta + C_2( 2C_1 M')^{2}(2C_0\delta)^{2\s-1}.
		\end{split}
	\end{equation}
	If $\delta$ is sufficiently small, depending on $C_0$, $C_1$, $C_2$, $M$, and $M'$, the bootstrap argument yields:
	\begin{equation*}
		\|\phi\|_{S^1(I)} \le 2C_1 M' ,\quad \|V(\theta)\phi\|_{X^*(I)}\le 2 C_0\delta.
	\end{equation*}
	In particular, the first estimate implies that \eqref{Stp7} holds. Next, we derive the claimed bounds for $w:=\phi-u$. Here, $w$ is the solution to
	\begin{equation}
		i\partial_t w=-\partial_z^2 w + \lmd\Fav(w+u)-\lmd\Fav(u)-e, \qquad w(t_0)=\phi(t_0)-u(t_0).
	\end{equation}
	From Lemma \ref{XYstar}, we obtain
	\begin{equation}
		\begin{split}
			&\|V(\theta)w\|_{X^*(I)}\\
			&\lesssim \|U(t-t_0)V(\theta)w(t_0)\|_{X^*(I)}+ \| F(V(\theta)(w+u))-F(V(\theta)u)\|_{Y^*(I)}+ \|e\|_{N^1(I)}	 \\
			&\lesssim \ep + \|F(V(\theta)(w+u))-F(V(\theta)u)\|_{Y^*(I)}.
		\end{split}
	\end{equation}
	We use Lemma \ref{XYstar} again to obtain
	\begin{equation}\label{StFFY}
		\begin{split}
			&\|F(V(\theta)(w+u))-F(V(\theta)u)\|_{Y^*(I)} \\
			&\lesssim\|V(\theta)w\|_{X^*(I)}(\|\phi\|_{S^1(I)}^2+\|u\|_{S^1(I)}^2)(\|V(\theta)\phi\|_{X^*(I)}^{2\s-2}+\|V(\theta)u\|_{X^*(I)}^{2\s-2})\\
			&\lesssim \|V(\theta)w\|_{X^*(I)}(M^2 + (M')^2)\delta^{2\s-2}.
		\end{split}
	\end{equation}
	Therefore, we have 
	\begin{equation}
		\begin{split}
			\|V(\theta)w\|_{X^*(I)}&\lesssim \ep + \|V(\theta)w\|_{X^*(I)}(M^2+ (M')^2)\delta^{2\s-2}.
		\end{split}
	\end{equation} 
	If $\delta$ is sufficiently small, depending on $M$ and $M'$, \eqref{Stp5} holds. By applying \eqref{Stp5} to \eqref{StFFY}, we obtain \eqref{Stp8}. On the other hand, we have
	\begin{equation}\label{StFFN0}
		\begin{split}
			\|w\|_{S^1(I)}&\lesssim \|\phi(t_0)-u(t_0)\|_{B^1}+ \|\Fav(w+u)-\Fav(u)\|_{N^1(I)} + \|e\|_{N^1(I)}\\
			&\lesssim M+ M' + \ep + \|\Fav(w+u)-\Fav(u)\|_{N^1(I)}
		\end{split}
	\end{equation} 
	and 
	\begin{equation}\label{StFFN}
		\begin{split}
			&\|\Fav(w+u)-\Fav(u)\|_{N^1(I)}\\
			&\lesssim \|w\|_{S^1(I)}(\|\phi\|_{S^1(I)}^2+\|u\|_{S^1(I)}^2)(\|V(\theta)\phi\|_{L_t^{2q}L_\theta^{q}L_x^{p_0}(I)}^{2\s-2}+\|V(\theta)u\|_{L_t^{2q}L_\theta^{q}L_x^{p_0}(I)}^{2\s-2})\\
			&\lesssim \|w\|_{S^1(I)}(M+(M')^2)(\delta^{\s-1}(M')^{\s-1}+\delta^{\s-1}M^{\s-1}).
		\end{split}
	\end{equation}
	Combining \eqref{StFFN0} and \eqref{StFFN}, and choosing $\delta$ small depending on $M$ and $M'$, we have \eqref{Stp6}.
	Finally, applying \eqref{Stp6} to \eqref{StFFN}, \eqref{Stp9} holds.
\end{proof}

\begin{prop}[Long-time perturbation]\label{Ltp}
	Let $I\subset \R$, $t_0\in I$, and $u\in C(I,B^1)$ be a solution to 
	\begin{equation}\label{NLSap2}
		i\partial_t u=-\partial_z^2u+\lmd \Fav(u)+e,
	\end{equation}
	for some function $e$. Assume that 
	\begin{equation}\label{Ltp1}
		\|u\|_{L_t^\infty B^1(I)}\le M,
	\end{equation}
	\begin{equation}\label{Ltp2}
		\|V(\theta)u\|_{L_t^{2q}L_\theta^{q}L_x^{p_0}(I)}\le L,
	\end{equation}
	for some positive constants $M$ and $L$, and there exists a unique solution $\phi\in C(I,B^1)$ to \eqref{NLS} with $\phi|_{t=t_0}=\phi(t_0)$ that satisfies:
	\begin{equation}\label{Ltp3}
		\|\phi\|_{L_t^\infty B^1(I)}\le M'
	\end{equation}
	for a positive constant $M'$. Assume also that there exists $\ep_1=\ep_1(M, L, M')$ such that the smallness conditions
	
	\begin{equation}\label{Ltp4}
		\|U(t)V(\theta)(\phi(t_0)-u(t_0))\|_{L_t^{2q}L_\theta^{q}L_x^{p_0}(I)}\le \ep
	\end{equation}
	\begin{equation}\label{Ltp5}
		\|e\|_{N^1(I)}\le \ep
	\end{equation}
	holds for some $\ep\in (0,\ep_1]$. Then the solution $\phi$ to \eqref{NLS} satisfies:
	\begin{equation}\label{Ltp6}
		\|V(\theta)(\phi-u)\|_{L_t^{2q}L_\theta^{q}L_x^{p_0}(I)} \le C(M,L,M') \ep^c
	\end{equation}
	\begin{equation}\label{Ltp7}
		\|\phi-u\|_{S^1(I)} \le C(M,L,M'),
	\end{equation}
	\begin{equation}\label{Ltp8}
		\|\phi\|_{S^1(I)} \le C(M,L,M'),
	\end{equation}
	where $0<c<1$.
\end{prop}

\begin{proof}
	Without loss of generality, we assume that $I=[0,\infty)$ and $t_0=0$.
	We first prove 
	\begin{equation}\label{uS1}
		\|u\|_{S^1(I)}\le C(M,L).
	\end{equation}
	From \eqref{Ltp2}, for any small $\eta>0$, we can divide the interval $I$ into $J(\eta,L)$ subintervals $I_j=[t_{j-1}, t_{j})$ ($t_J=\infty$),
	such that
	\begin{equation}
		\|V(\theta)u\|_{L_t^{2q}L_\theta^{q}L_x^{p_0}(I_j)}\le \eta.	
	\end{equation}
	From Lemma \ref{St4}, we have
	\begin{equation}\label{BootSL}
		\begin{split}
			\|u\|_{S^1(I_j)} &\lesssim \|u(t_j)\|_{B^1(I_j)} + \|\Fav(u)\|_{N^1(I_j)} +\|e\|_{N^1(I_j)} \\
			&\lesssim M + \|u\|_{S^1(I_j)}\|H^\h u\|_{L_t^4 L_z^\infty L_y^2(I_j)}^2\|V(\theta)u\|_{L_t^{2q}L_\theta^{q}L_x^{p_0}(I_j)}^{2\s-2} +\ep\\
			&\lesssim M + \eta^{2\s-2}\|u\|_{S^1(I_j)}^{3}+\ep.
		\end{split}
	\end{equation}
	By choosing $\eta>0$ and $\ep_1>0$ depending on $M$, the bootstrap argument yields
	\begin{equation}
		\|u\|_{S^1(I_j)}\lesssim M.
	\end{equation} 
	Summing these over all $I_j$ yields \eqref{uS1}. Using \eqref{Ltp2},  \eqref{uS1}, and \eqref{Xint1}, we have
	\begin{equation}\label{LuX}
		\|V(\theta)u\|_{X^*(I)}\le C(M,L).
	\end{equation}
	From \eqref{Xint1}, \eqref{S1B1}, \eqref{Ltp1}, \eqref{Ltp3}, \eqref{Ltp4}, we obtain that
	\begin{equation}\label{LXdif}
		\|U(t)V(\theta)(\phi(0)-u(0))\|_{X^*(I)}\le C_3 \ep^{a} (M+M')^{b},
	\end{equation}
	where, $a=\f{\s-1}{3\s-2}$, $b=\f{2\s-1}{3\s-2}$, and $C_3>0$. By \eqref{LuX}, we can divide $I$ into $J_1(M,L)$ subintervals $I_j=[t_{j-1}, t_j)$ such that for any $1\le j\le J_1$, it holds that
	\begin{equation}\label{LuXdel}
		\|V(\theta)u\|_{X^*(I_j)}\le \delta,
	\end{equation}
	where $\delta=\delta(M, M')$ is the constant that appears in the assumptions of Lemma \ref{Stp}. Additionally, for $\ep_0(M,M')$ given in Lemma \ref{Stp}, we set $\ep_1=\ep_1(M, L, M')$ such that
	\[ C_3\ep_1^a(M+M')^b\le \ep_0(M, M').\]
	Then, Lemma \ref{Stp} yields:
	\begin{equation}\label{Stp5-1}
		\|V(\theta)(\phi-u)\|_{X^*(I_1)} \le C(1) \ep^a (M+M')^b
	\end{equation}
	\begin{equation}\label{Stp6-1}
		\|\phi-u\|_{S^1(I_1)} \le C(1) (M+M)'
	\end{equation}
	\begin{equation}\label{Stp7-1}
		\|\phi\|_{S^1(I_1)}\le C(1) M'
	\end{equation}
	\begin{equation}\label{Stp8-1}
		\|F(V(\theta)\phi)- F(V(\theta)u))\|_{Y^*(I_1)} \le C(1)\ep^a  (M+M')^b
	\end{equation}
	\begin{equation}\label{Stp9-1}
		\|\Fav(\phi)-\Fav(u)\|_{N^1(I_1)}\le C(1)(M+M').
	\end{equation}
	Using Duhamel's formula, for any $t\in I_2$, we obtain
	\begin{equation}
		\begin{split}
			&U(t-t_1)(\phi(t_1)-u(t_1))\\
			&= U(t)(\phi(0)-u(0)) -i\lmd \int_{0}^{t_1} U(t-\tilde{t})\Big[\Fav(\phi(\tilde{t}))-\Fav(u(\tilde{t}))-e\Big]d\tilde{t}.
		\end{split}
	\end{equation} 
	To estimate the Duhamel term, we use the following lemma.
	\begin{lemma}\label{ExD}
		Let $G$ and $\Gav$ be the same as those in Lemma \ref{ExSt}. Then, we have
		\begin{equation*}
			\Big\| V(\theta)\int_{0}^{t_{j}} U(t-\tilde{t})\Gav(\tilde{t})d\tilde{t} \Big\|_{X^*(I_{j+1})} \lesssim \|G\|_{Y^*([0, t_j])},
		\end{equation*}
		\begin{equation*}
			\Big\| V(\theta)\int_{0}^{t_{j}} U(t-\tilde{t})\Gav(\tilde{t})d\tilde{t} \Big\|_{S^1(I_{j+1})} \lesssim \|G\|_{N^1([0,t_j])}.
		\end{equation*}
	\end{lemma}
	
	\begin{proof} Trace the proof of Lemmas \ref{St1} and \ref{ExSt}.
	\end{proof}
	By using Lemmas \ref{ExD}, \ref{XYstar}, \ref{St2}, and \eqref{LXdif}, we have
	\begin{equation}\label{LtpasX1}
		\begin{split}
			&\|U(t-t_1)V(\theta)(\phi(t_1)-u(t_1))\|_{X^*(I_2)}\\
			&\le \|U(t)V(\theta)(\phi(0)-u(0))\|_{X^*(I)} +C\|F(V(\theta)\phi)-F(V(\theta)u)\|_{Y^*(I_1)}+ C\|e\|_{N^1(I)} \\
			&\le C_3\ep^a (M+M')^b + CC(1)\ep^a  (M+M')^b +C\ep\\
			&\lesssim \ep^a (M+M')^b.
		\end{split}
	\end{equation}
	By choosing $\ep_1$ small depending on $M$ and $M'$ again, we can apply Lemma \ref{Stp} to $I_2$. This procedure is iterated. In other words, similar to \eqref{LtpasX1}, for each $j$, we obtain
	\begin{equation}
		\begin{split}
			&\|U(t-t_j)V(\theta)(\phi(t_j)-u(t_j))\|_{X^*(I_{j+1})}\\
			&\le \|U(t)V(\theta)(\phi(0)-u(0))\|_{X^*(I)} +C\| F(\phi)- F (u)\|_{Y([0,t_{j}])}+ C\|e\|_{N^1(I)} \\
			&\le C_3 \ep^a  (M+M')^b + C \sum_{k=1}^{j}C(k)\ep^a (M+M')^b  +C\ep\\
			&\lesssim \ep^a (M+M')^b,
		\end{split}
	\end{equation}
	provided that, we obtain the conclusions from Lemma \ref{Stp} on $I_k$ for all $0\le k\le j$. Summing the bounds corresponding to \eqref{Stp5-1} through  \eqref{Stp7-1} over all subintervals $I_j$, we obtain 
	\begin{equation}\label{Stp5-all}
		\|V(\theta)(\phi-u)\|_{X^*(I)} \le C(M,L,M') \ep^a (M+M')^b,
	\end{equation}
	\eqref{Ltp7}, and \eqref{Ltp8}. From \eqref{Stp5-all}, \eqref{Ltp7}, and Lemma \ref{Xint}, \eqref{Ltp6} is obtained. We have completed the proof.
\end{proof}

\section{Proof of the scattering result}\label{proofsc}

In this section, we focus on the case $\lmd=-1$ and prove that the global solution to \eqref{NLS} with $\phi(0)=\phi_0 \in \cK^+$ scatters in $B^1$. When $\lmd=+1$, we obtain the scattering result for any initial data in $B^1$ using a similar strategy. See Remark \ref{scdef} at the end of this section.
\subsection{Existence of a critical element}

For $\phi_0\in B^1$, we say that (SC)($\phi_0$) holds if the corresponding solution $\phi$ to \eqref{NLS} exists globally and satisfies
\begin{equation}
	\|V(\theta)\phi\|_{L_t^{2q} L_\theta^{q}L_x^{p_0}(\R\times [0,\f{\pi}{2}]\times \R^3)} <\infty,
\end{equation}  
that is, $\phi$ scatters. We consider
\begin{equation}
	S_c :=\sup\{ A : \text{ If } S[\phi_0]<A \text{ and } \phi_0\in \cK^+, \text{(SC)}(\phi_0) \text{ holds } \}.
\end{equation}
From Lemmas \ref{Ssc} and \ref{3.5}, $S_c>0$. Therefore, if we show $S_c=d$, the proof of Theorem \ref{main} is complete. Thus, we assume that $S_c<d$.

\begin{prop}\label{palais}
	Let $2<\s<4$. We assume that $S_c<d$. Suppose that $\{\phi_{n}\}_{n=1}^\infty$ is a sequence of solutions to \eqref{NLS} in $B^1$ such that $\phi_{n}(t) \in \cK^+$, $\limsup_{n\to \infty}S[\phi_n]=S_c$, and there exists $\{t_n\}_{n=1}^\infty\subset \R$ such that
	\begin{equation}\label{nonsc}
		\lim_{n\to\infty}\|V(\theta)\phi_n\|_{L_t^{2q}([t_n,\infty), L_\theta^{q}L_x^{p_0})} = \lim_{n\to\infty}\|V(\theta)\phi_n\|_{L_t^{2q}((-\infty,t_n], L_\theta^{q}L_x^{p_0})} =\infty	
	\end{equation}
	holds. Then, there exist $\psi \in  B^1$ and a sequence  $\{z_n\}_{n=1}^\infty \subset \R$ such that
	\begin{equation}
		\phi_n(t_n, \cdot, \cdot+z_n)\to \psi \quad \text{in} \quad  B^1 \quad \text{as} \quad n\to \infty.
	\end{equation}
\end{prop}

\begin{proof}
	From time-translation symmetry, we may assume that $t_n\equiv 0$. From Lemma \ref{3.5}, we have that 
	\begin{equation}\label{bddphin}
		\|\phi_n(t)\|_{B^1}\lesssim S[\phi_n] (\le d),	
	\end{equation} 
	and  $\{ \phi_{n}(0) \}_{n=1}^\infty$ is bounded by $B^1$. We use Proposition \ref{Lpd} to obtain a profile for any finite $ J\le J^*$.
	\begin{equation*}
		\phi_n(0,x)=\sum_{j=1}^J U(t_n^j)\psi^j(y,z-z_n^j) +r_n^J(x), 
	\end{equation*}
	\begin{equation*}
		U(-t_n^J)r_n^J(y,z+z_n^J) \rightharpoonup 0 \quad \text{ weakly in} \quad B^1,
	\end{equation*}
	\begin{equation}\label{decompS}
		S[\phi_n(0)] = \sum_{j=1}^JS[U(t_n^j)\psi^j] +S[r_n^J] +o_n(1),
	\end{equation}
	\begin{equation*}
		I[\phi_n(0)] = \sum_{j=1}^JI[U(t_n^j)\psi^j] +I[r_n^J] +o_n(1),
	\end{equation*}
	\begin{equation}\label{orthj}
		\lim_{n\to\infty}|t_n^j-t_n^k|+|z_n^j-z_n^k|=\infty, \quad \text{for any} \quad j\neq k,
	\end{equation}
	\begin{equation*}
		\lim_{J\to J^*} \limsup_{n\to \infty}\|U(t)V(\theta)r_n\|_{L_t^{2q}L_\theta^{q}L_x^{p_0}(\R \times [0,\f{\pi}{2}]\times \R^3)} =0.
	\end{equation*}
	If $J^*<\infty$, by setting $r_n^J\equiv 0 $ for all $J> J^*$, we may set $J^*=\infty$.
	Using the same argument as in the proof of \cite[Propositon 5.2]{Gdb}, we obtain
	\begin{equation}\label{decK+}
		U(t_n^j)\psi^j\in \cK^+, \quad r_n^J\in \cK^+\quad \text{for all sufficiently large }n,
	\end{equation}
	for any $J$ and $j$. Lemma \ref{3.5} implies that  $S[U(t_n^j)\psi^j] \ge0$ and $S[ r_n^J]\ge0$, and by \eqref{decompS}, it holds for each $1\le j$ that
	\begin{equation}\label{bdprof}
		0\le \limsup_{n\to \infty}S[U(t_n^j)\psi^j] \le \limsup_{n\to \infty}S[\phi_n(0)] \le S_c.
	\end{equation}
	Note that for each $j$, either $t_n^j\equiv 0$ or $t_n^j\to \pm\infty$ as $n\to\infty$. \\
	
	For each $j$, we introduce a nonlinear profile $v^j$ associated $\psi^j$ which depends on the limiting value of $t_n^j$.
	\begin{itemize}
		\item
		If $t_n^j\equiv 0$, we define $v^j$ to be the solution to \eqref{NLS} with $v^j(0)=\psi^j$.
		
		\item
		If $t_n^j \to +\infty$, we define $v^j$ to be the solution to \eqref{NLS} which scatters forward in time to $U(t)\psi^j$.
		
		\item
		If $t_n^j \to -\infty$, we define $v^j$ to be the solution to \eqref{NLS} which scatters backward in time to $U(t)\psi^j$.
	\end{itemize}
	Each $v^j$ belongs to $\cK^+$ throughout its lifespan, and exists globally. Indeed, if $t_n^j\equiv 0$, by \eqref{decK+} and the definition of $v^j$, it holds that  
	\[ S[v^j(0)]=S[\psi^j] \in \cK^+. \]
	By contrast, if  $t_n^j\to \pm \infty$, we have from \eqref{bdprof} that
	\[ \h\|\psi^j\|_{B^1}^2+ \h \|\psi^j\|_{L^2}^2=\lim_{t\to  \infty}S[U(t)\psi^j] \le S_c<d,\]
	and Lemma \ref{waveo} supports the claim.
	We denote $ v_n^j(t,x):= v^j(t+t_n^j,y, z-z_n^j)$ and $u_n^J:=\sum_{j=1}^J v_n^j$. Let 
	\begin{equation*}
		w_n^J(x):= \phi_n(0,x)-u_n^J(0,x)=\phi_n(0,x)-\sum_{j=1}^J v_n^j(0,x).
	\end{equation*}
	Because $r_n^J-w_n^J \to 0$ in $B^1$ as $n\to \infty$ for all $J$, we obtain
	\begin{equation*}
		U(-t_n^J)w_n(\cdot, \cdot+z_n^J) \rightharpoonup 0 \quad \text{ weakly in} \quad B^1,
	\end{equation*}
	\begin{equation*}
		\lim_{J\to \infty} \limsup_{n\to \infty}\|U(t)V(\theta)w_n\|_{L_t^{2q}L_\theta^{q}L_x^{p_0}(\R \times [0,\f{\pi}{2}]\times \R^3)} =0,
	\end{equation*}
	and $\limsup_{n\to \infty}S[w_n^J]\ge 0$ for each $J$. Therefore, we have
	\begin{equation}\label{bddSc}
		\sum_{j=1}^{\infty} S[v_n^j]=\sum_{j=1}^{\infty}S[v^j]\le \limsup_{n\to \infty}S[\phi_n(0)] \le S_c.
	\end{equation}
	Here, we consider two cases.\\
	
	\noindent
	\underline{\textbf{Case 1}: $\sup_{1\le j }S[v^j] < S_c$.}\\
	
	In this case, by the definition of $S_c, $ each $v^j$ satisfies \[ \|V(\theta)v^j\|_{L_t^{2q}L_\theta^{q}L_x^{p_0}(\R)}<\infty.\]
	From \eqref{bddSc}, Lemmas \ref{Ssc} and \ref{3.5}, there exists $J_0$ such that
	$j\ge J_0$ implies that
	\[ \|V(\theta)v^j\|_{L_t^{2q}L_\theta^{q}L_x^{p_0}(\R)}\lesssim S[v_j(0)]^\h. \]
	From the orthogonality, it holds for large $J$ that
	\begin{equation*}
		\begin{split}
			\|V(\theta)u_n^J\|_{L_t^{2q}L_\theta^{q}L_x^{p_0}(\R)}^2 &\lesssim \sum_{j=1}^{J_0-1} 	\|V(\theta)v_n^j\|_{L_t^{2q}L_\theta^{q}L_x^{p_0}(\R)}^2 + \sum_{j=J_0}^{J} S[v_j(0)] + o_n(1)\\
			&\lesssim \sum_{j=1}^{J_0-1} 	\|V(\theta)v^j\|_{L_t^{2q}L_\theta^{q}L_x^{p_0}(\R)}^2 + S_c  + o_n(1) \quad \text{ as } n\to\infty .
		\end{split}
	\end{equation*}
	Hence,
	\begin{equation}\label{bdduJsc}
		\limsup_{n\to \infty}\|V(\theta)u_n^J\|_{L_t^{2q}L_\theta^{q}L_x^{p_0}(\R)} \lesssim \Big(\sum_{j=1}^{J_0-1} \|V(\theta)v^j\|_{L_t^{2q}L_\theta^{q}L_x^{p_0}(\R)}^2 +  S_c\Big)^\h.
	\end{equation}
	The right-hand side does not depend on $J$. 
	We now prove that $u_n^J$ is a good approximation of $\phi_n$ for sufficiently large $n$ and $J$.

	\begin{lemma}\label{Lerror}
		We have
		\begin{equation}\label{error1}
			\limsup_{J\to \infty} \limsup_{n\to \infty} \|U(t)V(\theta)(u_n^J(0)-\phi_n(0))\|_{L_t^{2q}L_\theta^{q}L_x^{p_0}(\R)}=0,
		\end{equation}
		\begin{equation}\label{error2}
			\limsup_{n\to \infty}\|\Fav(u_n^J)-\sum_{j=1}^J\Fav(v_n^j)\|_{N^1(\R)}=0 \quad \text{for all} \quad J\ge1.
		\end{equation}
	\end{lemma}
	\begin{proof}
		By the decomposition, we obtain
		\begin{equation*}
			\begin{split}
				&\limsup_{J\to \infty} \limsup_{n\to \infty} \|U(t)V(\theta)(u_n^J(0)-\phi_n(0))\|_{L_t^{2q}L_\theta^{q}L_x^{p_0}(\R)}\\
				=&\limsup_{J\to \infty} \limsup_{n\to \infty} \|U(t)V(\theta)w_n^J\|_{L_t^{2q}L_\theta^{q}L_x^{p_0}(\R)}=0.
			\end{split}
		\end{equation*}
		Next, by using \eqref{eqH}, Lemma \ref{St2}, and Minkowski's inequality, we obtain
		\begin{equation}\label{defFav}
			\begin{split}
				&\Big\| \nabla_x \Big(  \Fav(u_n^J)-\sum_{j=1}^J\Fav(v_n^j) \Big)\Big\|_{L_t^{(2q)'}L_z^{p'}L_y^2(\R)} + \Big\| \langle y\rangle \Big(  \Fav(u_n^J)-\sum_{j=1}^J\Fav(v_n^j) \Big)\Big\|_{L_t^{(2q)'}L_z^{p'}L_y^2(\R)}\\
				&\lesssim \Big\|\int_0^{\f{\pi}{2}} V(\theta)\Big[ H^\h\Big(F(V(\theta)u_n^J)-\sum_{j=1}^J F(V(\theta)v_n^j)\Big)\Big] d\theta \Big\|_{L_t^{(2q)'}L_z^{p'}L_y^2(\R)}\\
				&\qquad \qquad \qquad  +  \Big\|\int_0^{\f{\pi}{2}} V(\theta)\Big[\partial_z \Big( F(V(\theta)u_n^J)-\sum_{j=1}^J F(V(\theta)v_n^j)\Big)\Big] d\theta \Big\|_{L_t^{(2q)'}L_z^{p'}L_y^2(\R)}\\
				&\lesssim \Big\|\nabla_x\Big( F(V(\theta)u_n^J)-\sum_{j=1}^J F(V(\theta)v_n^j) \Big)\Big\|_{L_t^{(2q)'}L_z^{p'}L_\theta^{q'} L_y^{p'}(\R)}\\
				&\qquad  \qquad \qquad +  \Big\|\langle y\rangle \Big(  F(V(\theta)u_n^J)-\sum_{j=1}^J F(V(\theta)v_n^j) \Big)\Big\|_{L_t^{(2q)'}L_z^{p'}L_\theta^{q'} L_y^{p'}(\R)},
			\end{split}
		\end{equation}
		where $F(f)=\lmd |f|^{2\s}f$. For $J\ge2$ there exists $C_{\s,J}>0$ such that for any $\{\al_j\}_{j=1}^J \subset \C$ it holds that
		\begin{equation*}
			\Big| F\Big(\sum_{j=1}^J \al_j\Big) - \sum_{j=1}^J F(\al_j)  \Big| \le C_{\s,J} \sum_{1\le j \neq k \le J}|\al_j|^{2\s}|\al_k|.
		\end{equation*}
		\begin{equation*}
			\begin{split}
				\Big| \nabla_x\Big(F\Big(\sum_{j=1}^J \al_j\Big) - \sum_{j=1}^J F(\al_j) \Big) \Big|\le& C_{\s, J}  \sum_{1\le j \neq k \le J}|\nabla_x \al_j||\al_k| \sum_{l=1}^J|\al_l|^{2\s-1} \\
				&+   \sum_{1\le j \neq k \le J}| \al_j||\al_k|\sum_{l=1}^J|\nabla_x \al_l| \sum_{m=1}^J|\al_m|^{2\s-2}  .
			\end{split}
		\end{equation*}
		Then, we have
		\begin{equation*}
			\begin{split}
				& \Big\|\nabla_x\Big( \big|\sum_{j=1}^JV(\theta)u_n^j \big|^{2\s}\sum_{j=1}^J V(\theta)v_n^j -\sum_{j=1}^J|V(\theta)v_n^j|^{2\s} V(\theta)v_n^j\Big)\Big\|_{L_t^{(2q)'}L_\theta^{q'} L_x^{p'}}\\
				&\lesssim \sum_{j\neq k} \|(\nabla_xV(\theta)v_n^j)V(\theta)v_n^k\|_{L_t^{q}L_\theta^\f{q}{2}L_x^{(\f{1}{p}+\f{1}{p_0})^{-1}}}  \sum_{l=1}^J\|V(\theta)v_n^l\|_{L_t^{2q}L_\theta^{q}L_x^{p_0}}^{2\s-3}\|H^\h v_n^l\|_{L_t^4 L_z^\infty L_y^2}^2\\
				&\quad +  \sum_{j\neq k} \|V(\theta)v_n^j V(\theta)v_n^k\|_{L_t^{q}L_\theta^\f{q}{2}L_x^\f{p_0}{2}} \sum_{l=1}^J\|\nabla_x V(\theta)v_n^l\|_{L_t^{2q}L_\theta^{q}L_x^p}\\
				&\qquad \qquad \qquad \qquad \qquad \qquad \times  \sum_{l=m}^J\|V(\theta)v_n^m\|_{L_t^{2q}L_\theta^{q}L_x^{p_0}}^{2\s-4}\|H^\h v_n^m\|_{L_t^4 L_z^\infty L_y^2}^2.
			\end{split}
		\end{equation*}
		To calculate these indices, we refer to Lemma \ref{St4}, Remarks \ref{ind} and \ref{indad}.
		Using the argument in the proof of Proposition \ref{Ltp}, $\|V(\theta)v^j\|_{L_t^{2q}L_\theta^{q}L_x^{p_0}(\R)}<\infty$ implies $\|v^j\|_{S^1(\R)}<\infty$. See \eqref{uS1}.  By the orthogonality,  
		\begin{equation*}
			\|(\nabla_xV(\theta)v_n^j)V(\theta)v_n^k\|_{L_t^{q}L_\theta^\f{q}{2}L_x^{(\f{1}{p}+\f{1}{p_0})^{-1}}} \to0, \qquad
			\|V(\theta)v_n^j V(\theta)v_n^k\|_{L_t^{q}L_\theta^\f{q}{2}L_x^\f{p_0}{2}} \to 0 \quad \text{as} \quad n\to\infty.
		\end{equation*}
		The second term on the right-hand side of \eqref{defFav} can be estimated in the same manner. Therefore, we have \eqref{error2}.
	\end{proof}
	
	\begin{lemma}\label{bdduJ} For all $J\ge1$, we have 
		\begin{equation}\label{bdduJ1}
			\limsup_{n\to\infty}\|u_n^J\|_{L_t^\infty B^1(\R)} \lesssim  S_c.
		\end{equation} 
	\end{lemma}
	\begin{proof} 
		Fix $J$. Let $\ep>0$ be  arbitrarily small. As each $v^j$ scatters forward and backward, the set $\{U(-t)v^j(t)\}_{t\in \R} $ is precompact. That is, for each $j$, there exist $L^j\in \N$ and $\{f_l^j \}_{l=1}^{L^j} \subset B^1$ such that the following statement holds:\\
		
		For any $t\in \R$, there exists $l$ such that $\|U(-t)v^j(t)-f_l^j\|_{B^1}\le \ep$.\\
		
		\noindent
		Therefore, for any $1\le j<k \le J$, $t\in \R$ and $n\in \N$, 
		\begin{equation}
			\begin{split}
				&	\min_{1\le l \le L^j, 1\le m \le L^k}|\langle v_n^j(t), v_n^k(t) \rangle_{B^1} - \langle U(t_n^j)f_l^j( \cdot -z_n^j),  U(t_n^k)f_m^k( \cdot -z_n^k) \rangle_{B^1}|\\
				=&\min_{1\le l \le L^j, 1\le m \le L^k}|\langle v_n^j(t), v_n^k(t) \rangle_{B^1} - \langle U(t+t_n^j)f_l^j( \cdot -z_n^j),  U(t+t_n^k)f_m^k( \cdot -z_n^k) \rangle_{B^1}|\\
				\lesssim  & S_c \ep.	
			\end{split}
		\end{equation}
		In contrast, by the orthogonality \eqref{orthj}, there exists $N^J$ such that $n\ge N^J$ implies
		\begin{equation}
			\max_{1\le j<k \le J}\max _{ 1\le l \le L^j, 1\le m\le L^k}|\langle U(t_n^j)f_l^j( \cdot -z_n^j),  U(t_n^k)f_m^k( \cdot -z_n^k) \rangle_{B^1}|\le \ep.
		\end{equation}
		Therefore, for any $1\le j<k \le J$, $n\ge N^J$ implies
		\begin{equation*}
			|\langle v_n^j(t), v_n^k(t) \rangle_{B^1}|\lesssim (S_c +1)\ep,
		\end{equation*}
		for all $t\in \R$. Then,
		\begin{equation*}
			\begin{split}
				\|u_n^J(t)\|_{B^1}^2 &\le \sum_{j=1}^J \|v_n^j(t)\|_{B^1}^2 + C\sum_{1\le j<k \le J} (S_c +1) \ep \lesssim S_c+ (S_c+1) J\ep.
			\end{split}
		\end{equation*}
		Hence, we obtain \eqref{bdduJ1}.
	\end{proof}
	
	From \eqref{bdduJsc}, Lemmas \ref{Lerror} and \ref{bdduJ}, we can apply Proposition \ref{Ltp} to $u_n^J$ and $\phi_n$ for sufficiently large $J$ depending on the right-hand side of \eqref{bdduJsc}, and $S_c$ and large $n$ depending on $J$. Then, we obtain  
	\begin{equation*}
		\|V(\theta)\phi_n\|_{L_t^{2q}L_\theta^{q}L_x^{p_0}(\R)} \le C(S_c)
	\end{equation*}
	for all sufficiently large $n$, contradicting \eqref{nonsc}. Therefore, Case 1 does not occur.\\

	\noindent
	\underline{\textbf{Case 2}: $\sup_{1\le j}S[v^j] = S_c$.}\\
	
	Comparing this with \eqref{bddSc}, we observe that $S[v^j]=0$ except for one element, and we may assume $S[v^1]=S_c$. Then, the profile decomposition is simplified to
	\begin{equation}\label{precpt0}
		\phi_n(0,x)=U(t_n^1)\psi^1(y, z-z_n^1) + r_n^1(x).
	\end{equation}
	Because $S[\phi_n(0)] \to S_c$, $S[U(t_n^1)\psi^1] \to S[v^1]=S_c$, and \eqref{decompS}, we have $S[r_n^1]\to 0$. As $r_n^1$ belongs to $\cK^+$, this implies that
	\begin{equation}\label{rem}
		r_n^1\to 0  \quad \text{ in } \quad B^1
	\end{equation} 
	as $n\to\infty$. If $t_n^1\equiv0$, then it holds that $\phi_n(0, \cdot, \cdot+z_n^1)\to \psi^1$ in $B^1$. This is the conclusion.\\
	Suppose that $t_n^1\to +\infty$ is $n\to \infty$. From \eqref{UVB}, we obtain that
	\begin{equation*}
		\|U(t)V(\theta)\psi^1\|_{L_t^{2q}L_\theta^q L_x^{p_0}(\R)}\lesssim \|\psi\|_{B^1}<\infty.
	\end{equation*}
	Combining this with \eqref{precpt0} and \eqref{rem}, we obtain
	\begin{equation}
		\|U(t)V(\theta)\phi_n(0)\|_{L_t^{2q}L_\theta^q L_x^{p_0}([0,\infty))} \to 0 \quad \text{as} \quad  n\to +\infty.
	\end{equation}
	We use Proposition \ref{Ltp} with $I=[0,\infty)$, $u=\phi_n$, $e=0$, and $\phi=0$ to obtain $\|V(\theta)\phi_n\|_{L_t^{2q}L_\theta^q L_x^{p_0}([0,\infty))}\to 0$, which contradicts \eqref{nonsc}. The same argument allows us to exclude the case $t_n^1\to-\infty$.
\end{proof}

\begin{prop}\label{criticale}
	Let $2<\s<4$. We assume that $S_c<d$. Then, there exists $\phi_{0,c}\in  B^1$ such that the solution $\phi_c$ to \eqref{NLS} with $\phi_c(0)=\phi_{0,c}$ satisfies $\phi_c(t)\in \cK^+$, $S[\phi(t)]=S_c$, and 
	\begin{equation}\label{nonsc2}
		\|V(\theta)\phi_c\|_{L_t^{2q}([0,\infty), L_\theta^{q}L_x^{p_0})} = \|V(\theta)\phi_c\|_{L_t^{2q}((-\infty,0], L_\theta^{q}L_x^{p_0})} =\infty.	
	\end{equation}	
	Furthermore, there exists a function $z\in C([0,\infty), \R)$ such that
	\begin{equation}\label{precpt1}
		\{\phi_c(t,\cdot, \cdot+z(t)): t\ge0\}
	\end{equation}
	is precompact.
\end{prop}

\begin{proof}
	From the definition of $S_c$ and the assumption $S_c<d$, there exists $\{ \phi_{n} \}_{n=1}^\infty$, a sequence of global solutions to \eqref{NLS} such that $\phi_n(t) \in \cK^+$,  $S[\phi_n(0)]\to S_c$. This satisfies
	\begin{equation}\label{non}
		\lim_{n\to\infty}\|V(\theta)\phi_n\|_{L_t^{2q}(\R , L_\theta^{q}L_x^{p_0})} = \infty.
	\end{equation}
	Indeed, if there exists $L>0$ such that 
	\[ \|V(\theta)\phi_n\|_{L_t^{2q}(\R , L_\theta^{q}L_x^{p_0})}\le L \]
	holds for any sufficiently large $n$, then by Proposition \ref{Ltp}, there exists a global solution to \eqref{NLS} $\tilde{\phi}$ such that 
	\[ S[\tilde{\phi}]=S_c+\ep, \quad \text{and} \quad \|V(\theta)\phi\|_{L_t^{2q}(\R , L_\theta^{q}L_x^{p_0})}<\infty, \]
	where $\ep$ is a positive constant that depends on $S_c$ and $L$. This contradicts the definition of $S_c$.
	Thus, \eqref{non} holds, and there exists  a sequence $\{t_n\}_{n=1}^\infty$ such that 
	\begin{equation}\label{nonsc3}
		\lim_{n\to\infty}\|V(\theta)\phi_n\|_{L_t^{2q}([t_n,\infty), L_\theta^{q}L_x^{p_0})} = \lim_{n\to\infty}\|V(\theta)\phi_n\|_{L_t^{2q}((-\infty,t_n], L_\theta^{q}L_x^{p_0})} =\infty.	
	\end{equation}
	By applying Proposition \ref{palais} to $\{ \phi_{n} \}_{n=1}^\infty$, we have 
	\begin{equation}\label{strongc}
		\phi_n(t_n, \cdot, \cdot+z_n)\to  \psi \quad \text{in} \quad B^1
	\end{equation}
	for some $\{z_n\}_{n=1}^\infty$ and $\psi\in B^1$.
	By setting $\phi_{0,c}=\psi$, we obtain the desired solution $\phi_c$. First, we have
	\begin{equation}
		S[\phi_c(t)]=S[\psi]=\lim_{n\to \infty} S[\phi_n]=S_c.
	\end{equation}
	Second, by combining Propositions \ref{Ltp}, \eqref{nonsc3}, and \eqref{strongc}, we deduce \eqref{nonsc2}.
	Finally, to prove the precompactness of \eqref{precpt1}, we apply Proposition \ref{palais} to $\phi_c$ and an arbitrary sequence $\{t_n\}_{n=1}^\infty\subset \R_\ge0$  with $\tilde{\phi}_n =\phi_c$ (where $\tilde{\phi}_n$ represents the sequence in Proposition \ref{palais}). Then, there exist $\psi\in B^1$ and $\{z_n\}_{n=1}^\infty \subset \R$ such that
	\begin{equation*}
		\phi_c(t_n, \cdot, \cdot+z_n)\to  \psi \quad \text{in} \quad B^1.
	\end{equation*}
\end{proof}

\subsection{Extinction of the critical element}

\begin{prop}\label{extinc}
	Let $\phi_c$ be the critical element constructed in Proposition \ref{criticale}. Then, $\phi_c\equiv0$.
\end{prop}
\begin{proof}
	We assume that $\phi_{0,c}\neq0$. 
	We define
	\begin{equation}
		W(t):= \int_{\R^3} \chi(z)|\phi_c(t)|^2 dx,
	\end{equation}
	where $\chi \in C_{\text{rad}}^4(\R)$ such that
	\begin{equation}\label{chi11}
		\chi(z)=
		\left\{
		\begin{array}{lll}
			z^2 &&	 0<|z|\le R \\
			0  &&  2R<|z|
		\end{array}
		\right. 
	\end{equation}
	and
	\begin{equation}\label{chi12}
		0\le \chi \le z^2,\quad |\chi'|\lesssim R, \quad \chi''\le 2, \quad \chi^{(4)}\le \f{4}{R}.
	\end{equation}
	By direct calculations, we obtain:
	\begin{equation}
		W'(t)= 2\text{Im}\int_{\R^3}\partial_z\phi_c(t) \overline{\phi_c(t)} \chi'(z) dx
	\end{equation}
	\begin{equation}\label{W"}
		\begin{split}
			W''(t)&= 4P[\phi_c(t)] + 4 \int_{\R^3}|\partial_z\phi_c(t)|^2 (\chi''(z)-2) dx \\ 
			&\quad -\f{4\s}{\pi(\s+1)}\int_0^{\f{\pi}{2}}\int_{\R^3}|V(\theta)\phi_c(t)|^{2\s+2}(\chi''(z)-2)d\theta dx	 - \int_{\R^3}|\phi_c(t)|^2 \chi^{(4)}(z) dx \\
			&=: 4P[\phi_c(t)]+ \cR_1+\cR_2+\cR_3.
		\end{split}
	\end{equation}
	Then, there exists $C_1>0$ such that 
	\begin{equation}\label{bddcRj}
		| \cR_1+\cR_2+\cR_3|\le C_1\int_{\{|z|\ge R\}} \Big[ |\partial_z\phi_c(t)|^2+|\phi_c(t)|^2 +\int_{0}^{\f{\pi}{2}}|V(\theta)\phi_c(t)|^{2\s+2}d\theta \Big] dx.
	\end{equation}
	Because $\|\phi_c(t)\|_{B^1} \simeq S[\phi_c(t)] = S_c$, there exists $C_2>0$ such that
	\begin{equation}\label{Wbdd2}
		|W'(t)|\le C_2 R.
	\end{equation}
	 Here, let $\eta>0$ satisfy
	\begin{equation}
		P[\phi_c(t)]\ge\eta
	\end{equation}
	for all $t\ge0$, refer to Step 2 in the proof of \cite[Lemma 5.5]{Gdb}.
	Due to the precompactness of \eqref{precpt1}, there exists a large $\rho$ such that
	\begin{equation}\label{cptcRj}
		\int_{\{|z-z(t)|\ge \rho\}} \Big[ |\partial_z\phi_c(t)|^2+|\phi_c(t)|^2 +\int_{0}^{\f{\pi}{2}}|V(\theta)\phi_c(t)|^{2\s+2}d\theta \Big] dx\le  \f{\eta}{C_1}.
	\end{equation}
	 On the other hand, by the same argument as in Step 1 of the proof of \cite[Lemma 5.5]{Gdb}, we have
	\begin{equation}\label{z(t)}
		\lim_{n\to\infty}\f{|z(t)|}{t}=0.
	\end{equation}
	To prove \eqref{z(t)}, we use the form
	\begin{equation}
		\Gamma_R(t):=\int_{\R^3} R\tilde{\chi}\Big(\f{z}{R}\Big) |\phi_c(t, x)|^2 dx,
	\end{equation}
	where $\tilde{\chi} \in C_0^\infty(\R)$ satisfies:
	\begin{equation*}
	\end{equation*}
	\begin{equation*}
		\tilde{\chi}(z)= \left\{\begin{array}{ll}z, & |z| \le 1\\
			0, &  |z|\ge 2^{\f{1}{3}}
		\end{array} \right.
	\end{equation*}
	and
	\begin{equation*}
		|\tilde{\chi}(z)|\le |z|, \quad \|\tilde{\chi}\|_{L^\infty}\le 2, \quad  \|\tilde{\chi}'\|_{L^\infty}\le 4.
	\end{equation*}
	From \eqref{z(t)}, there exists $t_0>0$ such that 
	\[ |z(t)|\le \f{\eta }{4C_2} t\]
	holds for all $t\ge t_0$. Let $t_1 > t_0$, and we set
	\[R_{t_1}:= \rho + \f{\eta t_1}{4C_2}.\]
	Then, it holds that $\{ |z|\ge R_{t_1} \}\subset \{|z-z(t)|\ge \rho \}$ for $t \in [t_0, t_1]$ and we obtain from \eqref{W"}, \eqref{bddcRj}, and \eqref{cptcRj} that
	\begin{equation}\label{Wbdd3}
		W''(t)\ge 4P[\phi_c(t)]-|\cR_1+\cR_2+\cR_3| \ge 4\eta-\eta=3\eta.
	\end{equation}
	However, if we take $R=R_{t_1}$, from \eqref{Wbdd3} and \eqref{Wbdd2}, we have 
	\begin{equation*}
		3\eta(t_1-t_0)\le \int_{t_0}^{t_1}W''(t)dt\le |W'(t_1)-W'(t_0)|\le 2 C_2 R_{t_1}=2C_2\rho+\f{\eta t_1}{2}.
	\end{equation*} 
	By choosing $t_1$ to be sufficiently large, we have a contradiction. Hence, $\phi_{0,c}=0$ and $\phi_c\equiv 0$.
\end{proof}

\begin{proof}[Proof of the scattering part of Theorem \ref{main}]
	We assume that $S_c<d$. Thus, we obtain a contradiction between Propositions \ref{criticale} and \ref{extinc}. Thus, $S_c=d$.
\end{proof}

\begin{remark}\label{scdef}
	When $\lmd=+1$, we first define
	\begin{equation}
		S_c :=\sup\{ A : \text{ If } S[\phi_0]<A \text{ and } \phi_0\in B^1, \text{(SC)}(\phi_0) \text{ holds } \},
	\end{equation}
	and we assume that $S_c<\infty$. Then, using the same argument as in this section, we have the contradiction. Note that for any $\psi\in B^1$, it holds that
	\begin{equation*}
		0\le \|\psi\|_{B^1}^2 \le 2S[\psi].
	\end{equation*}
\end{remark}

\section*{Acknowledgments}
\noindent
The author is deeply grateful to Professor Kenji Nakanishi for his  valuable advice on this study. \\
The author is supported by JST, the establishment of university fellowships towards the creation of science technology innovation, Grant Number JPMJFS2123.\\
The author would like to thank Editage (www.editage.jp) for English language editing.\\

\noindent
The author reports that there are no competing interests to declare.

\bibliographystyle{abbrv}
\footnotesize{
	\bibliography{ScNLSavbib}

\begin{thebibliography}{10}

\bibitem{WPsnlNLS}
J.~Albert and E.~Kahlil.
\newblock On the well-posedness of the {C}auchy problem for some nonlocal
  nonlinear {S}chrödinger equations.
\newblock {\em Nonlinearity}, \textbf{30}(no.6):2308--2333, 2017.

\bibitem{Scpht}
P.~Antonelli, R.~Carles, and J.~Drumond~Silva.
\newblock Scattering for nonlinear {S}chrödinger equation under partial
  harmonic confinement.
\newblock {\em Commun. Math. Phys.}, \textbf{334}(no.1):367--396, 2015.

\bibitem{Gdb}
A.~H. Ardila and R.~Carles.
\newblock Global dynamics below the ground states for {NLS} under partial
  harmonic confinement.
\newblock {\em Commun. Math. Sci.}, \textbf{19}(no.4):993--1032, 2021.

\bibitem{Bell}
J.~Bellazzini, N.~Boussaïd, L.~Jeanjean, and N.~Visciglia.
\newblock Existence and stability of standing waves for supercritical {NLS}
  with a partial confinement.
\newblock {\em Commun. Math. Phys.}, \textbf{353}(no.1):229--251, 2017.

\bibitem{N.Ben2}
N.~Ben~Abdallah, F.~Castella, F.~Delebecque-Fendt, and F.~Méhats.
\newblock The strongly confined {S}chrödinger-{P}oisson system for the
  transport of electrons in a nanowire.
\newblock {\em SIAM J. Appl. Math.}, \textbf{69}(no.4):1162--1173, 2009.

\bibitem{N.Ben}
N.~Ben~Abdallah, F.~Castella, and F.~M\'{e}hats.
\newblock Time averaging for the strongly confined nonlinear {S}chr\"{o}dinger
  equation, using almost-periodicity.
\newblock {\em J. Differ. Equ.}, \textbf{245}:154--200, 2008.

\bibitem{N.Ben3}
N.~Ben~Abdallah, F.~Méhats, C.~Schmeiser, and R.~M. Weishäupl.
\newblock The nonlinear {S}chrödinger equation with a strongly anisotropic
  harmonic potential.
\newblock {\em SIAM J. Math. Anal.}, \textbf{37}(no.1):189--199, 2005.

\bibitem{CRH}
T.~Buckmaster, P.~Germain, Z.~Hani, and J.~Shatah.
\newblock Analysis of {(CR)} in higher dimension.
\newblock {\em Int. Math. Res. Not. IMRN}, (no.4):1265--1280, 2019.

\bibitem{RNLSHP}
R.~Carles.
\newblock Remarks on nonlinear {S}chrödinger equations with harmonic
  potential.
\newblock {\em Ann. Henri Poincaré}, \textbf{3}(no.4):757--772, 2002.

\bibitem{SiS}
E.~Carneiro.
\newblock A sharp inequality for the {S}trichartz norm.
\newblock {\em Int. Math. Res. Not.}, \textbf{2009}(no.16):3127--3145, 2009.

\bibitem{Ord}
X.~Chen.
\newblock On the rigorous derivation of the 3{D} cubic nonlinear {S}chrödinger
  equation with a quadratic trap.
\newblock {\em Arch. Ration. Mech. Anal.}, \textbf{210}(no.2):365--408, 2013.

\bibitem{Sc3DNLSphp}
X.~Cheng, C.-Y. Guo, Z.~Guo, X.~Liao, and J.~Shen.
\newblock Scattering of the three-dimensional cubic nonlinear {S}chrödinger
  equation with partial harmonic potentials.
\newblock {\em preprint, arXiv:2105.02515}.

\bibitem{Sc2Dfres}
X.~Cheng, Z.~Guo, G.~Hwang, and H.~Yoon.
\newblock {G}lobal well-posedness and scattering of the two dimensional cubic
  focusing nonlinear {S}chr\"{o}dinger system.
\newblock {\em preprint, arXiv:2202.10757}.

\bibitem{ScNLSR2T}
X.~Cheng, Z.~Guo, K.~Yang, and L.~Zhao.
\newblock On scattering for the cubic defocusing nonlinear {S}chrödinger
  equation on the waveguide $\mathbb{R}^2\times\mathbb{T}$.
\newblock {\em Rev. Mat. Iberoam.}, \textbf{36}(no.4):985--1011, 2020.

\bibitem{ScNLSRT}
X.~Cheng, Z.~Guo, and Z.~Zhao.
\newblock On scattering for the defocusing quintic nonlinear {S}chrödinger
  equation on the two-dimensional cylinder.
\newblock {\em SIAM J. Math. Anal.}, \textbf{52}(no.5):4185--4237, 2020.

\bibitem{wpecNLSM}
X.~Cheng, Z.~Zhao, and J.~Zheng.
\newblock Well-posedness for energy-critical nonlinear {S}chrödinger equation
  on waveguide manifold.
\newblock {\em J. Math. Anal. Appl.}, \textbf{494}(no.2):Paper No. 124654, 14
  pp., 2021.

\bibitem{GbDMNLS}
M.-R. Choi, Y.~Hong, and Y.-R. Lee.
\newblock Global existence versus finite time blowup dichotomy for the
  dispersion managed {NLS}.
\newblock {\em arXiv:2311.02905}.

\bibitem{WPADM}
M.-R. Choi, D.~Hundertmark, and Y.-R. Lee.
\newblock Well-posedness of dispersion managed nonlinear {S}chrödinger
  equations.
\newblock {\em J. Math. Anal. Appl.}, \textbf{522}(no.1):126938, 37 pp., 2023.

\bibitem{DMNLDla}
M.-R. Choi, Y.~Kang, and Y.-R. Lee.
\newblock On dispersion managed nonlinear {S}chrödinger equations with lumped
  amplification.
\newblock {\em J. Math. Phys.}, \textbf{62}(no.7):Paper No. 071506, 16 pp.,
  2021.

\bibitem{ADM}
M.-R. Choi and Y.-R. Lee.
\newblock Averaging of dispersion managed nonlinear {S}chrödinger equations.
\newblock {\em Nonlinearity}, \textbf{35}(no.4):2121--2133, 2022.

\bibitem{mc3DNLS}
B.~Dodson.
\newblock Global well-posedness and scattering for the defocusing,
  ${L}^2$-critical, nonlinear {S}chr\"{o}dinger equation when $d \ge3$.
\newblock {\em J. Amer. Math. Soc.}, \textbf{25}(no.2):429--463, 2012.

\bibitem{mc1DNLS}
B.~Dodson.
\newblock Global well-posedness and scattering for the defocusing,
  ${L}^2$-critical, nonlinear {S}chr\"{o}dinger equation when $d = 1$.
\newblock {\em Amer. J. Math.}, \textbf{138}(no.2):531--569, 2016.

\bibitem{mc2DNLS}
B.~Dodson.
\newblock Global well-posedness and scattering for the defocusing,
  ${L}^2$-critical, nonlinear {S}chr\"{o}dinger equation when $d = 2$.
\newblock {\em Duke Math. J.}, \textbf{165}(no.16):3435--3516, 2016.

\bibitem{DHS}
T.~Duyckaerts, J.~Holmer, and S.~Roudenko.
\newblock Scattering for the non-radial {3D} cubic nonlinear {S}chrödinger
  equation.
\newblock {\em Math. Res. Lett.}, \textbf{15}(no.6):1233--1250, 2008.

\bibitem{RH}
J.~Fennell.
\newblock Resonant {H}amiltonian systems associated to the one-dimensional
  nonlinear {S}chrödinger equation with harmonic trapping.
\newblock {\em Commun. Partial Differ. Equ.}, \textbf{44}(no.12):1299--1344,
  2019.

\bibitem{MSti}
D.~Foschi.
\newblock Maximizers for the {S}trichartz inequality.
\newblock {\em J. Eur. Math. Soc.}, \textbf{9}(no.4):739--774, 2007.

\bibitem{Fr}
R.~L. Frank, F.~Méhats, and C.~Sparber.
\newblock Averaging of nonlinear {S}chr\"{o}dinger equations with strong
  magnetic confinement.
\newblock {\em Commun. Math. Sci.}, \textbf{15}(no.7):1933--1945, 2017.

\bibitem{RSt1D}
R.~Frier and S.~Shao.
\newblock A remark on the {S}trichartz inequality in one dimension.
\newblock {\em Dyn. Partial Differ. Equ.}, \textbf{19}(no.2):163--175, 2022.

\bibitem{GabTur1}
I.~R. Gabitov and S.~K. Turitsyn.
\newblock Averaged pulse dynamics in a cascaded transmission system with
  passive dispersion compensation.
\newblock {\em Opt. Lett.}, \textbf{21}:327--329, 1996.

\bibitem{GabTur2}
I.~R. Gabitov and S.~K. Turitsyn.
\newblock Breathing solitons in optical ﬁber links.
\newblock {\em JETP Lett.}, \textbf{63}:861--866, 1996.

\bibitem{CR}
P.~Germain, Z.~Hani, and L.~Thomann.
\newblock On the continuous resonant equation for {NLS}: {I. D}eterministic
  analysis.
\newblock {\em J. Math. Pures Appl.(9)}, \textbf{105}(no.1):131--163, 2016.

\bibitem{ScNLSRT2}
Z.~Hani and B.~Pausader.
\newblock On scattering for the quintic defocusing nonlinear {S}chrödinger
  equation on $ \mathbb{R}\times \mathbb{T}^2$.
\newblock {\em Commun. Pure Appl. Math.}, \textbf{67}(no.9):1466--1542, 2014.

\bibitem{MST}
Z.~Hani, B.~Pausader, N.~Tzvetkov, and N.~Visciglia.
\newblock Modified scattering for the cubic {S}chrödinger equation on product
  spaces and applications.
\newblock {\em Forum Math. Pi}, \textbf{3}(e4):63 pp, 2015.

\bibitem{MS}
Z.~Hani and L.~Thomann.
\newblock Asymptotic behavior of the nonlinear {S}chrödinger equation with
  harmonic trapping.
\newblock {\em Commun. Pure Appl. Math.}, \textbf{69}(no.9):1727--1776, 2016.

\bibitem{StppNLS4d}
S.~Herr, D.~Tataru, and N.~Tzvetkov.
\newblock Strichartz estimates for partially periodic solutions to
  {S}chrödinger equations in $4d$ and applications.
\newblock {\em J. Reine Angew. Math.}, \textbf{690}:65--78, 2014.

\bibitem{HR}
J.~Holmer and S.~Roudenko.
\newblock A sharp condition for scattering of the radial {3D} cubic nonlinear
  {S}chrödinger equation.
\newblock {\em Commun. Math. Phys.}, \textbf{282}(no.2):435--467, 2008.

\bibitem{SSld}
D.~Hundertmark and V.~Zharnitsky.
\newblock On sharp {S}trichartz inequalities in low dimensions.
\newblock {\em Int. Math. Res. Not.}, (Article ID 34080):18 pp., 2006.

\bibitem{IMN}
S.~Ibrahim, N.~Masmoudi, and K.~Nakanishi.
\newblock Scattering threshold for the focusing nonlinear {K}lein-{G}ordon
  equation.
\newblock {\em Anal. PDE}, \textbf{4}(no.3):405--460, 2011.

\bibitem{GwpNLSRT3}
A.~D. Ionescu and B.~Pausader.
\newblock Global well-posedness of the energy-critical defocusing {NLS} on
  $\mathbb{R}\times \mathbb{T}^3$.
\newblock {\em Commun. Math. Phys.}, \textbf{312}(no.3):781--831, 2012.

\bibitem{ECQHO}
C.~Jao.
\newblock The energy-critical quantum harmonic oscillator.
\newblock {\em Commun. Partial Differ. Equ.}, \textbf{41}(no.1):79--133, 2016.

\bibitem{K2}
J.~Kawakami.
\newblock Global approximation for the cubic {NLS} with strong magnetic
  confinement.
\newblock {\em preprint, arXiv.2306.02811}.

\bibitem{K}
J.~Kawakami.
\newblock Averaging of strong magnetic nonlinear {S}chr\"{o}dinger equations in
  the energy space.
\newblock {\em J. Differ. Equ.}, \textbf{377}:622--662, 2023.

\bibitem{KM}
C.~E. Kenig and F.~Merle.
\newblock Global well-posedness, scattering and blow-up for the
  energy-critical, focusing, non-linear {S}chrödinger equation in the radial
  case.
\newblock {\em Invent. Math.}, \textbf{166}(no.3):645--675, 2006.

\bibitem{ECNLSQP}
R.~Killip, M.~Visan, and X.~Zhang.
\newblock Energy-critical {NLS} with quadratic potentials.
\newblock {\em Commun. Partial Differ. Equ.}, \textbf{34}(10-12):1531--1565,
  2009.

\bibitem{NLScri}
R.~Killip and M.~Vişan.
\newblock Nonlinear {S}chrödinger equations at critical regularity, in
  {E}volution {E}quations.
\newblock {\em Clay. Math. Proc.}, \textbf{17}(Amer. Math. Soc., Providence,
  RI):325--437, 2013.

\bibitem{EMSt}
M.~Kunze.
\newblock On the existence of a maximizer for the {S}trichartz inequality.
\newblock {\em Commun. Math. Phys}, \textbf{243}(no.1):137--162, 2003.

\bibitem{LTDfNLSRdT}
Y.~Luo.
\newblock On long time behavior of the focusing energy-critical {NLS} on
  $\mathbb{R}^d\times \mathbb{T}$ via semivirial-vanishing geometry.
\newblock {\em J. Math. Pures Appl.}, \textbf{177}:415--454, 2023.

\bibitem{M}
F.~Méhats and C.~Sparber.
\newblock Dimension redution for rotating {B}ose-{E}instein condensates with
  ainosotropic confinement.
\newblock {\em Discrete Contin. Dyn. Syst.}, \textbf{36}(no.9):5097--5118,
  2016.

\bibitem{Ohta}
M.~Ohta.
\newblock Strong instability of standing waves for nonlinear {S}chrödinger
  equations with a partial confinement.
\newblock {\em Commun. Pure Appl. Anal.}, \textbf{17}(no.4):1671--1680, 2008.

\bibitem{StH}
A.~Poiret.
\newblock Solutions globales pour l'\'{e}quation de {S}chr\"{o}dinger cubique
  en dimension 3.
\newblock {\em Preprint}, arXiv:1207. 1578 [math.AP], 2012.

\bibitem{Random}
A.~Poiret, D.~Robert, and L.~Thomann.
\newblock Random-weighted {S}obolev inequalities on $\mathbb{R}^d$ and
  application to {H}ermite functions.
\newblock {\em Ann. Henri Poincaré}, \textbf{16}(no.2):651--689, 2015.

\bibitem{Iwa}
E.~Richman and C.~Sparber.
\newblock Strong magnetic field limit in a nonlinear {I}watsuka-type model.
\newblock {\em J. Differ. Equ.}, \textbf{302}:334--366, 2021.

\bibitem{RSC}
A.~Selvitella.
\newblock Remarks on the sharp constant for the {S}chrödinger {S}trichartz
  estimate and applications.
\newblock {\em Electron. J. Differ. Equ.}, \textbf{2015}(no.270):19 pp., 2015.

\bibitem{MSSt}
S.~Shao.
\newblock Maximizers for the {S}trichartz and the {S}obolev-{S}trichartz
  inequalities for the {S}chrödinger equation.
\newblock {\em Electron. J. Differ. Equ.}, \textbf{2009}(no.3):13 pp., 2009.

\bibitem{NLSRT}
H.~Takaoka and N.~Tzvetkov.
\newblock On 2{D} nonlinear {S}chrödinger equations with data on
  $\mathbb{R}\times\mathbb{T}$.
\newblock {\em J. Funct. Anal.}, \textbf{182}(no.2):427--442, 2001.

\bibitem{pseudo}
T.~Tao.
\newblock A pseudoconformal compactification of the nonlinear {S}chrödinger
  equation and applications.
\newblock {\em New York J. Math.}, \textbf{15}:265--282, 2009.

\bibitem{WpscNLSMRdT}
N.~Tzvetkov and N.~Visciglia.
\newblock Well-posedness and scattering for nonlinear schrödinger equations on
  $\mathbb{R}^d\times\mathbb{T}$ in the energy space.
\newblock {\em Rev. Mat. Iberoam.}, \textbf{32}(no.4):1163--1188, 2016.

\bibitem{Yajima}
K.~Yajima and G.~Zhang.
\newblock Local smoothing property and {S}trichartz inequality for
  {S}chrödinger equations with potentials superquadratic at infinity.
\newblock {\em J. Differ. Equ.}, \textbf{202}(no.1):81--110, 2004.

\bibitem{Scmres}
K.~Yang and L.~Zhao.
\newblock Global well-posedness and scattering for mass-critical, defocusing,
  infinite dimensional vector-valued resonant nonlinear {S}chrödinger system.
\newblock {\em SIAM J. Math. Anal.}, \textbf{50}(no.2):1593--1655, 2018.

\bibitem{Sc2Dcres}
K.~Yang and Z.~Zhao.
\newblock On scattering asymptotics for the {2D} cubic resonant system.
\newblock {\em J. Differ. Equ.}, \textbf{345}:447--484, 2023.

\bibitem{GwpcNLSM}
X.~Yu, H.~Yue, and Z.~Zhao.
\newblock Global well-posedness for the focusing cubic {NLS} on the product
  space $\mathbb{R}\times \mathbb{T}^3$.
\newblock {\em SIAM J. Math. Anal.}, \textbf{53}(no.2):2243--2274, 2021.

\bibitem{ScNLSR2T2}
Z.~Zhao.
\newblock Global well-posedness and scattering for the defocusing cubic
  {S}chrödinger equation on waveguide $\mathbb{R}^2\times \mathbb{T}^2$.
\newblock {\em J. Hyperbolic Differ. Equ.}, \textbf{16}(no.1):73--129, 2019.

\bibitem{ScNLSRmT}
Z.~Zhao.
\newblock On scattering for the defocusing nonlinear {S}chrödinger equation on
  waveguide $\mathbb{R}^m\times \mathbb{T}$ (when $m=2,3$).
\newblock {\em J. Differ. Equ.}, \textbf{275}:598--637, 2021.

\bibitem{LTDNLSM}
Z.~Zhao and J.~Zheng.
\newblock Long time dynamics for defocusing cubic nonlinear {S}chrödinger
  equations on three dimensional product space.
\newblock {\em SIAM J. Math. Anal.}, \textbf{53}(no.3):3644--3660, 2021.

\bibitem{SeDM}
V.~Zharnitsky, E.~Grenier, C.~K. R.~T. Jones, and S.~K. Turitsyn.
\newblock Stabilizing effects of dispersion management.
\newblock {\em Phys. D}, \textbf{152--153}:794--817, 2001.

\end{thebibliography}
}
\end{document}